\documentclass[a4paper,11pt]{article}

\usepackage{cmap}
\usepackage[english]{babel}

\usepackage[usenames,dvipsnames]{xcolor}
\definecolor{myblue}{rgb}{0.21, 0.34, 0.74}
\definecolor{mygrey}{rgb}{0.55, 0.57, 0.67}
\definecolor{myred}{rgb}{0.79, 0.0, 0.09}
\definecolor{mygreen}{rgb}{0.05, 0.5, 0.06}
\definecolor{randblue}{RGB}{54, 88, 191}
\definecolor{randred}{RGB}{191, 38, 51}
\definecolor{matlabyellow}{rgb}{0.9290, 0.6940, 0.1250}
\definecolor{matlaborange}{rgb}{0.8500, 0.3250, 0.0980}
\definecolor{brightyellow}{RGB}{255, 255, 0}
\definecolor{vibrantgreen}{RGB}{0, 220, 0}
\definecolor{skyblue}{RGB}{0, 190, 255}
\definecolor{yaleblue}{rgb}{0.06, 0.3, 0.57}

\usepackage[nottoc,notlot,notlof]{tocbibind}

\usepackage[T1]{fontenc}
\usepackage[utf8x]{inputenc}

\usepackage{lmodern}
\usepackage{libertine}

\usepackage{bbm}
\usepackage{mathrsfs}
\usepackage{amssymb}
\usepackage{amsmath}
\usepackage{multicol}
\usepackage{upgreek}
\usepackage{mathtools}
\usepackage{tikz}
\usetikzlibrary{arrows.meta}

\usepackage{amsthm}
\usepackage{caption}
\usepackage{enumitem}

\usepackage{comment}
\usepackage[font=small,labelfont=bf]{caption}

\usepackage[pdftex, colorlinks=true, bookmarksopen=true, linkcolor=blue, citecolor=blue]{hyperref}
\usepackage[noabbrev,capitalize,nameinlink]{cleveref}

\usepackage{stmaryrd}

\theoremstyle{plain}
\newtheorem{theorem}{Theorem}[section]

\newtheorem{proposition}[theorem]{Proposition}
\newtheorem{lemma}[theorem]{Lemma}

\newtheorem{claim}{Claim}
\newtheorem{remark}[theorem]{Remark}

\newtheorem{example}[theorem]{Example}

\usepackage{authblk}

\usepackage{adjustbox}

\usepackage[svgnames,pdf]{pstricks}

\usepackage{wrapfig}
\usepackage{upgreek}
\usepackage{bm}
\usepackage[scr=boondoxo]{mathalfa}

\newcommand{\supp}{\mathrm{supp}}

\newcommand{\R}{\mathbb{R}}
\newcommand{\proj}{\bm{\mathsf{P}}^\perp}
\newcommand{\diff}{\, \mathrm{d}}

\newcommand{\conv}{\mathrm{conv}}

\newcommand{\dive}{\mathrm{div}}
\renewcommand{\S}{\mathbb{S}^{d-1}}
\newcommand{\Tan}{\mathsf{T}}
\newcommand{\cP}{\mathscr{P}}
\newcommand{\bW}{\bm{\mathit{W}}}
\newcommand{\bV}{\bm{\mathit{V}}}
\newcommand{\bB}{\bm{\mathit{B}}}

\newcommand{\bU}{\bm{\mathit{U}}}
\newcommand{\Q}{\mathbb{Q}}

\renewcommand{\leq}{\leqslant}
\renewcommand{\geq}{\geqslant}
\renewcommand{\le}{\leqslant}
\renewcommand{\ge}{\geqslant}
\renewcommand{\sec}{\textsection}
\renewcommand{\beta}{\upbeta}
\renewcommand{\eta}{\upeta}

\usepackage{algorithm}
\usepackage{algpseudocode}

\numberwithin{equation}{section}

\usepackage{comment}
\usepackage[font=small,labelfont=bf]{caption}

\usepackage{tikz}
\usetikzlibrary{positioning, arrows.meta, shapes.geometric, calc, fit, backgrounds}
\usepackage{geometry}

\definecolor{color1}{RGB}{255,230,230} 
\definecolor{color2}{RGB}{230,255,230} 
\definecolor{color3}{RGB}{230,230,255} 

\definecolor{nodecolor}{RGB}{180,220,240} 

\definecolor{mycolor1}{RGB}{255,200,200} 
\definecolor{mycolor2}{RGB}{200,255,200} 
\definecolor{mycolor3}{RGB}{200,200,255} 

\title{ Measure-to-measure interpolation using Transformers }

\author{Borjan Geshkovski}
\affil{Inria \& Sorbonne Université}
\author{Philippe Rigollet}
\affil{MIT}
\author{Domènec Ruiz-Balet}
\affil{Universitat de Barcelona}

\date{ \today }

\begin{document}

%
%

\maketitle

%
%


\begin{abstract}

Transformers are deep neural network architectures that underpin the recent successes of large language models. Unlike more classical architectures that can be viewed as point-to-point maps, a Transformer acts as a measure-to-measure map implemented as specific interacting particle system on the unit sphere: the input is the empirical measure of tokens in a prompt and its evolution is governed by the continuity equation. In fact, Transformers are not limited to empirical measures and can in principle process any input measure.  As the nature of data processed by Transformers is expanding rapidly, it is important to investigate their expressive power  as maps from an arbitrary measure to another arbitrary measure. To that end, we provide an explicit choice of parameters that allows a single Transformer to match $N$ arbitrary input measures to $N$ arbitrary target measures, under the minimal assumption that every pair of input-target measures can be matched by some transport map.
			
\bigskip


\medskip

\noindent \textbf{\textsc{ams} classification.}\quad \textsc{41A25, 68T07, 37C10}.

\end{abstract}
	
\thispagestyle{empty}

\setcounter{tocdepth}{2}

\tableofcontents

\setlist[itemize,enumerate]{left=0pt}

\section{Introduction}

\emph{Transformers}, introduced in 2017 with the groundbreaking paper \cite{vaswani2017attention}, are the neural network architectures behind the recent successes of large language models. They owe their impressive results to the way they process data: inputs are length-$n$ sequences of $d$-dimensional vectors called \emph{tokens} (representing words, or patches of an image, for example), which are processed over several layers of parametrized nonlinearities. Unlike conventional neural networks however, all tokens are coupled and mixed at every layer via the so-called \emph{self-attention mechanism}. 

To make this discussion transparent we take a leaf out of several recent works \cite{sander2022sinkformers, vuckovic2020mathematical, geshkovski2023mathematical} 
which view Transformers as a flow maps on $\cP(\S)$---the space of probability measures over the unit sphere $\S$---realized by an interacting particle system: viewing each token as a particle, given an initial sequence of particles $(x_1(0),\ldots,x_n(0))\in(\S)^n$, one considers 
\begin{equation} \label{eq: ips}
    \dot{x}_i(t) = \mathsf{v}[\mu(t)](t,x_i(t)) \hspace{1cm} \text{ for } t\in[0,T],
\end{equation}
for all $i\in\llbracket1,n\rrbracket$; here $\mu(t)=\frac{1}{n}\sum_{j=1}^n \updelta_{x_j(t)}$ denotes the empirical measure. 
The vector field
\begin{equation} \label{eq: vf}
\mathsf{v}[\mu](t,x) = \proj_{x}(\bV(t)\mathscr{A}_{\bB}[\mu](t,x)+\bW(t)(\bU(t)x+b(t))_+)
\end{equation}
depends on the empirical measure through self-attention
\begin{equation*} 
\mathscr{A}_{\bB}[\mu](t,x)
\coloneqq \left.
\int e^{\langle \bB(t)x, x'\rangle}\,x'\mu(\diff x')
\;\middle/\;
\int e^{\langle \bB(t)x, \zeta\rangle}\mu(\diff \zeta)
\right.
\end{equation*}
The parameters $\bV(t),\bB(t),\bW(t),\bU(t)$, which are all $d\times d$ matrices, and $b(t)$, which is a $d$-dimensional vector, are to be used to steer the flow to one's liking. The vector field $\mathsf{v}[\mu(t)](t,\cdot)$ is a combination of the self-attention mechanism $\mathscr{A}_{\bB}[\mu(t)](t,\cdot)$ and a \emph{perceptron}
 at every layer $t$, ultimately projected onto $\Tan_x\S$ by virtue of the orthogonal projector $\proj_x\coloneqq I_d-xx^\top$, referred to as \emph{layer normalization}.
Practical implementations of Transformers are  discrete-time versions, of course, and \eqref{eq: ips} originates from a Lie-Trotter splitting scheme---see \cite{lu2019understanding, geshkovski2023mathematical} for details.

Since \eqref{eq: ips} only truly depends on the empirical measure, one can naturally turn to the \emph{continuity equation} which governs its evolution. We can thus equivalently see the Transformer as the solution map of the Cauchy problem
\begin{equation} \label{eq: cauchy.pb}
\begin{dcases}
\partial_t \mu(t) + \dive\Big(\mu(t)\, \mathsf{v}[\mu(t)]\Big)=0 &\text{ on } [0,T]\times\S\\
\mu(0) = \mu_0 &\text{ on } \S.
\end{dcases}
\end{equation}
Here $-\mathrm{div}$ denotes the adjoint of the spherical gradient $\nabla$.
As the number $n$ of particles can be large---orders of magnitude vary in different implementations, likely due to compute---in this paper we focus on \eqref{eq: cauchy.pb}, which makes sense for arbitrary measures, and encompasses \eqref{eq: ips} in the particular setting of empirical measures.

{\color{black} Transformers (more specifically, encoders) are used to solve learning tasks such as \emph{masked-language modeling}. One starts from an input sequence of $n$ tokens, masks a subset of positions, and trains the model to predict the original tokens at the masked positions. In practice, the model produces predictions at \emph{all} $n$ positions (i.e. a probability distribution over the vocabulary for each position), but the training loss is computed only on the masked positions. Consequently, for the purpose of our abstraction, we idealize the problem and identify the relevant “output” with the probability distributions associated with the masked tokens (equivalently, with a target measure supported on the masked positions).} Motivated by further ubiquitous tasks including sentiment analysis and image classification, and taking an approximation/control theory perspective, in this paper we consider the canonical learning problem in which we are given data consisting of $N\gg1$ pairs of input and output probability distributions
\begin{equation} \label{eq: data} \tag{$\mathscr{D}$}
(\mu_0^i, \mu_1^i)\in\mathscr{P}(\S)\times\mathscr{P}(\S) \hspace{1cm }\text{ for } i\in\llbracket1,N\rrbracket,
\end{equation}
and we seek to match them through the solution map of \eqref{eq: cauchy.pb}. In the context of the applications evoked above, one always works with discrete measures,  
but we consider a more general setting in what follows.
This is an \emph{ensemble transportation} or \emph{controllability} problem, since we seek to accomplish this matching of measures by means of the flow of \eqref{eq: cauchy.pb} for a \emph{single} parameter or control $\theta=(\bV(t),\bB(t),\bW(t),\bU(t), b(t))_{t\in[0,T]}$. 

In the discrete-time setting, and focusing solely on mapping sequences to sequences, the problem is first solved in \cite{yun2019transformers} by using $\bB=\beta\widetilde{\bB}$ and $\beta=+\infty$ (a formal limit), as well as additional bias vectors within the inner products of the self-attention mechanism, but without employing layer normalization. Further work has focused on seeing whether one can do matching solely using self-attention, namely, without the perceptron component or layer normalization---results in this direction include \cite{alberti2023sumformer, kratsios2021universal}. See \cite{chiang2023tighter, jiang2023approximation, edelman2022inductive, jiang2023brief, wang2024understanding, petrov2024prompting, sander2024towards} for further results. 

In the continuous time and/or arbitrary measure setting, much less is known---we are aware of \cite{adu2024approximate, agrachev2024generic, furuya2024transformers}. In \cite{agrachev2024generic}, still in the context of empirical measures, the authors focus on self-attention dynamics only (\(\bW \equiv 0\)) and prove that, \emph{generically}, two vector fields in the class of permutation-equivariant vector fields suffice to match two ensembles of empirical measures with the same number of atoms.
Their study is inspired by a flurry of works on matching one cloud of points to another using the flow of \eqref{eq: vf} with \(\bV \equiv 0\) (known as \emph{neural ODEs}), where tools from geometric control theory can be useful \cite{agrachev2020control, agrachev2022control, scagliotti2021deep, elamvazhuthi2022neural, tabuada2022universal}.
With the exception of \cite{tabuada2022universal, elamvazhuthi2022neural}, none of these papers actually state the specific vector fields that can be used, and none of them are constructive.
On another hand, \cite{adu2024approximate} address the setting of absolutely continuous measures, but use a slightly different vector field compared to \eqref{eq: vf}. 
Finally, \cite{furuya2024transformers} address the discrete-time system and arbitrary measures, but use a slightly different model motivated by \emph{in-context learning} \cite{garg2022can} and approximate a map \( \mathscr{P}(\Omega) \times \Omega \to \Omega \) over compact subsets $\Omega\subset\mathbb{R}^d$---the proof is based on a clever application of Stone-Weierstrass.

None of the above papers use layer normalization; moreover, the parameters used are not explicit due to the non-constructive strategy, and there are therefore no bounds on the number of switches. 
To address these pitfalls, we take inspiration from concurrent works on neural ODEs \cite{li2022deep, ruiz2023neural, cheng2023interpolation, geshkovski2026constructiveconditionalnormalizingflows} in which the parameters are fully explicit and piecewise constant by construction. Our goal is to focus on the most general case while constructing parameters that leverage  salient properties of all mechanisms involved in \eqref{eq: vf}---the prime example being the dynamic \emph{emergence of clusters} proven in \cite{geshkovski2024emergence, geshkovski2023mathematical} (see \cite{criscitiello2024synchronization, geshkovski2024dynamic, bruno2024emergence, geshkovski2024number, alcalde2025attention, polyanskiy2025synchronization, burger2025analysis, alcalde2025clustering, castin2025unified, abellaconsensus, alvarez2026perceptrons} for subsequent refinements), which has been empirically observed and referred to as \emph{token uniformity}, \emph{oversmoothing}~\cite{chen2022principle, ru2023token, guo2023contranorm, wu2024demystifying, wu2024role, dovonon2024setting, scholkemper2024residual}, or \emph{rank collapse}~\cite{dong2021attention, feng2022rank, noci2022signal, joudaki2023impact, zhao2023are, zhai2023stabilizing, noci2024shaped, bao2024self, cowsik2024geometric} in the literature. 
In fact, we solely use the long-time behavior of \eqref{eq: cauchy.pb} with explicit, well-chosen parameters throughout, and as such, our strategy also leads to a deeper understanding of the inner workings of all mechanisms in \eqref{eq: vf}.

\subsection{Main results} 

Set $\Uptheta\coloneqq(\mathscr{M}_{d\times d}(\R))^4\times \R^d.$ For any $T>0$ and $\theta=(\bV,\bB,\bW,\bU, b)\in L^\infty((0,T);\Uptheta)$, the Cauchy problem \eqref{eq: cauchy.pb} is well-posed, in the sense that for every $\mu_0\in\mathscr{P}(\S)$ there exists a unique weak solution $\mu\in \mathscr{C}^0([0,T];\mathscr{P}(\S))$. This in turn yields a continuous and invertible flow (or solution) map 
$$\Phi^t_\theta:\mathscr{P}(\S)\to\mathscr{P}(\S),$$
for $t\in[0,T]$, with 
$$\Phi^t_\theta(\mu_0) = \mu(t),$$ 
which we often use later on to simplify the presentation.
These results follow from  classical arguments using the Lipschitz properties of the underlying vector field---see \cite[\sec 6]{geshkovski2024emergence}, \cite{paul2022microscopic} for details.

Henceforth, for simplicity, assume\footnote{The assumption $\mu_1^i\not\equiv\mu_1^j$ for $i\neq j$ (and \eqref{eq: first.assp}, and more generally \eqref{eq: assumption.hole}) can be removed at the cost of additional technicalities---see {\bf\Cref{appendix: technical}}. 
} that $\mu_0^i\not\equiv\mu_0^j$ and $\mu_1^i\not\equiv\mu_1^j$ for $i\neq j$.

{\color{black}
To elucidate the working mechanisms of Transformers and obtain a simpler proof, we first focus on the idealized case where the target is supported on a single atom. A general result appears right after.
}

\begin{theorem} \label{thm: targets.atoms}
    Suppose $d\geq3$.
    Consider data 
    \eqref{eq: data} such that 
    \begin{enumerate}
    \item There exists $w_0\in\S$ such that
    \begin{equation} \label{eq: first.assp}
        w_0\notin\bigcup_{i} \supp(\mu_0^i). 
    \end{equation} 
    \item For any $i\in\llbracket1,N\rrbracket$, we have $\mu^i_1=\updelta_{x^i}$.
    \end{enumerate}
    Then for any $T>0$ and $\varepsilon>0$, there exists $\theta\in L^\infty((0,T);\Uptheta)$ such that for any $i\in\llbracket1,N\rrbracket$, the unique solution $\mu^i\in\mathscr{C}^0([0,T];\mathscr{P}(\S))$ to \eqref{eq: cauchy.pb} with data $\mu_0^i$ and parameters $\theta$ satisfies
    \begin{equation*}
        \mathsf{W}_2(\mu^i(T),\mu_1^i)\leq\varepsilon.
    \end{equation*}
    Moreover, $\theta$ can be chosen piecewise constant, with  $O(d\cdot N)$ switches, and $$\|\theta\|_{L^\infty((0,T);\Uptheta)}=O\left(\frac{d\cdot N}{T}+\log\frac{1}{\varepsilon}\right).$$
\end{theorem}

The fact that the parameters $\theta$ can be chosen to be piecewise-constant-in-time leads to a direct link with the discrete-time network used in practice: the number of switches provides a lower bound on the number of layers. Our estimates are in all likelihood sub-optimal (principally due to our inability to simultaneously use both components of the vector field in \eqref{eq: vf}, as seen in {\bf \Cref{sec: sketch}}) and we believe that there is great margin for improvement. The reader is referred to {\bf \Cref{sec: nb.disc}} and {\bf \Cref{sec: complexity.untanglement}} for further comments on this particular aspect.

\Cref{thm: targets.atoms} follows as a corollary of the proof of the following general result.

\begin{theorem} \label{thm: main.result}
Suppose $d\geq3$. Consider data \eqref{eq: data} such that
\begin{enumerate}
    \item There exist $w_0,w_1\in\S$ such that 
    \begin{equation} \label{eq: assumption.hole}
        w_0\notin\bigcup_{i} \supp(\mu_0^i) \quad \text{ and } \quad w_1\notin\bigcup_{i}\supp(\mu_1^i). 
    \end{equation}  
    \item For any $i\in\llbracket1,N\rrbracket$, there exists $\mathsf{T}^i\in L^2(\S;\S)$ such that $\mathsf{T}^i_{\#}\mu_0^i=\mu_1^i$.
\end{enumerate}
Then for any $T>0$ and $\varepsilon>0$, there exists $\theta\in L^\infty((0,T);\Uptheta)$ such that for any $i\in\llbracket1,N\rrbracket$, the unique solution $\mu^i\in\mathscr{C}^0([0,T];\mathscr{P}(\S))$ to \eqref{eq: cauchy.pb} with data $\mu_0^i$ and parameters $\theta$ satisfies
    \begin{equation*}
        \mathsf{W}_2(\mu^i(T),\mu_1^i)\leq\varepsilon.
    \end{equation*}
    Moreover, $\theta$ can be chosen piecewise constant.
\end{theorem}

Here $\mathsf{T}_{\#}\mu(A) = \mu(\mathsf{T}^{-1}(A))$ for $A\subset\S$ is the image measure. The special case in which the target is supported on $n$ atoms, and the input is either of this form or has a density, fits the above framework but admits a simpler proof, given in \Cref{proof: proof.restricted} (restricted case). The number of switches of  $\theta$ can be estimated by using structural properties of the measures---we postpone a discussion thereon to {\bf \Cref{sec: nb.disc}} and {\bf \Cref{sec: complexity.untanglement}}. 

\subsection{Overview of the proof} \label{sec: sketch}

We sketch the proof of \Cref{thm: main.result}.
The solution map $\Phi_{\text{fin}}^T:\mathscr{P}(\S)\to\mathscr{P}(\S)$ is constructed as\footnote{The philosophy is reminiscent to the proof of the Chow-Rashevskii theorem using iterated Lie brackets for the controllability of driftless systems \cite[\sec 3.3]{coron2007control}.}
$$
\Phi_{\text{fin}}^T\coloneqq(\Phi_{\theta_3}^{\frac{T}{3}})^{-1}\circ\Phi_{\theta_2}^{\frac{T}{3}}\circ\Phi_{\theta_1}^{\frac{T}{3}},
$$ 
where

\begin{enumerate}
\item $\Phi_{\theta_1}^t:\mathscr{P}(\S)\to\mathscr{P}(\S)$ is the solution map of \eqref{eq: cauchy.pb} on $[0,T/3]$, generated by piecewise constant parameters $\theta_1$, having $O(d\cdot N)$ switches, as to disentangle the supports of the input measures (the use of the attention component is \emph{necessary} for this step). After this step, the supports of the measures are disjoint:
\begin{equation} \label{eq: disentanglement.intro}
\supp\,\Phi_{\theta_1}^{\frac{T}{3}}(\mu_0^i)\cap\supp\,\Phi_{\theta_1}^{\frac{T}{3}}(\mu_0^j)=\varnothing\quad\text{ whenever }i\neq j.
\end{equation}
This is done in \Cref{prop: separation} in {\bf \Cref{sec: disentanglement}}.
The clue lies in following the insights of \cite{geshkovski2023mathematical}, which entail clustering of every individual measure to a single point mass in long time in the special regime $\bB=\beta I_d$ with $\beta\geq 0$ and $\bV=I_d$. Should the limit point masses corresponding to every input measure be located at different positions, the disentanglement property \eqref{eq: disentanglement.intro} would readily follow by taking the time horizon $T$ large enough. Unfortunately, characterizing the location of the limit point mass for general measures is an open problem. 
We instead consider a curated choice of $\bV$ to facilitate locating the limiting cluster for every measure, which we now sketch. Consider $N=2$ (the general case is argued by induction; see \Cref{lem: induction.barycenter}) and suppose that $\mathbb{E}_{\mu_0^1}[z]$ and $\mathbb{E}_{\mu_0^2}[z]$ are not colinear (this assumption is not needed, as seen in \Cref{lem: colinearity}). 
We can take $\bB\equiv0$ 
and 
\begin{equation*}
    \bV(t) \coloneqq \sum_{k=1}^{d-1} \alpha_k\alpha_k^\top 1_{[T_k,T_{k+1}]}(t),
\end{equation*}
where $\{\alpha_k\}$ is an orthonormal basis of $(\text{span}\,\mathbb{E}_{\mu_0^1}[z])^\perp$. Then there is some index $\ell$ such that $\langle \mathbb{E}_{\mu_0^1}[z], \alpha_\ell\rangle=0$ and $\langle \mathbb{E}_{\mu_0^2}[z], \alpha_\ell\rangle\neq 0$. Consequently $t\mapsto\langle \mathbb{E}_{\mu^i(t)}[z], \alpha_\ell\rangle$ remains constant when $i=1$, and does not change sign when $i=2$. After an elementary computation one can then see that any $x(t)\in\supp(\mu^2(t))$ converges to $\pm\alpha_\ell$ in long time, whereas $\mu^1(t)=\mu_0^1$ throughout. One can always rescale time so that the above holds at an arbitrary prescribed horizon, at the cost of increasing the norm of the parameters.

\item In the same vein, $\Phi_{\theta_3}^t:\mathscr{P}(\S)\to\mathscr{P}(\S)$ is the solution map of \eqref{eq: cauchy.pb} on $[2T/3, T]$, 
generated by piecewise constant parameters $\theta_3,$  
 as to disentangle the supports of the target measures:
\begin{equation*}
\supp\,\Phi_{\theta_3}^{\frac{T}{3}}(\mu_1^i)\cap\supp\,\Phi_{\theta_3}^{\frac{T}{3}}(\mu_1^j)=\varnothing\quad\text{ whenever }i\neq j.
\end{equation*}
Inverting $\Phi_{\theta_3}^t$ simply corresponds to running time backwards from $T$ to $2T/3$.
\item $\Phi_{\theta_2}^t:\mathscr{P}(\S)\to\mathscr{P}(\S)$ is the solution map of \eqref{eq: cauchy.pb} on $[T/{3},2T/{3}]$, generated by piecewise constant parameters $\theta_2$, alternating between $\bV\equiv0$ (namely, using solely the perceptron component) and $\bW\equiv0, \bV\equiv I_d$, which approximately matches the ensembles of disentangled input and target measures: 
\begin{equation*}
    \mathsf{W}_2\left((\Phi_{\theta_2}^{\frac{2T}{3}}\circ\Phi_{\theta_1}^{\frac{T}{3}})(\mu_0^i),\Phi_{\theta_3}^{\frac{T}{3}}(\mu_1^i)\right)\leq\varepsilon
\end{equation*}
for all $i$. This map can be constructed in three different ways depending on the nature of the target measures. If the target measures are point masses (\Cref{thm: targets.atoms}), one simply clusters the disentangled input measures to point masses using \Cref{prop: targets.atoms} in {\bf \Cref{sec: clustering}} ($\bW\equiv0, \bV\equiv I_d$) up to time $T/2$ say, and then matches the resulting point masses to the targets using \Cref{prop: interpolation.neural.ode} in {\bf \Cref{sec: neural.ode}} ($\bV\equiv0$) at time $2T/3$. 
This idea is then generalized to targets that are empirical measures with $M\geq 2$ atoms in {\bf \Cref{proof: proof.restricted}} (see the restricted case). 
The case of general, non-atomic target measures is significantly more involved. The construction is done in \Cref{lem: univ.approx} in {\bf \Cref{sec: proofs}} and the main idea is as follows. It can readily be seen (see \Cref{lem: hyp.propagation}) that the transport maps $\mathsf{T}^i$ are propagated by the flow maps constructed in the two previous steps, in the sense that there exists some integrable map $\uppsi:\S\to\S$ with $\uppsi|_{\supp (\Phi_{\theta_1}^{\frac{T}{3}}(\mu_0^i))}=\uppsi^i$ and 
\begin{equation*}
    \uppsi^i_\#\Phi_{\theta_1}^{\frac{T}{3}}(\mu_0^i) = \Phi^{\frac{T}{3}}_{\theta_3}(\mu_1^i).
\end{equation*} 
Since we construct $\Phi^t_{\theta_2}$ without using the nonlinear part of \eqref{eq: cauchy.pb}, we can identify $\Phi_{\theta_2}^t$ with a Lipschitz-continuous and invertible map from $\S$ to $\S$, which we also denote $\Phi^t_{\theta_2}$. Using standard arguments from optimal transport (\Cref{lem: monge}), we find
\begin{align*}
    \mathsf{W}_2\left((\Phi^{ \frac{2T}{3} }_{\theta_2} )_\# \Phi^{\frac{T}{3}}_{\theta_1}(\mu_0^i), \Phi_{\theta_3}^{\frac{T}{3}}(\mu_1^i)\right)&\lesssim \left\|\Phi_{\theta_2}^{\frac{2T}{3}}-\uppsi\right\|_{L^2(\mu)},
\end{align*}
where $\mu=\sum_{i=1}^N \Phi_{\theta_1}^{\frac{T}{3}}(\mu_0^i).$
The final result therefore boils down to approximating maps in $L^2(\S,\mu)$. This is technically involved due to the fact that $\mu$ can have both diffuse and atomic parts---both elements are treated using the clustering and matching constructions presented in {\bf \Cref{sec: clustering}} and {\bf \Cref{sec: neural.ode}} respectively.
\end{enumerate}

\begin{figure}[!ht]
    \centering
    \begin{tikzpicture}[node distance=2cm, scale=0.8, every node/.style={transform shape}]

    \tikzset{
        section/.style={rectangle, minimum width=3cm, minimum height=1cm, text centered, draw=black, fill=#1},
        arrow/.style={thick, ->, >=Stealth}
    }

    \node (1) [section=vibrantgreen, align=center] {
    {\hypersetup{linkcolor=white}\Cref{lem: hyp.propagation}{\color{white}:}}\\{\color{white}$\mathsf{T}^i$ are preserved}};
    \node (2) [section=skyblue, below of=1, align=center] {\hypersetup{linkcolor=white}\Cref{lem: monge}{\color{white}:}
    \\ {\color{white}$\varepsilon$-Matching $\Longleftrightarrow$}\\ {\color{white}Approximating maps}\\{\color{white}in $L^2(\S,\mu)$}};
    \node (3) [section=brightyellow, align=center, below of=2, yshift=-2cm] {\hypersetup{linkcolor=black}\Cref{prop: separation}:\\Disentanglement};
    \node (4) [section=skyblue, right of=2, xshift=4cm, align=center] {\hypersetup{linkcolor=white}\Cref{lem: univ.approx}\color{white}:\\ \color{white}Approximating maps in $L^2(\S,\mu)$};
    \node (5) [section=brightyellow, right of=4, xshift=3.5cm, align=center] {\hypersetup{linkcolor=black}\Cref{thm: main.result}};
    \node (6) [section=brightyellow, below of=2, xshift=4cm, yshift=-0.5cm, align=center] {\hypersetup{linkcolor=black}\Cref{prop: compression}\\ Clustering diffuse part\\ of $\mu$};
    \node (7) [section=brightyellow, below of=4, xshift=2.5cm, yshift=-0.5cm, align=center] {\hypersetup{linkcolor=black}\Cref{prop: interpolation.neural.ode}\\ Matching discrete part\\ of $\mu$};

    \draw [arrow] (1) to (2);
    \draw [arrow] (2) -- (4);
    \draw [arrow] (4) -- (5);
    \draw [arrow] (3) -- (2);
    \draw [arrow] (6) -- (4);
    \draw [arrow] (7) -- (4);
    \draw [arrow, bend left] (3) to (6); 
    \draw [arrow, bend right] (3) to (7); 

    \end{tikzpicture}
    \caption{High-level overview of the proof of \Cref{thm: main.result}. 
    }
    \label{fig: overview.2}
\end{figure}
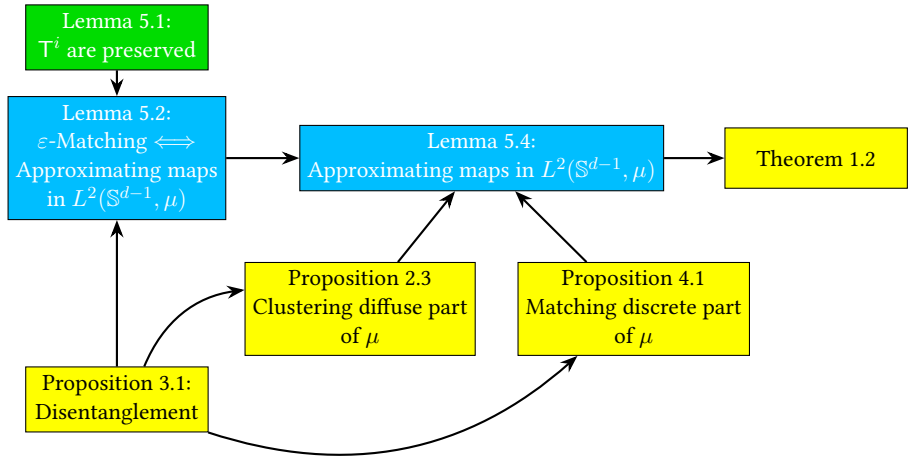

Matching general ensembles of measures \emph{cannot} be done with a single \emph{linear} continuity equation, as is done in the Benamou-Brenier reformulation of optimal transport for instance \cite{benamou2000computational}, namely \eqref{eq: cauchy.pb} in which the vector field $v$ does not depend on $\mu(t)$. Indeed, take for instance $\mu_0^1,\mu_0^2\in\mathscr{P}_{\mathrm{ac}}(\S)$ such that 
$$\supp\,\mu_0^1\cap \supp\,\mu_0^2\neq\varnothing,$$
and similarly $\mu_1^1,\mu_1^2\in \mathscr{P}_{\mathrm{ac}}(\S)$ such that
\begin{equation} \label{eq: empty.intersection}
\supp\,\mu_1^1\cap \supp\,\mu_1^2=\varnothing.
\end{equation}
Then there cannot exist a single-valued $\mathsf{T}:\S\to\S$ such that $\mathsf{T}_\#\mu_0^1=\mu_1^1$ and $\mathsf{T}_\#\mu_0^2=\mu_1^2$, since there would have to exist $x\in \supp\,\mu_0^1\cap \supp\,\mu_0^2$ for which $\mathsf{T}(x)$ would have to take two different values due to \eqref{eq: empty.intersection}. This elementary counterexample is the starting point of our strategy, as the self-attention mechanism $\mathscr{A}_{\bB}[\mu]$ provides a nonlinear dependence\footnote{One can draw parallels with the failure of the Kalman rank condition \cite{sontag2013mathematical, coron2007control} for the ensemble controllability of linear systems in finite dimensions.} of the solution map to \eqref{eq: cauchy.pb} with respect to $\mu$, which we use precisely to disentangle overlapping measures. In this regard, \Cref{thm: main.result} is an ensemble controllability result for a {\it nonlinear} continuity equation, thus extending existing results on the controllability of the {\it linear} continuity equation---see \cite{brockett2008control, khesin2009, agrachev2009controllability, agrachev2009optimal, raginsky2024some, chen2016optimal, duprez2019approximate}.

\subsection{Outline} The remainder of the paper is organized as follows. We comment on assumptions and extensions of \Cref{thm: main.result} in {\bf \Cref{sec: discussion}}. In {\bf \Cref{sec: clustering}}, we provide explicit parameters that yield long-time clustering (i.e., convergence to discrete measures). {\bf \Cref{sec: disentanglement}} presents how initial measures with overlapping support can be disentangled over time using clustering. {\bf \Cref{sec: neural.ode}} addresses the matching problem of clouds of points, which is used after clustering and disentanglement. The proofs of \Cref{thm: main.result} and \Cref{thm: targets.atoms} can be found in {\bf \Cref{sec: proofs}}. We discuss some interesting questions regarding the number of switches needed for disentanglement in {\bf \Cref{sec: complexity.untanglement}}.

\subsection{Discussion} \label{sec: discussion}

\subsubsection{On our assumptions}  

\begin{itemize}
    \item The requirement $d \geq 3$ in \Cref{thm: main.result} stems from matching disentangled measures. In $d = 2$ the order of particles is preserved. To carry through our strategy, one would need to use self-attention to re-order the input measures and disentangle them.

    \item  When the targets are more general than point masses (as in \Cref{thm: main.result}), we assume that a transport map exists for each pair $(\mu_0^i,\mu_1^i)$—e.g., if $\mu_0^i\ll\diff x_{\S}$ (Brenier--McCann \cite{brenier1991polar,mccann2001polar}) or when $\mu_0^i$ and $\mu_1^i$ are empirical with the same number of atoms. This assumption is inessential for our purposes, since our result is approximate and any two non-empirical measures can be matched arbitrarily well.

    \item If $\mu_0^i$ and $\mu_1^i$ are empirical with $n$ and $m$ atoms ($n>m$, $n/m\notin\mathbb{N}$), no transport map exists, so \eqref{eq: cauchy.pb} cannot approximate $\mu_1^i$ arbitrarily well. Grouping $\lfloor n/m\rfloor$ atoms of $\mu_0^i$ per atom of $\mu_1^i$ yields $\mathsf{T}^i$ with $\mathsf{W}_2(\mathsf{T}^i_\#\mu_0^i,\mu_1^i)=O(m/n)$, hence the flow of \eqref{eq: cauchy.pb} approximates all $N$ targets within $O(\varepsilon+m/n)$.
\end{itemize} 

\subsubsection{On exact matching} 
One can inquire if it is possible to have \emph{exact} matching. i.e. $\varepsilon=0$, in \Cref{thm: main.result}. 

\begin{itemize}
    \item We can exactly match $N$ empirical input measures to $N$ empirical target measures as long as they have the same number of atoms. This follows as a corollary of the proof of \Cref{thm: main.result}, since no quantization is required in \Cref{lem: univ.approx}. 
    
    \item Since $\mathsf{v}[\mu(t)](t,\cdot)$ is Lipschitz, we cannot do exact transportation of an absolutely continuous measure to a discrete one even when $N=1$. Similarly, we cannot match a single input measure with connected support to a target measure whose support has multiple connected components. 
\end{itemize}

\begin{remark}[Beyond $\mathsf{W}_2$] 
We use $\mathsf{W}_2$ for convenience, but the argument should adapt to the KL divergence and yields a stronger result by \cite{bolley2005weighted} (the required Gaussian moment for the second KL argument holds on $\S$). After disentangling, we can match in $TV$ and then apply a reverse Pinsker inequality as in \cite{alvarez2025constructive}.

\end{remark}

\subsubsection{On the number of parameter switches} \label{sec: nb.disc}

For piecewise-constant parameters,
\[
\#\mathrm{switches}
=\#\mathrm{switches}_{\text{disent}}
+\#\mathrm{switches}_{\text{cluster}}
+\#\mathrm{switches}_{\text{match}}.
\]
If all supports pairwise overlap, $
\#\mathrm{switches}_{\text{disent}}=O(d\cdot N)$. The overall count in full generality is driven by clustering. We discuss three regimes:

\begin{enumerate}[leftmargin=2em,itemsep=.3ex]
\item (Targets are a Dirac---\Cref{thm: targets.atoms}.)
After disentanglement, a single constant parameter clusters each input to a point \cite{geshkovski2023mathematical}, so $\#\mathrm{switches}_{\text{cluster}}=0$. Matching to targets via the perceptron gives $\#\mathrm{switches}_{\text{match}}=O(N)$ \cite{ruiz2023neural,li2022deep}. 

\item (Targets are $m$-atomic, inputs too or have a density---\Cref{thm: main.result}.)
If the inputs are $m$-atomic, clustering is not required. Otherwise partition each disentangled support into $m$ pieces and cluster each piece with one constant parameter using \Cref{lem: two.balls}: $\#\mathrm{switches}_{\text{cluster}}=O(m\cdot N)$. Match to the $m$-atomic targets by \Cref{prop: interpolation.neural.ode} with $\#\mathrm{switches}_{\text{match}}=O(m\cdot N)$. Hence $\#\mathrm{switches}=O((m+d)N).$

\item (Inputs and targets are empirical with $n$ and $m$ atoms respectively.)
If $n\gg m$ or $m\mid n$, use $m$ balls per measure in \Cref{prop: compression} (via \Cref{prop: separation} and clustering to atoms as in \Cref{prop: targets.atoms}), yielding $\#\mathrm{switches}_{\text{cluster}}=O(m\cdot N)$ and therefore $
\#\mathrm{switches}=O((m+d)N).$
In the discrete-time setting of \cite{yun2019transformers}, the number of layers is independent of $N$ but exponential in $d$.
\end{enumerate}
In the most general case of \Cref{thm: main.result}, $\#\mathrm{switches}_{\text{cluster}}$ can be exponential in $d$ due to packing-number arguments (\Cref{rem: nb.disc.clustering}).

\subsubsection{On generalities}

We comment on greater generality in the choice of the Transformer architecture, which typically varies slightly from implementation to implementation. 

\begin{itemize}
    \item {\bf Increasing the width.} We often use rank-1 constant $\bW,\bU$; using rectangular matrices (greater \emph{width}) could reduce the number of switches ($\approx$ depth).
    
    \item {\bf Multi-head attention} replaces $\bV(t)\mathscr{A}_{\bB}[\mu(t)](t,x)$ by $\sum_{h=1}^H \bV_h(t)\mathscr{A}_{\bB_h}[\mu(t)](t,x).$
    We do not know how to exploit multiple heads $H$ in our proofs---some insights appear in \cite{chen2024provably}.

    \item {\bf Discrete time.} The continuous-time formulation yields a time-reversible equation used in our construction. Our results are expected to hold for suitable discretizations of \eqref{eq: cauchy.pb} with a sufficiently small time step.

    \item {\bf Beyond the ReLU.} All results remain if $(\cdot)_+$ is replaced by any Lipschitz nonlinearity that agrees with ReLU near the origin. The key requirement is that the induced flow (with $\bV\equiv0$) leaves any chosen spherical cap invariant.
\end{itemize}

\subsection{Notation and basic definitions} 

Unless stated otherwise, all integrals are over $\S$. We write $\llbracket 1,n\rrbracket\coloneqq\{1,\ldots,n\}$.  
For $A\subset \S$, $\mathrm{conv}\,A$ is the Euclidean convex hull in $\R^d$, and $\conv_g A$ is the geodesic convex hull in $\S$. Balls $B(x,R)$ centered at $x$ of radius $R>0$ are in $\S$ taken w.r.t. the geodesic distance $d_g$.

\subsubsection*{Acknowledgments} 

We thank the reviewers for their excellent suggestions, which have greatly improved the quality of the paper.
B.G. was supported by a Sorbonne Emergences grant and a gift by Google.
P.R. was supported by NSF grants DMS-2022448, CCF-2106377, and a gift from Apple.

\section{Clustering of the input data} \label{sec: clustering}

We begin by investigating how the input measures can be clustered using \eqref{eq: cauchy.pb}, in the sense that they are in the vicinity of discrete measures with few atoms.

In {\bf \Cref{sec: diracs}}, we cover the special case of clustering to a single atom, while the case of general discrete measures is discussed in {\bf \Cref{sec: general.dirac}}. The results of this section are used in {\bf \Cref{sec: disentanglement}}, and they are also a key step in our final matching strategy.

\subsection{Clustering to a single point mass} \label{sec: diracs}

The following is an adaptation of \cite[Lemma 6.4]{geshkovski2023mathematical}.

\begin{proposition} \label{prop: targets.atoms}
Suppose $\bB\in \mathscr{M}_{d\times d}(\mathbb{R})$ and  $\supp\,\mu_0$ is contained in an open hemisphere.
Then the solution $\mu$ to \eqref{eq: cauchy.pb}--\eqref{eq: vf} with data $\mu_0$ and $(\bV(\cdot),\bB(\cdot), \bW(\cdot))\equiv (I_d,\bB, 0)$ satisfies $\mathrm{diam}(\conv_g\,\supp\,\mu(t))\to0$ as $t\to+\infty$. 

Moreover, for $\varepsilon>0$ there exist $z\in \conv_g\,\supp \,\mu_0$ and $T>0$ such 
\begin{equation*}
\mathsf{W}_{\infty}\left(\mu(T),\updelta_{z}\right)\leq \varepsilon \quad \text{ and } \quad \inf \left\{t\geq 0\colon \mathsf{W}_2(\mu(t),\updelta_{z})\leq \varepsilon\right\}=O\left(\log \frac{1}{\varepsilon}\right)
\end{equation*}
\end{proposition}

\begin{proof}
The characteristic flow is $
\dot x(t)=\proj_{x(t)}\mathscr{A}_{\bB}[\mu(t)](x(t)).$
As $\supp\,\mu_0$ lies in an open hemisphere, the vector $\gamma(x)\coloneqq\mathscr{A}_{\bB}[\mu_0](x)/\|\mathscr{A}_{\bB}[\mu_0](x)\|$
is well defined for every $x\in \conv_g\,\supp\,\mu_0$ and points strictly into $\mathrm{int}\,\conv_g\,\supp\,\mu_0$. A first-order expansion of the flow at any boundary point $x_0\in\partial\conv_g\,\supp\,\mu_0$ shows
\[
\langle x(\tau)-x_0,\gamma(x_0)\rangle>0 \quad\text{for all small }\tau>0,
\]
hence $x(\tau)\in\mathrm{int}\,\conv_g\,\supp\,\mu_0$. Therefore for $0\le t_1<t_2$,
\begin{equation*}
\conv_g\,\supp\,\mu(t_2)\,\subset\mathrm{int}\,\conv_g\,\supp\,\mu(t_1)\subset\conv_g\,\supp\,\mu(t_1).
\end{equation*}
Thus $\phi(t)\coloneqq\mathrm{diam}(\conv_g\,\supp\,\mu(t))$
is decreasing and bounded from below by $0$, so $\phi(t)\to\ell$ as $t\to+\infty$. If $\ell>0$, compactness yields times $t_k\to+\infty$ with boundary points that do not move inward, contradicting the strict interior pointing above. Thus $\ell=0$. Because the geodesic convex hulls are nested in time, and their diameter goes to zero, there exists a unique $z\in\conv_g\,\supp\,\mu(0)$ such that $\mu\rightharpoonup\updelta_{z}$ as $t\to+\infty$.

Pick $T$ so that $\mathrm{diam}(\conv_g\,\supp\,\mu(T))\leq\varepsilon$. With the coupling $\pi_*=(\mathrm{Id},\mathsf{T})_\#\mu(T)$ where $\mathsf{T}(x)=z$ for any $x\in\supp\,\mu(T)$ we get $\mathsf{W}_p(\mu(T),\updelta_z)\leq \varepsilon$; we conclude by letting $p\to+\infty$. The final conclusion follows from \cite[Theorem 2.3]{chen2025quantitative}.
\end{proof}

\subsection{Clustering to discrete measures} \label{sec: general.dirac}

The following result ensures that an ensemble of measures with disjoint supports can be clustered, up to arbitrary precision, to finitely many atoms within their own support, all by means of the same flow map. 

\begin{proposition} \label{prop: compression}
Suppose $\mu_0^i$ has no atoms and $\conv_g\,\supp\,\mu^i_0\cap\conv_g\,\supp\,\mu_0^j=\varnothing$ for $i\neq j$. Fix $M\geq1$, and for any $i\in\llbracket1,N\rrbracket$ consider
\begin{equation*}
    \mu^i_1\coloneqq\sum_{k=1}^M \alpha_k^i \updelta_{x_{k}^i}\in\mathscr{P}(\S)
\end{equation*}
where $x_k^i\in\conv_g\,\supp\,\mu_0^i$, with $x_k^i = x_{k'}^j$ if and only if $(k, i) = (k', j)$. Then for any $T>0$ and $\varepsilon>0$ there exist piecewise constant $(\bW,\bU,b):[0,T]\to\mathscr{M}_{d\times d}(\R)^2\times\R^d$ such that for any $i$, the corresponding solution $\mu^i$ to \eqref{eq: cauchy.pb}--\eqref{eq: vf} with data $\mu_0^i$, $\bV\equiv0$ and the above parameters, satisfies $\mathsf{W}_2(\mu^i(T),\mu_1^{i})\leq\varepsilon$ as well as, for $i\neq j$,
\begin{equation*}
\conv_g\,\supp\,\mu^i(T)\cap\conv_g\,\supp\,\mu^j(T)=\varnothing.
\end{equation*}
\end{proposition}

The number of switches in $(\bW,\bV, b)$ can also be accounted for---see \Cref{rem: nb.disc.clustering}.

\begin{proof} 
We split $[0,T]=\bigcup_{i\in\llbracket1,N\rrbracket} [T_{i-1},T_{i}],$
where $0=T_0<T_1<\ldots<T_N=T$ are to be determined later on. 
We look to apply \Cref{lem: tubular.mass.movement} separately within each interval, thus, dealing with one measure at a time. Namely, consider
\begin{equation*}
    (\bW,\bU,b)(t)=\sum_{i=1}^N (\bW_i,\bU_i,b_i)(t)1_{[T_{i-1},T_i)}(t),
\end{equation*}
where $(\bW_i,\bU_i,b_i)$ are, roughly speaking, piecewise constant parameters stemming from a repeated application of \Cref{lem: tubular.mass.movement}. We critically use \eqref{eq: eq.tubular.mass.movement} to ensure that when we act on the $i$-th measure in $[T_{i-1}, T_i]$, all the other measures remain invariant, so 
\begin{equation} \label{eq: init.invariant}
 \mu^i(T_{i-1}) = \mu_0^i.   
\end{equation}
Therefore, we take $i\in\llbracket1,N\rrbracket$ to be arbitrary. We proceed in three steps.

\subsubsection*{Step 1. Partitioning each support into $M$ pieces}

Let $\mathscr{C}^i\coloneqq\supp\,\mu^i_0$, and consider a partition $\{\mathscr{C}^i_k\}_{k\in\llbracket1,M\rrbracket}$ of $\mathscr{C}^i$ consisting of pairwise disjoint sets with connected interiors and satisfying $\mu_0^i(\mathscr{C}_k^i)\coloneqq\alpha_k^i$ and $x_k^i\in \mathrm{int}\,\mathscr{C}_k^i$.
Namely
\begin{equation*}
\mathscr{C}^i=\bigcup_{k} \mathscr{C}^i_k
\end{equation*}
with $\mathscr{C}_k^i\cap \mathscr{C}_{k'}^i=\varnothing$ if $k\neq k'$ (see \Cref{fig: shattering}).

\begin{figure}[!ht]
    \centering
    \includegraphics[scale=0.5]{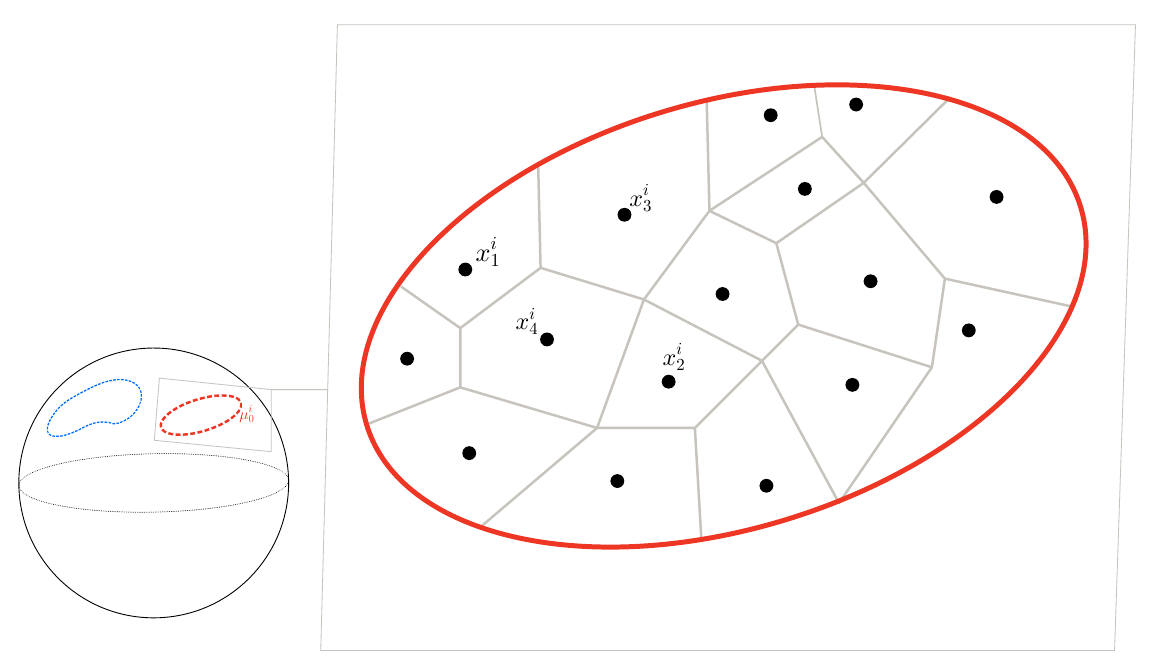}
    \caption{Partitioning $\mathscr{C}^i\coloneqq\supp\,\mu^i_0$ into $M$ pieces with connected interiors.}
    \label{fig: shattering}
\end{figure}

\subsubsection*{Step 2. Packing each part $\mathscr{C}_k^i$ with balls}

We henceforth also fix 
$k\in\llbracket1,M\rrbracket$.
Let $\delta>0$ to be fixed and determined later on. Pack $\mathscr{C}_k^i$ with $\mathsf{N}_{k}^i(\delta)\geq1$ disjoint open balls $B(z_{n,i,k},R_{n,i,k})\subset\mathscr{C}_k^i$ for $n\in\llbracket1,\mathsf{N}_{k}^i(\delta)\rrbracket$,
with $z_{n,i,k}\in \S$ and $R_{n,i,k}>0$, such that
\begin{equation} \label{eq: reminder.packing}
    \mu_0^i\left(\bigcup_{n} B(z_{n,i,k},R_{n,i,k})\right)=\alpha_k^i-\delta.
\end{equation}
We define a target ball contained in $\mathscr{C}_k^i$ to which we aim to send the mass contained in the packing \eqref{eq: reminder.packing}.
Fix the anchor $x_{k}^i\in \mathrm{int}\mathscr{C}^i_k$ and let $\eta>0$ to be determined later on (the same for all indices $(i, k)$) but small enough so that $
\mathscr{B}_k^i\coloneqq B(x_{k}^i,\eta)\subset\mathrm{int}\mathscr{C}^i_k.$
We also pick $\mathscr{B}_k^i$ to satisfy $\mathscr{B}_k^i\subset B(z_{n,i,k},R_{n,i,k})$ for some $n\in\llbracket1,\mathsf{N}_k^i(\delta)\rrbracket$ (see \Cref{fig: packing}).

\subsubsection*{Step 3. Sending most of the mass to $\mathscr{B}_k^i$}

\begin{figure}[!ht]
    \centering
    \includegraphics[scale=0.4]{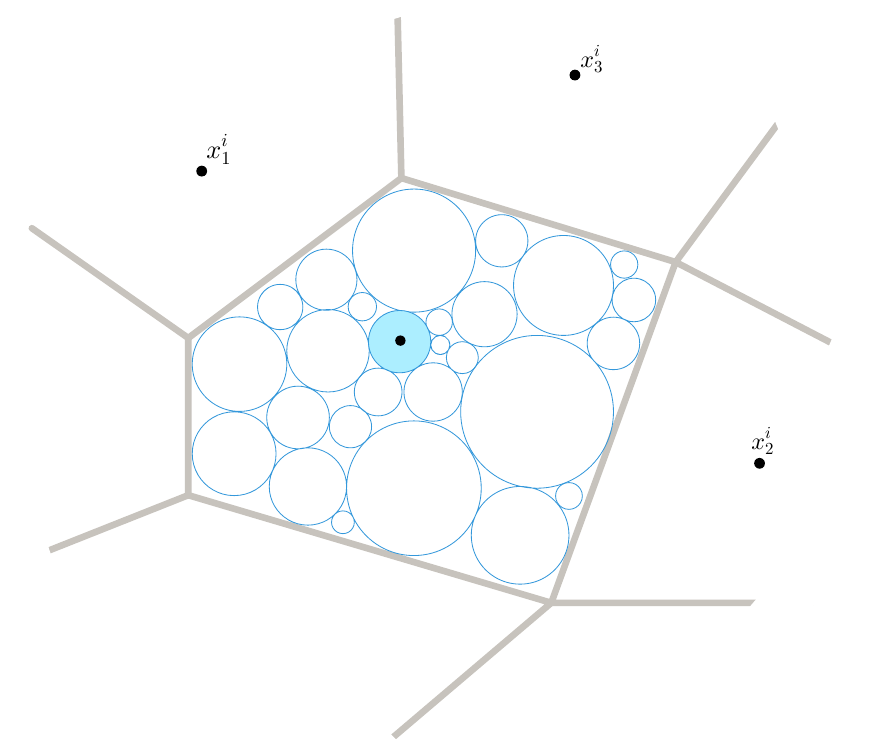}
    \caption{ Step 2: packing the piece $\mathscr{C}_k^i$ of the partition of $\mathscr{C}^i=\mathrm{supp}\,\mu_0^i$ with balls whose union has mass $\mu_0^i(\mathscr{C}_k^i)-\delta$. A single anchor point $x_k^i$ lies in this piece. The goal of Step 3 is to repeatedly use \Cref{lem: tubular.mass.movement} to transfer the mass of each ball to the one highlighted in blue. }
    \label{fig: packing}
\end{figure}

As $\mathrm{int}\mathscr{C}^i_k$ is connected and thus path-connected (here equivalent for open sets), for every $n\in\llbracket1,\mathsf{N}_{k}^i(\delta)\rrbracket$ we can find a sequence of open balls $\{\mathscr{B}_{\ell,n}\}_{\ell\in\llbracket0,L_{k, n}^i\rrbracket}\subset \mathscr{C}_k^i$ satisfying
\begin{equation}\label{eq: homotopy}
    \begin{aligned}
        &\mathscr{B}_{0,n}=B(z_{n,i,k},R_{n,i,k}),\quad\mathscr{B}_{\ell,n}\cap \mathscr{B}_{\ell+1,n}\neq \varnothing,\\
        &\mathscr{B}_{L_{k,n}^i,n}=\mathscr{B}_k^i,\quad\mathscr{B}_{\ell,n}\cap\mathscr{B}_{\ell',n}=\varnothing\hspace{0.5cm}\text{ if }|\ell'-\ell|\geq 2.
    \end{aligned}
\end{equation}
Set $L_k^i\coloneqq\max_{n} L_{k,n}^i$ and  fix an arbitrary $\overline{\varepsilon}>0$ to be determined later on. 
We apply \Cref{lem: tubular.mass.movement} for each piece $k\in\llbracket1,M\rrbracket$ and $n\in\llbracket1,\mathsf{N}_k^i(\delta)\rrbracket$---recalling \eqref{eq: init.invariant}---to find piecewise constant $(\bW_i, \bU_i, b_i):[T_{i-1}, T_i]\to\mathscr{M}_{d\times d}(\mathbb{R})^2\times\R^d$ with at most $
    M\cdot \max_{k} \mathsf{N}_k^i(\delta)\cdot L_k^i$
switches, such that
\begin{align}\label{eq: mass.lower.bound}
    \mu^i(T_i&,\mathscr{B}_k^i)\geq (1-\overline{\varepsilon})^{L_k^i} \mu_0^i\left(\bigcup_{n}\bigcup_{\ell=0}^{L_{k,n}^i} \mathscr{B}_{\ell, n} \right)\nonumber\\
    &\geq (1-\overline{\varepsilon})^{L_k^i} \mu_0^i\left(\bigcup_{n} B(z_{n,i,k},R_{n,i,k})\right)\overset{\eqref{eq: reminder.packing}}{=}(1-\overline{\varepsilon})^{L_k^i} \left(\alpha_k^i-\delta \right).
\end{align}
Moreover, $\mu^i(T_i)=\Phi^{T_i}_\#\mu^i_0$, and $\Phi^{T_i}(x)=x$ for all $x\notin\supp\,\mu_0^i$ because of \eqref{eq: eq.tubular.mass.movement}. 
Using Kantorovich-Rubinstein duality,
\begin{align*}
&\hspace{0.1cm}\mathsf{W}_1(\mu^i(T_i),\mu^i_1) = \sup_{\mathrm{Lip}(\phi)\leq 1}\left|\int\phi(\mu^i(T_i)-\mu_1^i)\right|\\
    &\hspace{0.2cm}= \sup_{\mathrm{Lip}(\phi)\leq 1}\left|\sum_{k}\int_{\mathscr{B}_k^i}\phi(\mu^i(T_i)- \mu_1^i)+\int_{\S\setminus\bigcup_{k}\mathscr{B}_k^i}\phi(\mu^i(T_i)-\mu_1^i)\right|.
\end{align*}
Note that without loss of generality we can maximize over all $\phi\in W^{1,\infty}(\S)$ with $\mathrm{Lip}(\phi)\leq1$ and of average $0$. Such functions have an $L^\infty(\S)$--norm bounded by the length of any geodesic, namely $2\pi$. 
Going term by term in the identity above, using \eqref{eq: mass.lower.bound} and the definition of $\mathscr{B}_k^i$ we find
\begin{align*}
\int_{\mathscr{B}_k^i}& \phi(\mu^i(T_i) -\mu_1^i)=
    \int_{\mathscr{B}_k^i} \phi\mu^i(T_i) -\alpha_k^i\phi(x_k^i)\\
    &=\int_{\mathscr{B}_k^i} \phi\mu^i(T_i) -(\alpha_k^i-\overline{\delta})\phi(x_k^i)-\overline{\delta}\phi(x_k^i)=\int_{\mathscr{B}_k^i} (\phi(x)-\phi(x_k^i))\mu^i(T_i) -\overline{\delta}\phi(x_k^i),
\end{align*}
where $\overline{\delta}\coloneqq\alpha_k^i-\mu^i(T_i,\mathscr{B}_k^i)>0$. 
By virtue of \eqref{eq: homotopy} and \eqref{eq: eq.tubular.mass.movement} we also gather that $\mu(T_i,\mathscr{C}_k^i)=\mu_0(\mathscr{C}_k^i)=\alpha_k^i$, and therefore 
$\alpha_k^i=\mu^i(T_i,\mathscr{C}_k^i)\geq \mu^i(T_i,\mathscr{B}_k^i).$
Owing to \eqref{eq: mass.lower.bound}, we find
$\overline{\delta}\leq \alpha_k^i-(1-\overline{\varepsilon})^{L_k^i}(\alpha_k-\delta),$
which clearly goes to $0$ as $\overline{\varepsilon}$ and $\delta$ go to $0$. Therefore
\begin{equation*}
    \left|\int_{\mathscr{B}_k^i} \phi(\mu^i(T_i) -\mu_1^i)\right|\leq \eta \mu^i(T_i,\mathscr{B}_k^i)+\overline{\delta}\|\phi\|_{L^\infty(\S)},
\end{equation*}
which tends to $0$ as $\delta$, $\overline{\varepsilon}$ and $\eta$ tend to zero. On the other hand, thanks to \eqref{eq: mass.lower.bound},
\begin{align*}
\left|\int_{\S\setminus\bigcup_{k}\mathscr{B}_k^i}\phi(\mu^i(T_i)-\mu_1^i)\right|&\leq 2\pi \mu^i\left(T_i,\S\setminus\bigcup_{k} \mathscr{B}_k^i\right)\\
    &\leq 2\pi \left|1-(1-\overline{\varepsilon})^{\max_{k} L_k^i}\sum_{k}(\alpha_k^i-\delta)\right|\\
    &\leq 2\pi \left|1-(1-\overline{\varepsilon})^{\max_{k} L_k^i}(1-M\delta)\right|,
\end{align*}
which also tends to 0 as $\overline{\varepsilon}$ and $\delta$ tend to $0$.
Therefore, we can choose $\overline{\varepsilon},\delta$ and $\eta$ small enough so that $\mathsf{W}_1\left(\mu^i(T_i),\mu_1^i\right)\leq \varepsilon.$
We can conclude since all Wasserstein distances are equivalent on $\S$.
\end{proof}

\begin{remark} \label{rem: nb.disc.clustering}
We deduce that $(\bW, \bU, b)$ have at most $
    N \cdot M \cdot \max_{\substack{(i, k)\in\llbracket1,N\rrbracket\times\llbracket1,M\rrbracket}} \mathsf{N}_k^i(\delta) \cdot \max_{n\in\llbracket1,\mathsf{N}_k^i(\delta)\rrbracket} L_{k,n}^i$
    switches; $\mathsf{N}_k^i(\delta)$ and $L_{k, n}^i$ being defined in Steps 2 and 3 respectively.
\end{remark}

\section{Disentangling supports} \label{sec: disentanglement}

We show that flows generated by self-attention can disentangle measures with overlapping supports--it actually suffices to consider $\bB\equiv0$. 
Set $\mathbb{Q}_1^{d-1}\coloneqq\S\cap (\R_{>0})^d$ and
 \begin{equation} \label{eq: average.vf}
    \mathsf{v}[\mu](t,x) = \proj_x(\bV(t)\mathbb{E}_{\mu(t)}[z]+\bW(t)(\bU(t)x+b(t))_+).
\end{equation}

\begin{proposition} \label{prop: separation}
Let $T>0$ and $\mu_0^i\in\mathscr{P}(\Q_1^{d-1})$.
There exists a piecewise constant $\theta:[0,T]\to\Uptheta$, having at most $O(d\cdot N)$ switches, such that for all $i\in\llbracket1,N\rrbracket$, the solution $\mu^i$ to \eqref{eq: cauchy.pb}--\eqref{eq: average.vf} with data $\mu_0^i$ and $\theta$ satisfies $\conv_g\, \supp\,\mu^i(T)\cap \conv_g \,\supp\,\mu^j(T)=\varnothing$
if $i\neq j$. 
\end{proposition}

We defer the proof to {\bf \Cref{sec: proof.prop.separation}}.
\Cref{prop: separation} entails the existence of a continuous solution map $\Phi^T_{\theta}:\mathscr{P}(\S)\to\mathscr{P}(\S)$ 
which satisfies
\begin{equation*}
\conv_g\, \supp\,\Phi^T_\theta(\mu_0^i)\cap \conv_g\, \supp\,\Phi^T_\theta(\mu_0^j)=\varnothing
\end{equation*}
for all $i\neq j$.
This is of course totally equivalent to what is stated in \Cref{prop: separation}, but in subsequent arguments, referring directly to the flow map $\Phi^T_\theta$ instead of the parameters $\theta$ significantly eases the presentation, and we choose to do so.

\begin{figure}[!ht]
    \centering
    \begin{tikzpicture}[node distance=3cm, scale=0.85, every node/.style={transform shape}]

    \tikzset{
        section/.style={rectangle, minimum width=3cm, minimum height=1cm, text centered, draw=black, fill=#1},
        arrow/.style={thick, ->, >=Stealth}
    }

    \node (1) [section=vibrantgreen, align=center] {\hypersetup{linkcolor=white}\Cref{lem: first.quadrant}\color{white}:\\ \color{white}Transport to $\mathbb{Q}_1^{d-1}$};
    \node (2) [section=skyblue, above of=1, align=center] {\hypersetup{linkcolor=white}\Cref{lem: perturbation}\color{white}:\\ \color{white}Measures can be\\ \color{white}made ``non-colinear''};
    \node (3) [section=skyblue, right of=2, xshift=2cm, align=center] {\hypersetup{linkcolor=white}\Cref{lem: induction.barycenter}\color{white}:\\ \color{white}disentangle\\ \color{white}``non-colinear'' measures};
    \node (6) [section=brightyellow, below of=3, align=center] {\hypersetup{linkcolor=black}\Cref{prop: separation}:\\Disentanglement};

    \draw [arrow] (1) -- (2);
    \draw [arrow] (2) -- (3);
    \draw [arrow] (3) -- (6);

    \end{tikzpicture}
    \caption{High-level overview of the proof of \Cref{prop: separation}.}
    \label{fig: overview.1}
\end{figure}
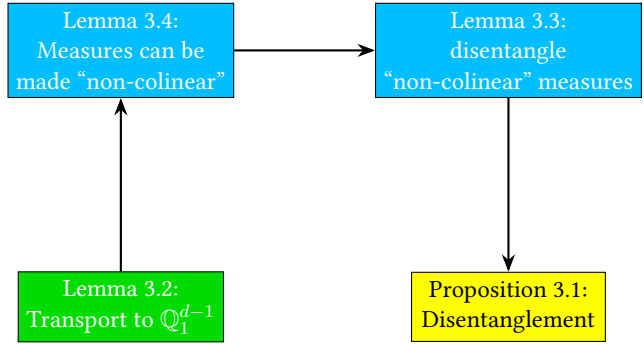

\subsection{\texorpdfstring{Transportation to $\mathbb{Q}_1^{d-1}$}{Transportation to}}

Working with initial measures supported on $\mathbb{Q}_1^{d-1}$ is without loss of generality due to

\begin{lemma} \label{lem: first.quadrant}
Suppose $T>0$ and $\mu_0^i\in\mathscr{P}(\S)$ with $
    \bigcup_{i} \mathrm{supp}\,\mu^i_0\subsetneq\S.$
There exists a piecewise constant $\bW:[0,T]\to\mathscr{M}_{d\times d}(\R)$, having at most one switch and satisfying
\begin{equation*}
\|\bW\|_{L^\infty((0,T);\mathscr{M}_{d\times d}(\R))}\leq C/T
\end{equation*}
for some $C=C(N)>0$, such that for any $i\in\llbracket1,N\rrbracket$ the solution $\mu^i$ to \eqref{eq: cauchy.pb}--\eqref{eq: vf} with data $\mu_0^i$ and $\bV\equiv\bB\equiv\bU\equiv0$, $b\equiv{\bf 1}$, satisfies
$\mathrm{supp}\,\mu^i(T) \subset \mathbb{Q}_1^{d-1}$.
\end{lemma}

\begin{proof}
Pick $\omega\in\S\setminus\bigcup_{i}\supp\,\mu_0^i$ and choose $\bW_1{\bf 1}=-\omega$. The characteristics solve $\dot x(t)=-\proj_{x(t)}\omega$, thence $\dot{\langle x(t),\omega}\rangle=-1+\langle x(t),\omega\rangle^2<0$ off $\{\pm\omega\}$. It ensues that there is some $T_0>0$ such that $\supp\,\mu^i(T_0)\subset B(-\omega,\pi/8)$ for all $i$. Pick $\alpha\in\mathbb{Q}_1^{d-1}$ with $d_g(-\omega,-\alpha)>\pi/8$ so that $-\alpha\notin\supp\,\mu^i(T_0)$, and on $[T_0,T]$ take $\bW\equiv\bW_2$ with $\bW_2{\bf 1}=\alpha$. The characteristics are $\dot x(t)=\proj_{x(t)}\alpha$, so $\dot{\langle x(t),\alpha\rangle}=1-\langle x(t),\alpha\rangle^2\ge0$ drives every point into any prescribed small cap around $\alpha$. Choosing $T-T_0$ large  gives $\supp\,\mu^i(T)\subset\mathbb{Q}_1^{d-1}$ for all $i$. The parameter $\bW$ is piecewise constant with one switch; by time-rescaling the two phases, we deduce the bound.
\end{proof}

\subsection{A pair of lemmas}

The proof of \Cref{prop: separation} is based on the following lemmas.

\begin{lemma} \label{lem: induction.barycenter} 
Let $\mu_0^i\in\mathscr{P}(\Q_1^{d-1})$ be such that $\mathbb{E}_{\mu_0^i}[x] \text{ is not colinear with }\mathbb{E}_{\mu_0^j}[x]$ for $i\neq j$.
    Then for any $T>0$, $\varepsilon>0$, and $\nu_0\in \mathscr{P}(\Q_1^{d-1})$ such that $\mathbb{E}_{\nu_0}[x]$ is colinear with $\mathbb{E}_{\mu_0^N}[x]$, there exists a piecewise constant $\theta:[0,T]\to\Uptheta$ having at most $O(d\cdot N)$ switches such that 
    \begin{equation*}
       \conv_g\, \supp\,\nu(T)\cup  \conv_g\,\supp\,\mu^N(T)\subset B\left(\left.
\mathbb{E}_{\mu^j_0}[z]
\;\middle/\;
\left\|\mathbb{E}_{\mu^j_0}[z]\right\|\right.,\varepsilon\right)
    \end{equation*}
    and $\mu^i(T)=\mu^i_0$ for $i\neq j$, where $\mu^i, \nu$ 
    denote the unique solutions to \eqref{eq: cauchy.pb}--\eqref{eq: average.vf} corresponding to data $\mu_0^i$, $\nu_0$, and parameters $\theta$.
\end{lemma}

We postpone the proof to {\bf \Cref{proof: induction.barycenter}}.

\begin{lemma}\label{lem: perturbation} \label{lem: colinearity}
Let $T>0$ and let $\mu_0,\nu_0\in\mathscr{P}(\Q_1^{d-1})$ be two different measures such that $\mathbb{E}_{\mu_0}[x]= \gamma_1 \mathbb{E}_{\nu_0}[x]$ for some $\gamma_1\in(0,1]$.
\begin{enumerate}
    \item If $\gamma_1=1$, then, setting $\bV\equiv0$, there exist $\bW,\bU\in\mathscr{M}_{d\times d}(\R)$ and $b\in \mathbb{R}^d$ such that the solutions $\mu, \nu$ to \eqref{eq: cauchy.pb}--\eqref{eq: average.vf} corresponding to $\mu_0, \nu_0$ and these parameters, satisfy
\begin{equation*}
    \mathbb{E}_{\mu(T)}[x]\neq \mathbb{E}_{\nu(T)}[x].
\end{equation*}
Moreover the Lipschitz-continuous and invertible flow map $\Phi^T:\S\to\S$ induced by the characteristics of \eqref{eq: cauchy.pb}--\eqref{eq: average.vf} with these parameters satisfies
\begin{equation}\label{eq: identity.flow}
    \Phi^T(x)=x\hspace{1cm}\text{ for }x\in \S\setminus\left( \conv_g\,\supp\,\mu_0\cup{ \conv_g}\,\supp\,\nu_0\right).
\end{equation}
\item If $\gamma_1\neq 1$, then, setting $\bB\equiv0$, there exist $(\bV,\bW,\bU)\in L^\infty((0,T);\mathscr{M}_{d\times d}(\mathbb{R})^3)$ and $b\in L^\infty((0,T);\mathbb{R}^d)$,  piecewise constant with at most $2$ switches, such that the solutions $\mu, \nu$ to \eqref{eq: cauchy.pb}--\eqref{eq: average.vf} corresponding to data $\mu_0, \nu_0$ and these parameters satisfy
\begin{equation*}
    \mathbb{E}_{\mu(T)}[x]\neq\gamma_2\mathbb{E}_{\nu(T)}[x]
\end{equation*}
for all $\gamma_2\in\mathbb{R}$.
\end{enumerate}
\end{lemma}

We postpone the proof to {\bf\Cref{proof: perturbation}}.

\subsection{\texorpdfstring{Proof of \Cref{prop: separation}}{Proof of Proposition}} \label{sec: proof.prop.separation}

\begin{proof}[Proof of \Cref{prop: separation}]

We argue by induction over $N$. The base case $N=1$ is trivially satisfied. Assume 
\begin{equation}\label{eq: empty.inter.sup}
\conv_g\,\supp\,\mu^i_0\cap\conv_g\,\supp\,\mu^j_0=\varnothing \hspace{1cm} \text{ for } i\neq j\in\llbracket1,N-1\rrbracket,
\end{equation}
and let $\mu_0^N\in \mathscr{P}(\Q_1^{d-1})$ be arbitrary. 
We prove there exist $\theta$ as in the statement with
\begin{equation*}
 \conv_g\,\supp\,\mu^i(T)\cap  \conv_g\,\supp\,\mu^j(T)=\varnothing\hspace{1cm} \text{ for } i\neq j\in\llbracket1,N\rrbracket.
\end{equation*}
Since $\supp\,\mu^i_0\subset\Q_1^{d-1}$, \eqref{eq: empty.inter.sup} implies that
\begin{equation*}
    \mathbb{E}_{\mu^i_0}[x]\text{ is not colinear with }\mathbb{E}_{\mu^j_0}[x] \hspace{1cm} \text{ for } i\neq j\in\llbracket1,N-1\rrbracket.
\end{equation*}
Now if $\mathbb{E}_{\mu_0^N}[x]$ is not colinear with $\mathbb{E}_{\mu_0^i}[x]$ for all $i\in\llbracket1,N-1\rrbracket$, one can conclude by a simple application of \Cref{lem: induction.barycenter}, by choosing $\varepsilon$ small enough and considering only the measure $\mu^N$, that the diameter of the convex hull is shrunk until achieving the separation. 
On another hand, as a consequence of \eqref{eq: empty.inter.sup}, $\mathbb{E}_{\mu^N_0}[x]$ is colinear with $\mathbb{E}_{\mu^i_0}[x]$ for at most one $i\in\llbracket1,N-1\rrbracket$. 
Suppose that this is the case, and without loss of generality, we label this index $i=N-1$. We now proceed as follows.
\begin{enumerate}
    \item In $\left[0,T/4\right]$, we apply \Cref{lem: induction.barycenter}, with $\varepsilon>0$ small enough, to guarantee the existence of piecewise constant $\theta_1\in L^\infty((0,T/4); \Uptheta)$ having $O(d\cdot N)$ switches, such that the solution to \eqref{eq: cauchy.pb} satisfies
\begin{align}
    \conv_g\,\supp\,\mu^j\left(T/4\right)\cap  \conv_g\,\supp\,\mu^N\left(T/4\right)=\varnothing\label{eq: first.step}\\
     \conv_g\,\supp\,\mu^j\left(T/4\right)\cap  \conv_g\,\supp\,\mu^{N-1}\left(T/4\right)=\varnothing\nonumber
\end{align}
  for all $j\in\llbracket1,N-2\rrbracket$.
\item In $\left[T/4,T/2\right]$, we apply the first part of \Cref{lem: perturbation} to find constant $\theta_2$ such that 
\begin{equation*}
    \mathbb{E}_{\mu^{N-1}\left(\frac{T}{2}\right)}[x]\neq \mathbb{E}_{\mu^{N}\left(\frac{T}{2}\right)}[x],
\end{equation*}
whereas, thanks to \eqref{eq: identity.flow} and the Lipschitz character of the ODE,
\begin{equation*}
     \conv_g\,\supp\,\mu^j\left(T/2\right)\cap  \conv_g\,\supp\,\mu^{N-1}\left(T/2\right)=\varnothing
\end{equation*}
for all $j\in\llbracket1,N-2\rrbracket$.
\item 
In $[T/2,3T/4]$, we apply the second part of \Cref{lem: colinearity} to $\mu^{N-1}\left(T/2\right)$ and $\mu^{N}\left(T/2\right)$ so that there are some piecewise constant $\theta_3\in L^\infty((T/2,3T/4);\Uptheta)$ such that 
\begin{equation*}
    \mathbb{E}_{\mu^N\left(\frac{3T}{4}\right)}[x] \text{ is not colinear with } \mathbb{E}_{\mu^{N-1}\left(\frac{3T}{4}\right)}[x].
\end{equation*}
Furthermore, owing to \eqref{eq: first.step}, and noting that $\bV=I_d$ in \Cref{lem: perturbation}, along with  the fact that $\conv_g\,\supp(\mu(t))\subset \conv_g\,\supp(\mu_0)$, 
we also have  
\begin{equation*}
    \conv_g\,\supp\,\mu^i\left(3T/4\right)\cap \conv_g\,\supp\mu^j\left(3T/4\right)=\varnothing
\end{equation*}
for all $i\neq j\in\llbracket1,N-1\rrbracket$, and for all $i\in\llbracket1,N-2\rrbracket$ and $j=N$. 
\item The assumption of \Cref{lem: induction.barycenter} is now fulfilled by all $N$ measures, so by picking $\varepsilon>0$ small enough and applying \Cref{lem: induction.barycenter} in $\left[3T/4,T\right]$, the conclusion follows.\qedhere
\end{enumerate}
\end{proof}

\section{Matching discrete measures} \label{sec: neural.ode}

The goal of this section is to prove the following result.

\begin{proposition} \label{prop: interpolation.neural.ode}
Suppose $d\geq3$. Consider 
\begin{equation*} \tag{$\mathscr{D}$}
  (x_0^i,y^i)\in\S\times\S \hspace{1cm} \text{ for } i\in\llbracket1,M\rrbracket, 
\end{equation*}
with $x_0^i\neq x_0^j$ and $y^i\neq y^j$ for $i\neq j$, and suppose that for any $i$ there exist $\gamma_i\in\S$ and $\varepsilon_i>0$ such that $\langle\gamma_i,x_0^i-y^i\rangle=0$ and  $x_0^j\notin H_{\varepsilon_i}^{\gamma_i}$
for $j\neq i$, where
\begin{equation*}
H_{\varepsilon_i}^{\gamma_i}\coloneqq\{x\in \S\colon|\langle x,\gamma_i\rangle|\leq \varepsilon_i\}.
\end{equation*}
For any $T>0$ there exist piecewise constant $\theta=(\bW,\bU,b):[0,T]\to\mathscr{M}_{d\times d}(\R)^2\times\R^d$, having at most $6M$ switches, such that for any $i$, the solution $x^i(\cdot)\in \mathscr{C}^0([0,T];\S)$ to
\begin{equation} \label{eq: neural.ode.sphere}
\begin{cases}
\dot{x}^i(t)=\proj_x\bW(t)(\bU(t)x^i(t)+b(t))_+ &\text{ in } [0, T]\\
x^i(0)=x_0^i,
\end{cases}
\end{equation}
satisfies $x^i(T)=y^i.$
Moreover, there exists $C>0$, not depending on $\mathscr{D}$ nor $T$, such that
\begin{equation*}
\|\theta\|_{L^\infty((0,T);\Uptheta)}\leq \frac{C\cdot M}{\displaystyle T\min_{i}\varepsilon_{i}}.
\end{equation*}
\end{proposition}

The proof of \Cref{prop: interpolation.neural.ode} follows directly from the following result, combined with a straightforward induction argument. 

\begin{proposition} \label{lem: induction.neural.ode}
Suppose $d\geq3$. Consider 
\begin{equation*} \tag{$\mathscr{D}$}
  (x_0^i,y^i)\in\S\times\S \hspace{1cm} \text{ for } i\in\llbracket1,M\rrbracket, 
\end{equation*}
with $x_0^i\neq x_0^j$ and $y^i\neq y^j$ for $i\neq j$, with $x_0^i=y^i$ for $i\in\llbracket1,M-1\rrbracket$, and suppose that there exist $\gamma\in\S$ and $\varepsilon>0$ such that $\langle\gamma,x_0^M-y^M\rangle=0$  and $x_0^i\notin H_{\varepsilon}^{\gamma}$
for all $i\in\llbracket1,M-1\rrbracket$. 

For any $T>0$ there exist piecewise constant $\theta=(\bW,\bU,b):[0,T]\to\mathscr{M}_{d\times d}(\R)^2\times\R^d$, having at most $6$ switches, such that for any $i$, the solution $x^i(\cdot)$ to 
\eqref{eq: neural.ode.sphere} satisfies $x^i(T)=y^i.$
Moreover, there exists $C>0$, not depending on $\mathscr{D}$ and $T$, such that
    \begin{equation*}
        \|\theta\|_{L^\infty((0,T); \Uptheta)}\leq \frac{C}{T\cdot\varepsilon}.
    \end{equation*}
\end{proposition}

The proof is geometrically intuitive but rather technical, so we start with an overview to guide the reader.

\begin{enumerate}
\item \textbf{Active and inactive points.}
Only the active pair $(x_0^M,y^M)$ is allowed to move; the inactive points are the remaining $x_0^i$ with $i<M$. The goal is to send $x_0^M$ to $y^M$ while leaving all inactive points fixed.

\item \textbf{Anchors.}
Choose $\omega\in\S$ with $\langle\gamma,\omega\rangle=0$ and $d_g(\omega,x_0^M),d_g(\omega,y^M)\ge\pi/2$, so $\omega$ is far from both endpoints and orthogonal to $\gamma$. Then pick $\omega_+$ on a geodesic from $\omega$ to $\gamma$ at distance $\pi/8$ from $\omega$, and $\omega_-$ on a geodesic from $\omega$ to $-\gamma$ at distance $\pi/8$ from $\omega$. 
These points will serve to “park” the inactive points inside a small cap around $\omega$.

\item \textbf{Gates.}
A gate is a scalar weight that is $0$ on a chosen spherical cap $\mathscr S$ (no motion there) and positive outside (push toward some $z\in\S$). Concretely it comes from one ReLU $(\langle a,x\rangle-\uptau)_+$ multiplied by a tangent direction via projection.

\item \textbf{Motion.}
We use three actions implemented by the parameters used in the proof.

\begin{enumerate}
\item "Gather" action ($\uppsi_1$).
Use two opposite gates with drift targets $\omega_+$ and $\omega_-$ so that any point in the “on” region moves along a geodesic toward $\omega_\pm$. This monotonically increases $\langle x,\omega_\pm\rangle$ (see estimate \eqref{eq: estimate.neural}) and then exponentially settles near $\omega_\pm$ (cf.\ \eqref{eq: Hartman.Grobman}). After a short time, all inactives lie inside a small cap $B(\omega,3\pi/16)$, while points in the “off” halfspace stay put.

\item "Corridor" ($\uppsi_2$).
Activate a gate with $a=-\omega$ and $\uptau=\cos(3\pi/16)$, so it vanishes on $B(\omega,3\pi/16)$ (the cap containing the inactive points) and is positive outside. Choose $z_1,z_2\in \S$ on a geodesic from $x_0^M$ to $z_2$ that passes through $y^M$ and remains outside the cap. Evolve the system for one time interval with drift target $z_1$; this drives $x_0^M$ toward $z_1$ and makes it uniformly close (cf.\ \eqref{eq: geodesic.toll}). Then switch the drift target to $z_2$ and evolve for the next interval; the trajectory follows the same corridor and reaches $y^M$. Throughout, the inactive points do not move because the gate is identically zero on $B(\omega,3\pi/16)$.

\item "Restore" action ($\uppsi_1^{-1}$).
Re-run the gather action with flipped drifts (same gates, opposite $z$), which time-reverses the first action and returns all inactive points to their original positions; the active point stays at $y^M$ since the corridor gate remains off on the cap.
\end{enumerate}
The composition $(\uppsi_1)^{-1}\circ\uppsi_2\circ\uppsi_1$ maps $x_0^M\mapsto y^M$ and fixes all other $x_0^i$. Choosing drift size $\|\bW\|\asymp (T\varepsilon)^{-1}$ results in total time $T$ and yields $\|\theta\|_{L^\infty}\lesssim 1/(T\varepsilon)$.
\end{enumerate}

\begin{proof}[Proof of \Cref{lem: induction.neural.ode}]
The parameters take the form
\begin{align*}
    (\bW(t), \bU(t),b(t)) = \sum_{j=1}^6 (\bW_j,\bU_j,b_j) 1_{\left[\frac{(j-1)T}{6},\frac{jT}{6}\right]}(t),
\end{align*}
where $\bU_5 = \bU_1$, $\bU_6 = \bU_2$, $\bU_3 = \bU_4$, $b_5 = b_1$, $b_6 = b_2$, $b_3 = b_4$, $\bW_5 = -\bW_1$, and $\bW_6 = -\bW_2$.  The precise matrices and vectors, as well as $\bW_3$, $\bW_4$, are defined later on, and \( T > 0 \) is adjusted later by rescaling the norm of the parameters.

\subsubsection*{Step 1. The anchor points}
In this step we find three anchor points which serve to build the parameters in what follows. 
Since $\langle \gamma, x_0^M-y^M\rangle=0$, we can find some $\omega\in\S$ such that 
\begin{equation}\label{eq: great.circle.hyperplane}
\langle \gamma, \omega\rangle=0,
\end{equation}
as well as
\begin{equation}\label{eq: omega.distance}
d_g(\omega, x_0^M)\geq \frac{\pi}{2},\hspace{1cm}\text{ and } \hspace{1cm}  d_g(\omega, y^M)\geq\frac{\pi}{2}.
\end{equation}
Because of \eqref{eq: great.circle.hyperplane}, we consider the point $\omega_+$ lying on the minimizing geodesic between $\omega$ and $\gamma$, satisfying $d_g(\omega_+,\omega)=\pi/{8}.$
Similarly, we consider the point $\omega_-$ lying on the minimizing geodesic between $\omega$ and $-\gamma$, satisfying 
$d_g(\omega_-,\omega)=\pi/{8}.$
We have 
\begin{align*}
d_g(\omega_+,x_0^M)&\geq d_g(\omega,x_0^M)-d_g(\omega,\omega_+)\geq\frac{3\pi}{8},\\
d_g(\omega_+,y^M)&\geq \frac{3\pi}{8},\quad d_g(\omega_-,x_0^M)\geq \frac{3\pi}{8},\quad d_g(\omega_-,y^M)\geq \frac{3\pi}{8}.
\end{align*}
As a consequence, the hyperplane $\{x\in\S\colon\langle\omega,x\rangle=\cos(\pi/8+\tau)\}$
is a separating hyperplane for the ball $B\left(\omega,\pi/8+\tau\right)$ and the points $x_0^M$ and $y^M$ for every $\tau\in(0,3\pi/8)$; namely
\begin{equation*}
\langle \omega,x_0^M\rangle -\cos\left(\frac{\pi}{8}+\tau\right)=\cos d_g(\omega,x_0^M)-\cos\left(\frac{\pi}{8}+\tau\right)<0,
\end{equation*}
where the inequality is by virtue of \eqref{eq: omega.distance}. 
Analogous computations hold for $y^M$, whereas
\begin{equation*}
\langle\omega,x\rangle -\cos\left(\frac{\pi}{8}+\tau\right)>0
\end{equation*}
for all $x\in B(\omega,\pi/8+\tau/2)$ and $\tau\in(0,3\pi/8)$ (see \Cref{fig: setup.ball} for an illustration of the geometric setup).

\begin{figure}[!ht]
    \centering
    \includegraphics[scale=0.55]{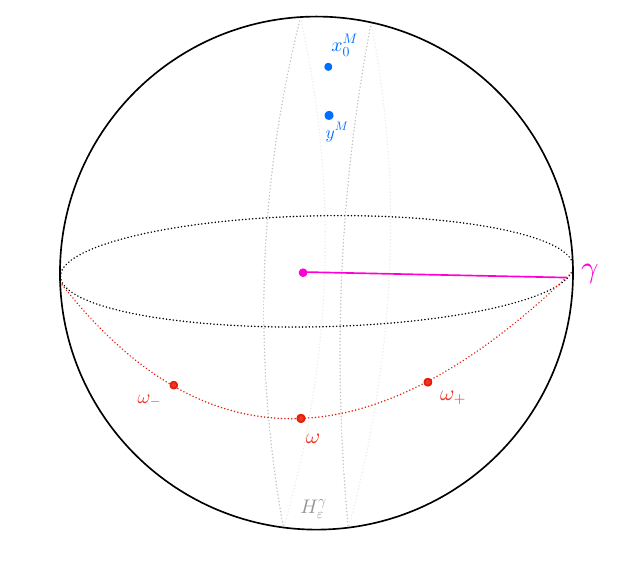}
    \caption{The geometric configuration of Step 1.}
    \label{fig: setup.ball}
\end{figure}
            
\subsubsection*{Step 2. Isolating $x_0^M$ and $y^M$} 
Let $\epsilon\coloneqq\min\{\varepsilon,\pi/4\}$,  
$ \bU_1= {\bf1}\gamma^\top$, $b_1=-\frac{\epsilon}{2}{\bf1}$, so
 $(\bU_1x+b_1)_+=(\langle \gamma, x\rangle -\epsilon/2)_+{\bf1}.$
Pick any $\bW_1$ with
$\bW_1{\bf1}=\omega_+$ and set $\mathscr{S}_+\coloneqq\{ x\in \S\colon\langle \gamma, x\rangle \geq \epsilon\}.$
Obviously $\omega_+\in \mathscr{S}_+$. 
Observe that the trajectories of the ODE
\begin{equation} \label{eq: neural.ode.separation}
\dot{x}(t)=(\langle \gamma, x(t)\rangle -\epsilon/2)_+\proj_{x(t)}\omega_+\hspace{1cm}\text{ for $t\geq0$,}
\end{equation}
follow the Riemannian gradient flow of the distance between $\omega_+$ and $x$ in $\mathscr{S}_+$. Indeed,
\begin{equation*}
     \nabla_1 d_g(x,\omega_+)=-\frac{\proj_x(\omega_+)}{\sqrt{1-\langle x,\omega_+\rangle^2 }}.
\end{equation*}
Then, setting $f(x)=(\langle \gamma, x\rangle -\epsilon/2)_+$, we have
 \begin{align*}
    \dot{x}(t)&=-f(x(t))\,\sqrt{1-\langle x(t),\omega_+\rangle^2 } \,\nabla_1 d_g(x(t),\omega_+)=-\overline{f}(x(t))\,\nabla_1 d_g(x(t),\omega_+).
 \end{align*}
 Since $\overline{f}$ is a nonnegative scalar function, by appropriately reparameterizing time, we conclude that $x(t)$ follows the desired gradient flow.
 In turn, the trajectory $x(t)$ of \eqref{eq: neural.ode.separation} starting from any $x_0\in\mathscr{S}_+$ always lies on the minimal geodesic from $x_0$ to $\omega_+\in\mathscr{S}_+$. Since $\mathscr{S}_+$ is geodesically convex, we gather that $x(t)\in\mathscr{S}_+$ for all $t\geq0$. Then, notice that 
 \begin{equation*}
     \overline{f}(x)=0 \qquad \Longleftrightarrow \qquad x=\omega_+ \quad \text{ or } \quad x\in  \left\{y\in\S:\langle \gamma,y\rangle\leq \frac{\epsilon}{2}\right\}.
 \end{equation*}
 Thus, unless $x(t)=\omega_+$, $\overline{f}$ is uniformly bounded from below on $\mathscr{S}_+$, since $\nabla_1 d_g(x,\omega_+)=0$ is equivalent to $x=\pm \omega_+$, 
 we can conclude that $x(t)\to\omega_+$ as $t\to+\infty$ for any $x_0\in\mathscr{S}_+$ by applying the LaSalle invariance principle.

For any $x_0\in\S$, set $T_{\frac{\pi}{16}}(x_0)\coloneqq\inf\{t\geq 0\colon x(t)\in B(\omega_+,\pi/{16})\},$
where $x(\cdot)$ is the solution to the Cauchy problem for \eqref{eq: neural.ode.separation} with data $x_0$. 
Since $\|\gamma\|=1$, $\|\bW_1\|_{\text{op}}\leq1$ and $\|b_1\|\leq\epsilon\sqrt{d}/2$, bounding the $L^\infty$--norm of the parameters comes from bounding $T_{\frac{\pi}{16}}(x_0)$ uniformly over $x_0\in\mathscr{S}_+$ and rescaling time. 
For every $x_0\in B(\omega_+,\pi/{16})$ we see that $T_{\frac{\pi}{16}}(x_0)$ is trivially $0$, whereas for  $x_0\in \mathscr{S}_+\setminus B(\omega_+,\pi/{16})$ one has
\begin{align} \label{eq: estimate.neural}
\frac{\diff}{\diff t} \langle x(t),\omega_+\rangle&=\left( \langle \gamma, x(t)\rangle-\frac{\epsilon}{2}\right)_+\left(1-\langle \omega_+,x\rangle^2\right)\nonumber\\
&\geq \frac{\epsilon}{2}\left(1-\langle \omega_+,x\rangle^2\right)\geq \frac{\epsilon}{2}\left(1-\cos^2\left(\frac{\pi}{16}\right)\right).
\end{align}
Hence $T_{\frac{\pi}{16}}(x_0)=O(1/\epsilon)$ for all $x_0\in \mathscr{S}_+$.  

Finally, by following the same arguments leading to \eqref{eq: estimate.neural}, beyond some large enough time, and for every $x_0\in \mathscr{S}_+$, we can apply the Hartman-Grobman theorem \cite{shub2013global}: the behavior near the critical point $\omega_+$ is governed\footnote{Note that the critical point $\omega_+$ is hyperbolic since we are working in $\mathsf{T}_{\omega_+}\S$. On $\mathbb{R}^d$, there is a zero eigenvalue associated to the radial direction.} by the linearized system 
 \begin{equation*}
 \begin{cases}
     \dot{y}(t)=-\left(\langle \gamma,\omega_+\rangle-\displaystyle\epsilon/2\right)y(t) & \text{ in } \R_{\geq0}\\
     y(0)=y_0\in \mathsf{T}_{\omega_+}\S,
     \end{cases}
 \end{equation*}
 which is exponentially stable. In other words, for all $x_0\in\mathscr{S}_+$,
 \begin{equation}\label{eq: Hartman.Grobman}
     d_g(x(t),\omega_+)\leq Ke^{-\lambda t} \hspace{1cm} \text{  for all } t\geq 0, 
 \end{equation}
and for some $\lambda>0$ and $K\geq 1$ which depend on $x_0$, $\epsilon$ and $\gamma$ only.

Similarly, let $\bU_2=- {\bf1}\gamma^\top$ and $b_2=-\frac{\epsilon}{2}{\bf1}$, so $(\bU_2x+b_2)_+=( \langle -\gamma, x\rangle -\epsilon/2)_+ {\bf1}$. 
Choose any $\bW_2$ with $\bW_2{\bf1}=\omega_-$ and set $\mathscr{S}_-\coloneqq\{ x\in \S\colon \langle -\gamma, x\rangle \geq \epsilon\}$. After reasoning similarly for $\mathscr{S}_-$ as for $\mathscr{S}_+$, and by rescaling time so that $\|\bW_1\|_{\text{op}}$ and $\|\bW_2\|_{\text{op}}$ are $O(1/(T\cdot\epsilon))$, 
we deduce that for any $T>0$ there exist piecewise constant $\theta_1:[0,T/3]\to\mathscr{M}_{d\times d}(\mathbb{R})^2\times \mathbb{R}^d$, with two switches, such that the associated flow map $\upphi^{T/3}_{\theta_1}$ of \eqref{eq: neural.ode.sphere} is a Lipschitz-continuous invertible map satisfying
\begin{align}\label{eq: sphere.separation}
\upphi^{\frac{T}{3}}_{\theta_1}(x_0^i)\in B(\omega,3\pi/16),\quad
\upphi^{\frac{T}{3}}_{\theta_1}(x)=x\hspace{1cm}\text{ if }x\in H_\epsilon^\gamma.
\end{align}
(See \Cref{fig:?} for an illustration of the the isolation of $x_0^M$ and $y^M$.)

\begin{figure}
    \centering
        \includegraphics[scale=0.6]{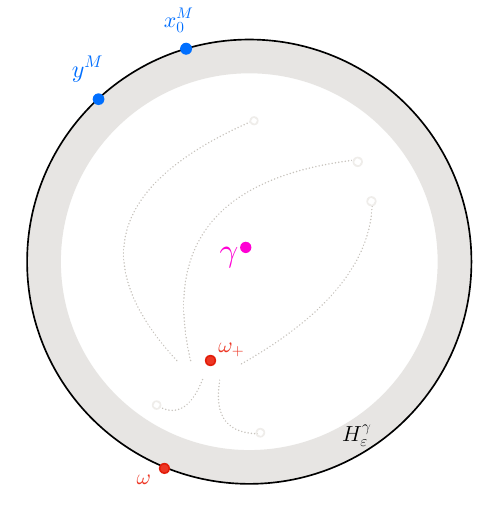}
    \hspace{0.5cm}
        \includegraphics[scale=0.6]{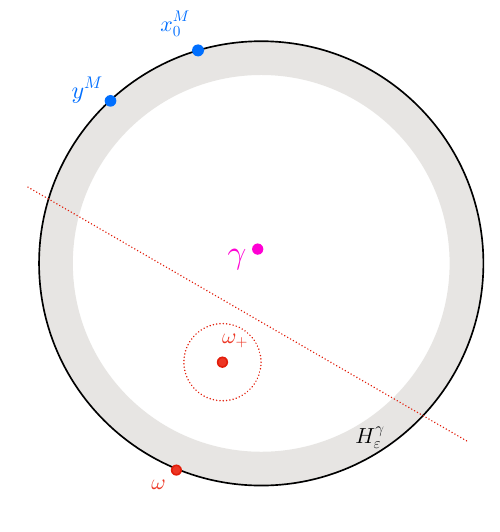}
    \caption{({\it Left}) All points in one spherical cap converge to $\omega_+$. ({\it Right}) All points aside the $M$-th one are in a neighbourhood of $\omega_+$ or $\omega_-$. Consequently there is a separating hyperplane between $x_0^M$ and $y^M$ and the rest of the points (dashed).}
    \label{fig:?}
\end{figure}

\subsubsection*{Step 3. Steering $x_0^M$ to $y^M$}

By virtue of \eqref{eq: sphere.separation}, the hyperplane $\{x\in\S\colon\langle\omega, x\rangle =\cos(3\pi/16)\}$
is a separating hyperplane: it separates $x_0^N$ and $y^N$ from $\upphi_{\theta_1}^{T/3}(x_0^i)$ for $i\in\llbracket1,M-1\rrbracket$. Consider $\bU_3=- {\bf1}\omega^\top$ and $b_3=\cos({3\pi}/{16}){\bf1}$. We have 
\begin{align*}
    \left( \bU_3x+b_3 \right)_+=0 \hspace{1cm} \text{ for } &x\in B(\omega, {3\pi}/{16}),\\
    \left( \bU_3x+b_3 \right)_+= \underbrace{\left(\langle -\omega, x\rangle+\cos\left(\frac{3\pi}{16}\right)\right)}_{>0} {\bf1} \hspace{1cm} \text{ for } &x\in \S\setminus B(\omega,{3\pi}/{16}).
\end{align*}
Take two points $z_1,z_2\in \S\setminus B(\omega, {3\pi}/{16})$ such that
\begin{enumerate}
    \item $\{c(s)\}_{s\in[0,1]}$ is a geodesic satisfying $c(0)=x_0^M$ and $c(1)=z_2$;
    \item $c(1-s_0)=y^M$ for some $s_0\in(0,1)$;
    \item $c(s_{z_1})=z_1$ for some $s_{z_1}\in(0,1-s_0)$;
    \item $\{c(s)\}_{s\in[0,1]}\subset \S\setminus B(\omega, {3\pi}/{16})$;
    \item $d_g(z_1,x_0^M)\leq \kappa\pi$ and $d_g(z_1,z_2)\leq \kappa\pi$ for some\footnote{can be chosen as such because $\{c(s)\}_{s\in[0,1]}\subset \S\setminus B(\omega,{3\pi}/{16})$---indeed, take $\kappa={29}/{32}$.} $\kappa<1$.
\end{enumerate}
Consider any $d\times d$ matrix $\bW_3$ such that $\bW_3{\bf1}=z_1$.
The Cauchy problem \eqref{eq: neural.ode.sphere} with these parameters, for the $i=M$-th particle, reads
\begin{equation}\label{eq: neural.ode.step3}
\begin{cases}
\dot{x}(t)=\left(\langle -\omega, x(t)\rangle+\cos\left(\frac{3\pi}{16}\right)\right)_+\proj_{x(t)}z_1 &\text{ on } \R_{\geq0}\\
x(0)=x_0^M.
\end{cases}
\end{equation}
Since $d_g(x_0^M,z_1)\leq\kappa\pi$, and since the minimizing geodesic between $x_0^M$ and $z_1$ is contained in $\{c(s)\}_{s\in [0,1]}\subset B(\omega,{3\pi}/{16})$, we gather that there exists some large enough time $\tau>0$ such that
\begin{equation} \label{eq: geodesic.toll}
    d_g\left( x(\tau),z_2\right)\leq d_g\left( x(\tau),z_1\right)+d_g\left( z_1,z_2\right)\lesssim e^{-\lambda\tau}+\kappa\pi\leq\kappa_2\pi
\end{equation}
for some $\kappa_2<1$ and $\lambda>0$. This comes from the long-time convergence of \eqref{eq: neural.ode.step3} to $z_1$, which can be shown by following the same arguments as for \eqref{eq: Hartman.Grobman}, replacing $\omega_+$ by $z_1$. 
For any $d\times d$ matrix $\bW_4$ such that $\bW_4{\bf1}=z_2$, the Cauchy problem \eqref{eq: neural.ode.sphere}, for the $i=M$-th particle, reads
\begin{equation} \label{eq: neural.ode.step3.2}
\begin{cases}
\dot{x}(t)=\left(\langle -\omega, x(t)\rangle+\cos\left(\frac{3\pi}{16}\right)\right)_+\proj_{x(t)}z_2 &\text{ for } t\geq\tau\\
x(\tau) = x(\tau)
\end{cases}
\end{equation}
where $x(\tau)$ is the solution of \eqref{eq: neural.ode.step3} at $t=\tau$. Since $d_g(z_1,z_2)\leq \kappa \pi$, $y^M$ lies on the minimizing geodesic between $x(\tau)$ and $z_2$. All the while, thanks to \eqref{eq: geodesic.toll}, taking $T$ even larger than before, we deduce that that the solution to \eqref{eq: neural.ode.step3.2} satisfies
$x(T)=y^M.$
Thus, as in the previous step, we deduce that for any $T>0$ there exist piecewise constant $\theta_2:[T/3,2T/3]\to\mathscr{M}_{d\times d}(\mathbb{R})^2\times\mathbb{R}^d$, with two switches, such that the associated flow map $\upphi^{2T/3}_{\theta_2}$ of \eqref{eq: neural.ode.sphere} is a Lipschitz-continuous invertible map satisfying
\begin{align*}
\upphi^{\frac{2T}{3}}_{\theta_2}(x)&=x\hspace{1cm}\text{ if }x\in B(\omega, {3\pi}/{16}),\\
\upphi^{\frac{2T}{3}}_{\theta_2}(x_0^M)&=y^M,
\end{align*}
and
\begin{equation*}
    (\upphi_{\theta_2}^{\frac{2T}{3}}\circ\upphi_{\theta_1}^{\frac{T}{3}})(x_0^i)= \upphi_{\theta_1}^{\frac{T}{3}}(x_0^i)\in B(\omega,{3\pi}/{16})\qquad\text{ for } i\in\llbracket1,M-1\rrbracket
\end{equation*}
as well as
\begin{equation*}
    (\upphi_{\theta_2}^{\frac{2T}{3}}\circ\upphi_{\theta_1}^{\frac{T}{3}})(x_0^M)= \upphi_{\theta_2}^{\frac{2T}{3}}(x_0^M)=y^M.
\end{equation*}

\subsubsection*{Step 4. Bringing $x^i(T)$ back to $y^i$}
We conclude by applying the inverse of $\upphi^{\frac{T}{3}}_{\theta_1}$: defining
$\upphi^T_{\text{fin}}\coloneqq(\upphi^{\frac{T}{3}}_{\theta_1})^{-1}\circ \upphi^{\frac{2T}{3}}_{\theta_2} \circ\upphi^{\frac{T}{3}}_{\theta_1},$
we have $\upphi^T_{\text{fin}}(x_0^i)=y^i$ for all $i\in\llbracket1,M\rrbracket$, as desired.
\end{proof}

\begin{remark}
\Cref{prop: compression} yields a flow map that clusters the support of the input measure, which in turn allows to reduce a universal approximation property of maps in $L^p(\S;\S)$ to interpolation via flow maps  proved in \Cref{prop: interpolation.neural.ode}. 
Indeed it suffices to consider a simple function 
$\varphi(x)=\sum_{i=1}^N y_i1_{\Omega_i}(x)$
with $y_i\in \S$. Universal approximation in $L^p(\S;\S)$, $p<+\infty$, is equivalent to the $\mathsf{W}_p$-approximate interpolation of
\begin{equation*}
\diff\mu_0^i(x)=1 _{\Omega_i}(x)\diff x,\quad
\mu_1^i=|\Omega_i|\updelta_{y_i}
\end{equation*}
for $i\in\llbracket1,N\rrbracket$. Note that, by construction, the supports of $\mu_0^i$ (and of $\mu_1^i$) are pairwise disjoint. Thus the attention component of the vector field is not needed to perform this task. This is generalized in the next section.
\end{remark}

\section{Proofs of the main results} \label{sec: proofs}

Our overarching goal is to construct the solution map $\Phi_{\mathrm{fin}}^T:\mathscr{P}(\S)\to\mathscr{P}(\S)$ of the form
$$\Phi_{\mathrm{fin}}^T\coloneqq(\Phi^{\frac{T}{3}}_{\theta_3})^{-1}\circ \Phi^{\frac{T}{3}}_{\theta_2}\circ \Phi^{\frac{T}{3}}_{\theta_1}, 
$$
where $\Phi_{\theta_1}^{T/3}$ and $\Phi_{\theta_3}^{T/3}$ 
stem from \Cref{prop: separation}, so that $\Phi_1(\mu_0^i)$---and $\Phi_3(\mu_1^i)$---for $i\in\llbracket1,N\rrbracket$, have pairwise disentangled supports. The map 
$\Phi_{\theta_2}^{T/3}$ 
is constructed in this section (see \Cref{fig: overview.2} for a schematic overview of the entire proof). 

The main clue lies in the following three lemmas.

\begin{lemma}
\label{lem: hyp.propagation}
Suppose that for every $i$ there exists $\mathsf{T}^i\in L^2(\S;\S)$ with $\mathsf{T}^i_\#\mu_0^i=\mu_1^i$.
Consider the flow map $\Phi_1:\mathscr{P}(\S)\to\mathscr{P}(\S)$ (resp. $\Phi_3$) given by \Cref{prop: separation} with data $\mu_0^i\in\mathscr{P}(\S)$ (resp. $\mu_1^i\in\mathscr{P}(\S)$). 
Then there exists a Lipschitz-continuous and invertible map $\uppsi:\S\to\S$ such that
\begin{equation} \label{eq: diffeo.supports}
\uppsi\big|_{\supp\,\Phi_1(\mu_0^i)}=\uppsi^i\big|_{\supp\,\Phi_1(\mu_0^i)}
\end{equation}
for any $i\in\llbracket1,N\rrbracket$, where $\uppsi^i:\S\to\S$ is another Lipschitz-continuous and invertible map that satisfies
\begin{equation*}
(\uppsi^i\circ \mathsf{T}^i_{\Phi_1})_\#\mu_0^i= (\uppsi\circ \mathsf{T}^i_{\Phi_1})_\#\mu_0^i=(\mathsf{T}^i_{\Phi_3})_\#\mu_1^i=\Phi_3(\mu_1^i)
\end{equation*}
for some Lipschitz-continuous and invertible maps $\mathsf{T}_{\Phi_1}^i, \mathsf{T}_{\Phi_3}^i:\S\to\S$.
\end{lemma}

The proof can be found in {\bf \Cref{proof: hyp.propagation}}.

\begin{lemma}
\label{lem: monge}
Suppose $\mu\in \mathscr{P}(\S)$ and $\mathsf{T}^1,\mathsf{T}^2:\S\to\S$ measurable, with $\mathsf{T}^1$ bijective. Then
\begin{equation} \label{eq: linearized.wasserstein}
\mathsf{W}_2\left(\mathsf{T}^1_\#\mu,\mathsf{T}^2_\#\mu\right)\lesssim\left\|\mathsf{T}^1-\mathsf{T}^2\right\|_{L^2(\mu)}.
\end{equation}
\end{lemma}

The proof is elementary, but brief, thus we provide it for completeness.

\begin{proof}[Proof of \Cref{lem: monge}]
Since $\mathsf{T}^1$ is bijective, there is a measurable $\uppsi:\S\to \S$ such that $\uppsi(\mathsf{T}^1(x))=\mathsf{T}^2(x)$ for all $x\in\S$. 
Then
\begin{equation*}
\mathsf{W}_2^2\left(\mathsf{T}^1_\#\mu,\mathsf{T}^2_\#\mu\right)\lesssim\int\|x-\uppsi(x)\|^2(\mathsf{T}^1_\#\mu)(\diff x)=\left\|\mathsf{T}^1-\mathsf{T}^2\right\|_{L^2(\mu)}^2.\qedhere
\end{equation*}
\end{proof}

\begin{remark}
    When $\mu$ is absolutely continuous with respect to the Lebesgue measure, and
    $\mathsf{T}^1$ and $\mathsf{T}^2$ are the optimal transport maps between $\mu$ and $\nu_1$, and $\mu$ and $\nu_2$ respectively, the upper bound in \eqref{eq: linearized.wasserstein} is the \emph{linearized optimal transport distance} (see  \cite{delalande2023quantitative, jiang2023algorithms}).
\end{remark}

Finally,

\begin{lemma} \label{lem: univ.approx}
Suppose $\varepsilon>0$ and $\mu\in\mathscr{P}(\S)$. 
For every $\uppsi\in L^2(\S;\S)$, there exists a Lipschitz-continuous and invertible map $\uppsi_\varepsilon:\S\to\S$ induced by the solution map of \eqref{eq: neural.pde.sphere}, namely 
$\Phi^T_{\theta_\varepsilon}(\mu)=(\uppsi_\varepsilon)_\#\mu$ for some piecewise constant $\theta_\varepsilon: [0,T]\to\Uptheta$ with finitely many switches, such that
\begin{equation*}
\left\|\uppsi-\uppsi_\varepsilon\right\|_{L^2(\mu)}\leq\varepsilon.
\end{equation*}
\end{lemma}

The proof of \Cref{lem: univ.approx} is involved, so we postpone it to {\bf \Cref{proof: univ.approx}}.

\subsection{\texorpdfstring{Proof of \Cref{thm: main.result}}{Proof of Theorem}}

We provide \emph{two} proofs: we first provide the proof in full generality, followed by a simpler proof that doesn't rely on \Cref{lem: univ.approx}, under stronger structural assumptions on the input and target measures.

\begin{proof}[Proof of \Cref{thm: main.result} (general case)]

We split the proof in three steps.

\subsubsection*{Step 1. Disentanglement} 

We begin by rendering the supports of the initial measures $(\mu_0^i)_{i\in\llbracket1,N\rrbracket}$ (resp. the target measures $(\mu_1^i)_{i\in\llbracket1,N\rrbracket}$) pairwise disjoint by applying \Cref{prop: separation} to \eqref{eq: cauchy.pb} with data $\mu_0^i$ at time $t=0$ (resp. $\mu_1^i$ at time $t=2T/3$) for all $i$. 
This entails the existence of two parameterized flow maps: $\Phi_{\theta_1}^{t}:\mathscr{P}(\S)\mapsto \mathscr{P}(\S)$ for $t\in[0,T/3]$, and $
    \Phi_{\theta_3}^{t}:\mathscr{P}(\S)\mapsto \mathscr{P}(\S)$
for $t\in[2T/3, T]$, induced by \eqref{eq: cauchy.pb}, which are such that
\begin{equation*}
    \supp\,\Phi_{\theta_1}^{\frac{T}{3}}(\mu_0^i)\cap \supp\,\Phi_{\theta_1}^{\frac{T}{3}}(\mu_0^j)=\varnothing\hspace{1cm}\text{ if }i\neq j,
\end{equation*}
and 
\begin{equation*}
    \supp\,\Phi_{\theta_3}^{T}(\mu_1^i)\cap \supp\,\Phi_{\theta_3}^{T}(\mu_1^j)=\varnothing\hspace{1cm}\text{ if }i\neq j.
\end{equation*}
Since \eqref{eq: cauchy.pb} is well-posed and time-reversible, 
we further gather that there exists some constant $C=C(T,\theta_3)>0$ such that
\begin{equation} \label{eq: continuity.inverse.flow}
\mathsf{W}_2\left((\Phi_{\theta_3}^{T})^{-1}(\mu),(\Phi_{\theta_3}^{T})^{-1}(\nu)\right)\leq C\cdot\mathsf{W}_{2}(\mu,\nu)
\end{equation}
holds for any $\mu,\nu\in \mathscr{P}(\S)$.

\subsubsection*{Step 2. Matching}

By \Cref{lem: hyp.propagation} there exists a Lipschitz-continuous, invertible $\uppsi:\S\to \S$ with 
\begin{equation*}
    \uppsi\Big|_{\supp\,\Phi^{\frac{T}{3}}_{\theta_1}(\mu_0^i)}=\uppsi^i\Big|_{\supp\,\Phi^{\frac{T}{3}}_{\theta_1}(\mu_0^i)},
\end{equation*}
for $i\in\llbracket1,N\rrbracket$ where $\uppsi^i:\S\to\S$ satisfies 
\begin{equation} \label{eq: first.psi}
    \uppsi^i_\# \Phi^{\frac{T}{3}}_{\theta_1}(\mu_0^i)=\Phi^{\frac{T}{3}}_{\theta_3}(\mu_1^i).
\end{equation}
Consider $\mu=\sum_{i=1}^N \Phi_{\theta_1}^{\frac{T}{3}}(\mu_0^i),$
and use \Cref{lem: univ.approx} to find a flow map $\uppsi_\varepsilon:\S\to\S$ such that
\begin{equation} \label{eq: second.psi}
\left\|\uppsi^i\Big|_{\supp\,\Phi_{\theta_1}^{\frac{T}{3}}(\mu_0^i)}-\uppsi_\varepsilon\Big|_{\supp\,\Phi_{\theta_1}^{\frac{T}{3}}(\mu_0^i)}\right\|_{L^2(\Phi_{\theta_1}^{\frac{T}{3}}(\mu_0^i))}\leq\|\uppsi-\uppsi_\varepsilon\|_{L^2(\mu)}\leq\frac{\varepsilon}{C}
\end{equation}
for $i\in\llbracket1,N\rrbracket$. 
\Cref{lem: univ.approx} yields  $\Phi_{\theta_2}^{t}:\mathscr{P}(\S)\to \mathscr{P}(\S)$
for $t\in[T/3,2T/3]$, induced by \eqref{eq: cauchy.pb} and defined, for $\nu\in \mathscr{P}(\S)$, as $
    \Phi_{\theta_2}^{2T/3}(\nu)=(\uppsi_\varepsilon)_\#\nu,$
which by virtue of \eqref{eq: first.psi}, \eqref{eq: second.psi} and \Cref{lem: monge}, satisfies
\begin{equation}\label{eq: gluing.flow}
    \mathsf{W}_2\left((\Phi_{\theta_2}^{\frac{2T}{3}}\circ \Phi_{\theta_1}^{\frac{T}{3}})(\mu_0^i),\Phi_{\theta_3}^{T}(\mu_1^i)\right)\leq\frac{\varepsilon}{C}
\end{equation}
for all $i\in\llbracket1,N\rrbracket$.

\subsubsection*{Step 3. Continuity}

We apply the inverse of $\Phi_{\theta_3}^{T}$ to conclude that $\Phi_{\text{fin}}^T\coloneqq(\Phi_{\theta_3}^{T})^{-1}\circ\Phi_{\theta_2}^{2T/3}\circ \Phi_{\theta_1}^{T/3}$
satisfies
\begin{align*}
\mathsf{W}_2\left(\Phi_{\text{fin}}^T (\mu_0^i),\mu_1^i\right) &= \mathsf{W}_2\left(\Phi_{\text{fin}}^T (\mu_0^i), ((\Phi_{\theta_3}^T)^{-1}\circ\Phi_{\theta_3}^T) (\mu_1^i)\right)\\
    &\overset{\eqref{eq: continuity.inverse.flow}}{\leq}  C\cdot\mathsf{W}_2\left((\Phi_{\theta_2}^{\frac{2T}{3}}\circ \Phi_{\theta_1}^{\frac{T}{3}}) (\mu_0^i),\Phi_{\theta_3}^T (\mu_1^i)\right)\overset{\eqref{eq: gluing.flow}}{\leq}\varepsilon,
\end{align*}
for all $i\in\llbracket1,N\rrbracket$, as desired.
\end{proof}

We now provide a different proof under the assumption that the input measures are absolutely continuous, and the targets are empirical measures with $M$ atoms. The advantage of this proof is that it provides an explicit estimate on the number of parameter switches.

\begin{proof}[Proof of \Cref{thm: main.result} (restricted case)] \label{proof: proof.restricted}
We assume that the target measures $\mu_1^i$ are all empirical measures with $M \geq 2$ atoms: $\mu_1^i = \frac{1}{M}\sum_{m=1}^M \updelta_{y_m^i},$
for some $y_m^i \in \S$. The input measures $\mu_0^i$ are assumed to be absolutely continuous with respect to the normalized Lebesgue measure, in addition to satisfying \eqref{eq: assumption.hole}. Under these assumptions, the following proof
is very similar to that of \Cref{thm: targets.atoms}---it avoids the packing step of \Cref{prop: compression}, and avoids a direct application of \Cref{lem: univ.approx}, steps where the number of switches are hard to track. We split the proof in three steps.

\subsubsection*{Step 1. Disentanglement}
As before, we first disentangle the measures using \Cref{prop: separation}. Furthermore since the vector field in \eqref{eq: cauchy.pb} is Lipschitz, absolute continuity of all measures is preserved over time, and thus we find flow maps $\Phi_{\theta_1}^{t}:\mathscr{P}_{\mathrm{ac}}(\S)\mapsto \mathscr{P}_{\mathrm{ac}}(\S)$
for $t\in[0,T/5]$, and $\Phi_{\theta_5}^{t}:\mathscr{P}(\S)\mapsto\mathscr{P}(\S)$
for $t\in[4T/5, T]$, induced by the characteristics of \eqref{eq: cauchy.pb} and piecewise constant parameters having $O(d\cdot N)$ switches, which satisfy
\begin{equation*}
   \conv_g\,\supp\,\Phi_{\theta_1}^{\frac{T}{5}}(\mu_0^i)\cap \conv_g\,\supp\,\Phi_{\theta_1}^{\frac{T}{5}}(\mu_0^j)=\varnothing\hspace{1cm}\text{ if }i\neq j,
\end{equation*}
and 
\begin{equation*}
    \conv_g\,\supp\,\Phi_{\theta_5}^{T}(\mu_1^i)\cap \conv_g\,\supp\,\Phi_{\theta_5}^{T}(\mu_1^j)=\varnothing\hspace{1cm}\text{ if }i\neq j.
\end{equation*}
We label the disentangled targets as
\begin{equation} \label{eq: relabelling.target}
    \Phi_{\theta_5}^{T}(\mu_1^i)=\frac{1}{M}\sum_{m=1}^M\updelta_{\widetilde{y}_m^i}.
\end{equation}

\subsubsection*{Step 2. Clustering}
Let $\varepsilon_1>0$ to be chosen later on.
We first employ \Cref{prop: targets.atoms} to cluster the disentangled input measures: there exists a flow map
$\Phi^t_{\theta_2}:\mathscr{P}_{\text{ac}}(\S)\mapsto\mathscr{P}_{\text{ac}}(\S)$
for $t\in[T/5, 2T/5]$, which satisfies 
    \begin{equation}\label{eq: smallness.two}
        \mathrm{diam}(\supp\,(\Phi^{\frac{2T}{5}}_{\theta_2}\circ \Phi^{\frac{T}{5}}_{\theta_2})(\mu^i_0))\leq \varepsilon_1
    \end{equation}
    for all $i\in\llbracket1,N\rrbracket$. Instead of using \Cref{lem: univ.approx} to approximate arbitrary transport maps as done in Step 2 in the previous proof, we rather use \Cref{lem: two.balls} recursively to reduce the problem to an ensemble matching of points. As a consequence of Step 1 and \eqref{eq: smallness.two}, there exists some $\kappa>0$ such that
    \begin{align}\label{eq: min.dist}
    \inf_{\substack{
    x\in\conv_g\,\supp\,\nu^i,\\y\in\conv_g\,\supp\,\nu^j,\\\,i\neq j
    }}d_g(x,y)\geq 2\kappa,
    \end{align}
    where we set $\nu^i\coloneqq (\Phi^{2T/5}_{\theta_2}\circ \Phi^{T/5}_{\theta_2})(\mu^i_0).$
    We use the following.

\begin{claim}\label{cl: balls}
  There exists some small enough $\varepsilon_1>0$ such that for all $i\in\llbracket1,N\rrbracket$, the measures $\nu^i$ are such that there exist balls $B(x_m^i,r^i)$, for $m\in\llbracket1,M\rrbracket$, satisfying
  \begin{enumerate}
      \item\begin{align*}
      &\nu^i\left(B(x_m^i,r^i)\setminus B(x_{m-1}^i,r^i)\right)=\frac{1}{M}\hspace{1cm}\text{ for }m\in\llbracket2,M\rrbracket,\\
     &\nu^i\left(B(x_1^i,r^i)\right)=\frac{1}{M}.\nonumber
  \end{align*}
      \item For any $m\in\llbracket1,M-1\rrbracket$ there exists $z_m^i\in B(x_m^i,r^i)$ such that 
    \begin{equation*}
         z_m^i\notin B(x_{m'}^i,r^i)\hspace{1cm}\text{ for }m'\geq m+1.
    \end{equation*}
    \item 
    For $j\neq i$,
        \begin{equation}\label{eq: measure.zero}
         \nu^i\left(B(x_m^j,r^j)\right)=0 \hspace{1cm} \text{ for all } m\in\llbracket1,M\rrbracket.
     \end{equation}
  \end{enumerate}       
\end{claim}

We postpone the proof of \Cref{cl: balls} to after the present one.
Fix $i\in\llbracket1,N\rrbracket$. Applying \Cref{lem: two.balls} $M$ times successively using the balls stemming from \Cref{cl: balls} and $z_m^i$ in place of $\omega$, we obtain $M$ Lipschitz-continuous invertible flow maps $\upvarphi_m^i:\S\to\S$ of \eqref{eq: neural.ode.sphere} corresponding to constant parameters, such that $\uppsi^i\coloneqq\upvarphi_M^i\circ\upvarphi_{M-1}^i\circ\cdots \circ\upvarphi_1^i,$
  because of \eqref{eq: measure.zero}, satisfies
     \begin{equation}\label{eq: identity.map}
         \uppsi^i_{\#}\nu^j=\nu^j \hspace{1cm}\text{ for }j\neq i,
     \end{equation}
     as well as
     \begin{equation}\label{eq: multi.cluster}
\mathsf{W}_2\left(\uppsi^i_{\#}\nu^i,\alpha^i\right)\leq \varepsilon
     \end{equation}    
     where $
         \alpha^i=\frac{1}{M}\sum_{m} \updelta_{z_m^i}.$
     Due to \eqref{eq: identity.map} and \eqref{eq: multi.cluster}, the map
$\uppsi\coloneqq\uppsi^N\circ\uppsi^{N-1}\circ\cdots\circ\uppsi^1$
     is a flow map of \eqref{eq: neural.pde.sphere} induced by parameters having $O(M\cdot N)$ switches, and satisfying $\mathsf{W}_2(\uppsi_{\#}\nu^i,\alpha^i)\leq \varepsilon$ for all $i\in\llbracket1,N\rrbracket$. All in all, the flow map
     \begin{align*}
        \Phi^{\frac{3T}{5}}_{\theta_3}:\mathscr{P}(\S)&\mapsto\mathscr{P}(\S)\\
         \Phi^{\frac{3T}{5}}_{\theta_3}(\mu)&=\uppsi_\#\mu
     \end{align*}
     is such that
     \begin{equation}\label{eq: third.flow}
         \mathsf{W}_2\left((\Phi^{\frac{3T}{5}}_{\theta_3}\circ \Phi^{\frac{2T}{5}}_{\theta_3}\circ\Phi^{\frac{T}{5}}_{\theta_3})\mu_0^i,\alpha^i\right)\leq \varepsilon
     \end{equation}
     for all $i\in\llbracket1,N\rrbracket$.

     \subsubsection*{Step 3. Matching} 
     We apply\footnote{Should the assumptions of \Cref{prop: interpolation.neural.ode} not hold, we consider a slight perturbation of the target measures ($\mathsf{W}_2(\mu_1^i,\tilde{\mu}_1^i)\leq \varepsilon$).} \Cref{prop: interpolation.neural.ode} to
     \begin{equation} \tag{$\mathscr{D}$}
     (z_m^i, \widetilde{y}_m^i)\in\S\times\S \hspace{1cm} \text{ for } (i, m)\in\llbracket1,N\rrbracket\times\llbracket1,M\rrbracket,
     \end{equation}
     with $\widetilde{y}_m^i$ as in \eqref{eq: relabelling.target}. This yields a flow map $\upphi:\S\to\S$ of \eqref{eq: neural.ode.sphere} induced by piecewise constant parameters with $O(M\cdot N)$ switches satisfying $\upphi(z_m^i)=\widetilde{y}_m^i$
     for all $(i, m)$. 
     Define
    \begin{equation*}
         \Phi^{\frac{4T}{5}}_{\theta_4}(\mu)\coloneqq\upphi_\#\mu.
     \end{equation*}
     Using the triangle inequality, the definition of $\alpha^i$ in Step 2, and the continuity of the solution to \eqref{eq: neural.ode.sphere} with respect the initial conditions and \eqref{eq: third.flow}, we find 
     \begin{equation*}
\mathsf{W}_2\left((\upphi\circ\uppsi)_{\#}\nu^i,\Phi_{\theta_5}^T(\mu_1^i)\right)\lesssim_{M, N}\varepsilon
     \end{equation*}
     for all $i$ where the implicit constant is independent of $\varepsilon$. This yields
               \begin{equation*} 
         \mathsf{W}_2\left((\Phi_{\theta_4}^{\frac{4T}{5}}\circ\Phi^{\frac{3T}{5}}_{\theta_3}\circ \Phi^{\frac{2T}{5}}_{\theta_2}\circ\Phi^{\frac{T}{5}}_{\theta_1})(\mu_0^i),\Phi_{\theta_5}^{T}(\mu_1^i)\right)\lesssim_{N, M}\varepsilon
     \end{equation*}
     for all $i$.
     The conclusion follows by applying the inverse of $\Phi^T_{\theta_5}$ as in the previous proof. Pasting the parameters used in all of the steps above, the resulting number of switches is $O((d+M)N)$. 
\end{proof}

\begin{proof}[Proof of \Cref{cl: balls}]
Fix $i\in\llbracket1,N\rrbracket$. Due to \eqref{eq: smallness.two}, we have
    \begin{equation}\label{eq: boundary.ball}
        \supp\,\nu^i\subset B(x_{M}^i,C\varepsilon_1)
    \end{equation}
    for some $x_{M}^i\in \conv_g\,\supp\,\nu^i$ and $C>0$. Take $\varepsilon_1$ small enough so that
    \begin{equation}\label{eq: bound}
        \kappa\geq 4C\varepsilon_1.
    \end{equation}
    Take 
    $x_1^i\in \partial B(x_M^i,\kappa/2).$
    Consider the minimizing geodesic $\gamma:[0,1]\to \S$ between $x_1^i$ and $x_{M}^i$, and the function
    \begin{align*} 
        f:[0,1]\times [0,\pi]&\mapsto [0,1]\nonumber\\
        (s,r)&\mapsto f(s,r)=\nu^i\left(B(\gamma(s),r)\right).
    \end{align*}
    Since $\nu^i$ is absolutely continuous, we have:  1). $f\in \mathscr{C}^0([0,1]\times [0,\pi];[0,1])$: 2). $f(1,r)=1$ for all $r\geq C\varepsilon_1$; 3). $f(0,r)=0$ for all $r\leq \kappa/2 -C\varepsilon_1$: 4). $f(0,r)=1$ for all $r\geq 1/2+C\varepsilon_1.$
    By continuity, there exists $r^i\in (\kappa/2-C\varepsilon_1,\kappa/2+C\varepsilon_1)$ such that $f(0,r^i)=1/M.$ Furthermore, due to \eqref{eq: boundary.ball} and \eqref{eq: min.dist} we also have
    \begin{equation*}
        \nu^j(B(\gamma(s),r^i))=0\hspace{1cm}\text{ for }s\in[0,1],\, j\neq i\in\llbracket1,N\rrbracket.
    \end{equation*}
    Finally, $f(\cdot,r^i)$ is continuous and monotonically increasing provided $r^i\geq C\varepsilon_1$, which is guaranteed by \eqref{eq: bound}. So pick $\{s_m^i\}_{m\in\llbracket2,M-1\rrbracket}\subset(0,1)$ so that $f(s_m^i,r^i)=m/M.$
    Hence, the desired balls are
    \begin{equation*}
        B(x_m^i=\gamma(s_m^i),r^i),\hspace{1cm} \text{ with } (s_1^i, s_M^i)=(0, 1).
    \end{equation*}
    Since for fixed $i$ all balls have the same radius, the existence of $z_m^i$ is immediate.
    \end{proof}

\subsection{\texorpdfstring{Proof of \Cref{thm: targets.atoms}}{Proof of Theorem}}

\begin{proof}[Proof of \Cref{thm: targets.atoms}]

The proof follows the same ideas as that of \Cref{thm: main.result}, but is significantly simpler since some steps can be omitted completely. Indeed, we can consider $\Phi^T_{\text{fin}}\coloneqq\Phi^{T}_{\theta_3}\circ\Phi^{\frac{2T}{3}}_{\theta_2}\circ\Phi^{\frac{T}{3}}_{\theta_1},$
where 
\begin{itemize}
    \item $\Phi^t_{\theta_1}:\mathscr{P}(\S)\to\mathscr{P}(\S)$ for $t\in[0,T/3]$ is provided by \Cref{prop: separation};
    \item $\Phi^{t}_{\theta_2}:\mathscr{P}(\S)\to\mathscr{P}(\S)$ for $t\in[T/3, 2T/3]$ is provided by \Cref{prop: targets.atoms}, which can be applied since $\Phi^{\frac{T}{3}}(\mu_0^i)$ are pairwise disjoint and supported in a single hemisphere for all $i$;
    \item $\Phi^t_{\theta_3}$ for $t\in[2T/3,T]$ is provided by \Cref{prop: interpolation.neural.ode}.
\end{itemize}
To conclude, we derive the bound on the norm of $\theta$. 

\begin{enumerate}
    \item In the proof of \Cref{prop: separation}, one sees that $
        \|\theta_1\|_{L^\infty((0,T);\Uptheta)}\lesssim dN/T$
    where the implicit constant depends only on the supports of the initial measures.
    
    \item Once the measures are disentangled, we further cluster them before using \Cref{prop: interpolation.neural.ode}. 
    By \Cref{prop: targets.atoms} we deduce   $\mathsf{W}_2(\mu^i(T_\delta),\updelta_{x_0^i})\leq \delta$ for every $i$ with $T_\delta=O(\log1/\delta)$, which implies $\|\theta_2\|_{L^\infty((0,T);\Uptheta)}\lesssim \log1/\delta.$

\item Finally, we apply \Cref{prop: interpolation.neural.ode} for the ensemble of atoms $x_0^i$: since all measures are $\delta$-close to $\updelta_{x_0^i}$, we have
$\mathsf{W}_2(\mu^i(T),\updelta_{y^i})\leq e^{O(N\cdot T)}\delta$ at cost $\|\theta_3\|_{L^\infty((0,T);\Uptheta)}\lesssim N/T$.
\end{enumerate}
All in all, $\mathsf{W}_2(\mu^i(T),\updelta_{y^i})\leq \varepsilon$ with  $\|\theta\|_{L^\infty((0,T);\Uptheta)}=O\left(d\cdot N/T+\log1/\varepsilon\right)$.
\end{proof}

\section{Complexity of disentanglement} \label{sec: complexity.untanglement}

As shown above, a Transformer can disentangle the supports of $N$ probability measures on $\S$ using piecewise-constant parameters with $O(d\cdot N)$ switches. The linear dependence on $N$ arises from separating one measure at a time; we believe this is suboptimal. A sharper understanding of the $\omega$--limit set of \eqref{eq: cauchy.pb}--\eqref{eq: vf} for $\bV\equiv I_d$, $\bW\equiv0$ and constant $\bB$ could reduce this to $O(1)$ switches. Indeed, if for given data $\{\mu_0^i\}_{i\in\llbracket1,N\rrbracket}$ there exists a constant $\bB$ and distinct points $z_1,\dots,z_N\in\S$ such that $\mathsf{W}_\infty(\mu^i(T),\updelta_{z_i})\to0$ as $T\to+\infty,$  then a single constant parameter suffices to disentangle all measures. Characterizing $z_i$ from $\mu_0^i$ in general appears difficult.

\begin{example} \label{ex: symmetry}
Let $d=2$ and $\mu_0^i(\diff x)=|\mathscr{B}_i|^{-1}\mathbf{1}_{\mathscr{B}_i}\,\diff x$, with $\mathscr{B}_i\subsetneq\mathbb{S}^1$ connected, and distinct barycenters, with directions $u_i\in\mathbb{S}^1$ (unit vector pointing to the midpoint angle $\theta_i$).
With $(\bV,\bB,\bW)\equiv(I_d,0,0)$ attention reduces to 
$m_i(t)\coloneqq\int x\mu_t^i(\diff x)$ and the field is
$\mathsf v[\mu^i(t)](x)=\proj_x m_i(t)$.
Let $\mathsf{R}_i$ be the reflection across the axis $\mathbb R u_i$. Since $\mu_0^i$ is uniform on the symmetric arc $\mathscr B_i$,
$\mathsf{R}_{i\#}\mu_0^i=\mu_0^i$, hence $m_i(0)\parallel u_i$. The vector field is $\mathsf{R}_i$-equivariant, so $\mathsf{R}_{i\#}\mu_t^i=\mu_t^i$ for all $t\geq0$, hence
$m_i(t)=\alpha_i(t)u_i$ with $\alpha_i(t)\ge0$. 
Write $x(\theta)=(\cos\theta,\sin\theta)$ and parametrize angles by $\vartheta\coloneqq\theta-\theta_i$.
For any particle following a characteristic of $\mu^i(t)$, we have
$\dot x(t)=\proj_{x(t)}(m_i(t))=\alpha_i(t)\proj_{x(t)}(u_i)$ as well as $\dot\vartheta(t)=-\alpha_i(t)\sin\vartheta(t)$.
The Lyapunov function $\mathsf{E}(\vartheta)\coloneqq1-\cos\vartheta$ satisfies $\dot{\mathsf{E}}=\sin\vartheta\dot\vartheta=-\alpha_i(t)\sin^2\vartheta\le0$, with equality only at $\vartheta=0$ (since $\alpha_i(t)>0$ as the arc is connected and strictly within the circle).
Thus $\vartheta(t)\to 0$ and each trajectory converges to $u_i$; hence $\mu^i(t)\rightharpoonup\delta_{u_i}$.
Since the $u_i$ are distinct, the limits are distinct, and disentanglement holds with the single constant parameter choice above.
\end{example}

Beyond \Cref{ex: symmetry}, we obtain $O(1)$ switches for generic discrete inputs.

\begin{proposition}\label{prop:generic-discrete}
Let $d\ge3$, $\pi=\mathrm{Unif}((\S)^n)$, and sample $\mu_0^1,\dots,\mu_0^N$ i.i.d.\ from
\[
\mathscr{S}_n\coloneqq\left\{\frac1n\sum_{k=1}^n\updelta_{x_k}:x_k\in\S\right\}\simeq (\S)^n.
\]
Consider \eqref{eq: cauchy.pb}--\eqref{eq: vf} with $\bB\equiv\beta I_d$ ($\beta\ge0$), $\bV\equiv I_d$, $\bW\equiv0$. Let $x_*^i$ be the (a.s.\ existing) limit cluster point of $\mu^i(t)$. Then $\mathbb{P}\big[x_*^i\neq x_*^j\big]=1$ for $i\neq j$.
\end{proposition}

\begin{proof}[Sketch of proof]
By \cite[Sec.~6]{geshkovski2023mathematical}, for $\pi$-a.e. $\mu\in\mathscr S_n$ the solution to \eqref{eq: cauchy.pb}--\eqref{eq: vf} converges to a Dirac $\delta_{X}$ with $X\in\S$. Define $f:\mathscr S_n\to\S$ as $f(\mu)=X$ with $\mu(t)\to\delta_X$.
The flow is deterministic and rotation-invariant, and $\pi$ is rotation-invariant, hence $f(\mathsf{R}_\#\mu)=\mathsf{R}f(\mu)$ and $\mathsf{R}_\#\pi=\pi$ for all $\mathsf{R}\in\mathbb{O}(d)$.
Let $\nu\coloneqq f_\#\pi$ be the law of $X$. If $\nu$ had an atom at $x_0$ of mass $p>0$, then by equivariance $\mathsf{R} x_0$ would also be an atom of mass $p$ for every rotation $\mathsf{R}$. Picking $M$ rotations with distinct points yields $\nu(\S)\ge Mp$, a contradiction for large $M$. Thus $\nu$ is non-atomic.

Now take two independent initial conditions $\mu_0^i,\mu_0^j\sim\pi$, and let $X_i=f(\mu_0^i)$, $X_j=f(\mu_0^j)$. Since the flow is deterministic, $X_i,X_j$ are independent with common law $\nu$. Non-atomicity gives
\[
\mathbb{P}[X_i=X_j]=\int\nu(\{x\})\,\mathrm d\nu(x)=0,
\]
so $\mathbb{P}[x_*^i\neq x_*^j]=1$.
\end{proof}

The main bottleneck in switch complexity arises in \Cref{prop: compression}, which relies solely on the perceptron part of the vector field; adding width could parallelize this step. It is plausible that nonlinear self-attention effects further reduce switches, e.g.\ via \emph{dynamic metastability} \cite{geshkovski2024dynamic, bruno2024emergence}. Such results may close the gap between the $O(d\cdot N)$ bound and the $O(1)$ one suggested by the $\omega$–limit set heuristic above.

\appendix

\section{\texorpdfstring{On condition \eqref{eq: assumption.hole}}{Proof of Lemma}}
\label{appendix: technical}

\begin{lemma} \label{lem: mass.concentration.Q1}
Let $\nu_0^i\in \mathscr{P}(\S)$, $i\in\llbracket1,N\rrbracket$ and $\eta>0$. 
Suppose there exists a Lipschitz-continuous and invertible $\upphi_\eta:\S\to\S$ such that
\begin{equation}\label{eq: mass.first.quadrant}
    \upphi_{\eta\#} \nu_0^i(\Q_1^{d-1})=1-\eta
\end{equation}
for all $i$. Then for all $i$ there exists $\mu_0^i\in\mathscr{P}(\Q_1^{d-1})$ and a universal constant $C>0$ such that
$\mathsf{W}_2(\upphi_{\eta\#}\nu_0^i, \mu_0^i)\leq C\eta.$
\end{lemma}

\begin{proof}
   Consider $\mu_0^i(A)=\nu_0^i(A\cap \Q_1^{d-1})+\nu_0^i(\S\setminus\Q_1^{d-1})\updelta_{x_0^i}(A)$ for $x_0^i\in\Q_1^{d-1}$ and Borel $A\subset\S$. 
\end{proof}

\Cref{lem: first.quadrant} provides $\upphi_\eta$ satisfying \eqref{eq: mass.first.quadrant}. 
With \Cref{lem: mass.concentration.Q1} we can extend \Cref{thm: main.result} to the setting of measures whose support fill $\S$---namely, the assumption $\omega\notin\bigcup_{i}\supp\,\mu_0^i$ can be removed. 
The result then follows by a continuity argument: apply \Cref{thm: main.result} to the measures $\mu_0^i$ given by \Cref{lem: mass.concentration.Q1} to get
$\mathsf{W}_2(\mu^i(T),\nu^i(T))\lesssim_T\mathsf{W}_2(\mu_0^i,\upphi_\eta\nu_0^i)\lesssim_T \eta,$ where $\mu^i(t)$ is the solution to \eqref{eq: cauchy.pb} given by \Cref{thm: main.result} with data $\mu_0^i$, and $\nu(t)$ is the solution to \eqref{eq: cauchy.pb} with data $(\upphi_\eta)_\#\nu_0^i$.
On the other hand, we can simply approximate the targets $\mu_1^i$ by measures that directly satisfy \eqref{eq: assumption.hole}.

\section{Technical proofs}

\subsection{Transporting mass through overlapping balls} 
\label{proof:tubular.mass.movement}

\begin{lemma} \label{lem: tubular.mass.movement}
Consider $K+1$ open balls $\mathscr{B}_K,\ldots,\mathscr{B}_1,\mathscr{B}_0\subset \S$ satisfying
\begin{align*}
        \mathscr{B}_k\cap\mathscr{B}_{k-1}\neq\varnothing \hspace{1cm} &\text{ for } k\in\llbracket1,K\rrbracket\\
        \mathscr{B}_k\cap\mathscr{B}_{k'}=\varnothing \hspace{1cm} &\text{ if }|k-k'|\geq 2.
\end{align*}
Then for any $T>0$ and $\varepsilon>0$, there exist $(\bW,\bU,b):[0,T]\to\mathscr{M}_{d\times d}(\mathbb{R})^2\times\R^d$,  piecewise constant having at most $K$ switches, such that for any $\mu_0\in \mathscr{P}(\S)$, the corresponding unique solution $\mu$ to 
\begin{equation} \label{eq: neural.pde.sphere}
\begin{cases}
\partial_t\mu(t)+\dive(\proj_x\bW(t)(\bU(t) x+b(t))_+\mu(t))=0 &\text{ on } [0, T]\times\S\\
\mu(0)=\mu_0 &\text{ on } \S
\end{cases}
\end{equation}
satisfies
    \begin{equation*}
        \mu(T,\mathscr{B}_K)\geq (1-\varepsilon)^K\mu_0\left(\bigcup_{k}\mathscr{B}_k\right).
    \end{equation*}
    Moreover, $\mu(T)=\upphi^T_\#\mu_0$ for a Lipschitz-continuous, invertible map $\upphi^t:\S\to\S$ which satisfies, for all $t\in[0,T]$,
    \begin{equation}\label{eq: eq.tubular.mass.movement}
        \upphi^t(x)=x\hspace{1cm}\text{ for }x\notin \bigcup_{k} \mathscr{B}_k.
    \end{equation}
\end{lemma}

We now focus on proving \Cref{lem: tubular.mass.movement}, itself relying on the following lemma.

\begin{lemma} \label{lem: two.balls}
    Consider two open balls $\mathscr{B}_0,\mathscr{B}_1\subset \S$ such that $\mathscr{B}_0\cap\mathscr{B}_1\neq\varnothing$.
    For any $\varepsilon>0$ and $T>0$, there exist $\bW,\bU\in\mathscr{M}_{d\times d}(\R)$ and $b\in\mathbb{R}^d$ such that for any $\mu_0\in\mathscr{P}(\S)$, the unique solution $\mu$ to \eqref{eq: neural.pde.sphere} satisfies
    \begin{equation} \label{eq: claim.statement.lemmaA1}
        \mu(T, \mathscr{B}_0\cap\mathscr{B}_1)\geq(1-\varepsilon)\mu_0(\mathscr{B}_0).
    \end{equation}
    Moreover $\mu(T)=\upphi^T_\#\mu_0$ where the Lipschitz-continuous and invertible flow map $\Phi^t:\S\to\S$ of \eqref{eq: neural.ode.sphere} satisfies, for all $t\in[0,T]$,
    \begin{equation*}
    (\upphi^t)_{\mid\S\setminus \mathscr{B}_0}\equiv\mathrm{Id}.    
    \end{equation*}
    Furthermore, for any fixed $\omega\in \mathrm{int}\,\mathscr{B}_0$ we can choose $\bW,\bU$ and $b$ so that the solution to \eqref{eq: neural.pde.sphere} satisfies $\mathsf{W}_2(\mu(T),\alpha)\leq \varepsilon$
    where $\alpha(A)=\mu_0(\mathscr{B}_0)\updelta_{\omega}(A)+\mu_0(A\setminus\mathscr{B}_0),$
    for any Borel $A\subset \S$.
\end{lemma}
    
\begin{proof}[Proof of \Cref{lem: two.balls}]
As done in previous proofs, we can take all time horizons to be as large as desired throughout by rescaling the norm of the parameters.
Let $z\in\S$ denote the center and $R>0$ the radius of $\mathscr{B}_0$. Take an arbitrary $\omega\in\mathrm{int}(\mathscr{B}_0\cap\mathscr{B}_1)$. We consider
$\bU=-{\bf1} z^\top$ and $b=\cos(R){\bf1}$,
as well as any $\bW\in \mathscr{M}_{d\times d}(\mathbb{R})$ such that $\bW{\bf1}=\omega$.
Then $\bW(\bU x+b)_+=(-\cos d_g(z,x)+\cos(R))_+\omega,$
and note that
\begin{equation}\label{eq: activation.ball}
    (-\cos d_g(z,x)+\cos(R))_+> 0\hspace{1cm}\iff \hspace{1cm} x\in\mathscr{B}_0.
\end{equation}
Now observe that
\begin{equation}\label{eq: dissipation.sphere}
    \frac{\diff }{\diff t}\langle x(t),\omega\rangle=(-\cos d_g(z,x(t)) +\cos(R))_+(1-\langle x(t),\omega\rangle^2),
\end{equation}
which is positive whenever $x(t)\in\mathscr{B}_0\setminus\{\omega\}$. We claim that this implies the existence of a time $T_\varepsilon>0$ for which
\begin{equation} \label{eq: claim.lemmaA1} 
    \mu(T_\varepsilon,\mathscr{B}_0\cap\mathscr{B}_1)\geq(1-\varepsilon)\mu_0(\mathscr{B}_0).
\end{equation}
To prove this claim, let $\delta>0$ be fixed and to be determined later on. 
Because of \eqref{eq: dissipation.sphere}, there exists some $T_\delta>0$ such that
\begin{equation}\label{eq: claim.lemmaA11}
    \upphi^{T_\delta}(x)\in\mathscr{B}_0\cap\mathscr{B}_1\hspace{1cm}\text{ for }x\in B(z,R-\delta),
\end{equation}
where $\upphi^{T_\delta}:\S\to \S$ is the flow map of \eqref{eq: neural.ode.sphere}. Since $\mu(T_\delta)=\upphi^{T_\delta}_\#\mu_0$, we have
\begin{equation} \label{eq: mass.concentration}
    \mu(T_\delta, \mathscr{B}_0\cap\mathscr{B}_1)=\mu_0((\upphi^{T_\delta})^{-1}(\mathscr{B}_0\cap\mathscr{B}_1))\overset{\eqref{eq: claim.lemmaA11}}{\geq}\mu_0(B(z,R-\delta)).
\end{equation}
Taking $\delta>0$ small enough so that $\mu_0(B(z,R-\delta))\geq (1-\varepsilon) \mu_0(\mathscr{B}_0)$ yields claim \eqref{eq: claim.lemmaA1}. We conclude that \eqref{eq: claim.statement.lemmaA1} holds by rescaling time. Finally, by virtue of \eqref{eq: activation.ball}, the flow map $\upphi^t$ is such that $\upphi^t(x)=x$ for
$x\in\S\setminus\mathscr{B}_0$ and $t\in[0,T]$.

As for the second part of the statement, take $\mathscr{B}_1=B(\omega,\eta)\subset\mathscr{B}_0$ with $\eta>0$ to be determined later on.
Owing to \eqref{eq: mass.concentration}, we can argue in the same fashion as in the proof of \Cref{prop: compression}. We have
\begin{align*}
\mathsf{W}_1(\mu(T_\delta),\alpha) &= \sup_{\mathrm{Lip}(g)\leq 1}\left|\int g(\mu(T_\delta)-\alpha)\right|\\
    &\hspace{0.2cm}= \sup_{\mathrm{Lip}(g)\leq 1}\left|\int_{\mathscr{B}_0}g(\mu(T_\delta)- \alpha)+\int_{\S\setminus \mathscr{B}_0}g(\mu(T_\delta)-\alpha)\right|.
\end{align*}
Let $\overline{\varepsilon}>0$ be arbitrary and to be chosen small enough later.
Using \eqref{eq: mass.concentration}---with $\overline{\varepsilon}$ instead of $\varepsilon$---and the definition of $\mathscr{B}_1$, we find  
\begin{align*}
\left|\int_{\mathscr{B}_0\setminus\mathscr{B}_1} g \left(\mu(T_\delta) -\alpha\right)+ \int_{\mathscr{B}_1} g \left(\mu(T_\delta) -\alpha\right)\right|&\leq \left|\int_{\mathscr{B}_0\setminus\mathscr{B}_1} g \left(\mu(T_\delta) -\alpha\right)\right|\\
      &\hspace{-6cm}+ \left|\int_{\mathscr{B}_1} g \mu(T_\delta) -  \mu(T_\delta,\mathscr{B}_1)g(\omega)-(\mu_0(\mathscr{B}_0)-\mu(T_\delta,\mathscr{B}_1))g(\omega)\right|\\ &\hspace{-6cm}\leq\|\nabla g\|_{L^\infty(\S)}\cdot\eta\cdot\overline{\varepsilon}\cdot\mu_0(\mathscr{B}_0)+\eta+\overline{\varepsilon}\cdot\mu_0(\mathscr{B}_0),
\end{align*}
which tends to $0$ as $\overline{\varepsilon}$ and $\eta$ tend to zero.
On the other hand,
\begin{align*}
\left|\int_{\S\setminus\mathscr{B}_0}g(\mu(T_\delta)-\alpha)\right|=0
\end{align*}
by construction.
Pick $\overline{\varepsilon}$ and $\eta$ small enough so that $ \mathsf{W}_1(\mu(T_\delta),\alpha)\leq \varepsilon$ to conclude.
\end{proof}

We finally provide the brief proof of \Cref{lem: tubular.mass.movement}:

\begin{proof}[Proof of \Cref{lem: tubular.mass.movement}]
Write 
$
[0,T)=\bigcup_{k\in\llbracket1,M\rrbracket} [t_{k-1},t_k)
$
where $t_k=kT/K$, and proceed by backward induction:
\begin{align*}
\mu(T,\mathscr{B}_K)&=\mu(T,\mathscr{B}_K\setminus\mathscr{B}_{K-1})+\mu(T,\mathscr{B}_K\cap \mathscr{B}_{K-1})\\
&\geq \mu(t_{K-1},\mathscr{B}_K\setminus\mathscr{B}_{K-1})+(1-\varepsilon)\mu(t_{K-1},\mathscr{B}_{K-1}),
\end{align*}
where the last inequality follows from \Cref{lem: two.balls}. Using $\mathscr{B}_k\cap\mathscr{B}_{k'}=\varnothing$ whenever $|k-k'|\geq 2$, we arrive to
\begin{equation*}
    \mu(T,\mathscr{B}_k)\geq  (1-\varepsilon)^{K}\left(\sum_{k=1}^K\mu_0(\mathscr{B}_k\setminus\mathscr{B}_{k-1})+\mu_0(\mathscr{B}_0)\right)=(1-\varepsilon)^{K}\mu_0\left(\bigcup_{k\in\llbracket0,K\rrbracket}\mathscr{B}_k\right),
\end{equation*}
whereupon the conclusion follows.
\end{proof}

\subsection{\texorpdfstring{Proof of \Cref{lem: induction.barycenter}}{Proof of Lemma}}
\label{proof: induction.barycenter}

\begin{proof}[Proof of \Cref{lem: induction.barycenter}]

The proof is split in three steps.

\subsubsection*{Step 1. Isolating $\mu_0^N$ and $\nu_0$}

Throughout this step, $\bW\equiv0$.
Let $0<T_0<\ldots<T_{d-1}$ to be chosen later on and
\begin{equation*}
    \bV(t)=\sum_{k=1}^{d-1}\alpha_k\alpha_k^\top 1_{[T_{k-1},T_k]}(t)
\end{equation*}
with $\{\alpha_k\}_{k\in\llbracket1,d-1\rrbracket}$ an orthonormal basis of $\mathrm{span }(\{\mathbb{E}_{\mu_0^N}[z]\})^\perp$, namely 
$\langle \mathbb{E}_{\mu_0^N}[x],\alpha_k\rangle=0$ 
for all $k\in\llbracket1,d-1\rrbracket$. We proceed recursively starting from $k=1$. The solution to
\begin{equation} \label{eq: cauchy.pb.reloaded}
\begin{cases}  \partial_t\mu(t)+\dive(\proj_x\langle\alpha_1,\mathbb{E}_{\mu(t)}[x]\rangle \alpha_1\mu(t))=0 &\text{ on } \mathbb{R}_{\geq0}\times\S,\\ 
\mu(0)=\mu_0 &\text{ on } \S
\end{cases}
\end{equation}
for $\mu_0\in\mathscr{P}(\mathbb{Q}_1^{d-1})$ satisfies
\begin{equation*}
    \frac{\diff}{\diff t}\langle\mathbb{E}_{\mu(t)}[x],\alpha_1\rangle =\langle\mathbb{E}_{\mu(t)}[x],\alpha_1 \rangle \left(1-\int\langle x',\alpha_1\rangle^2\mu(t, \diff x')\right).
\end{equation*}
This implies
\begin{equation*}
    \langle\mathbb{E}_{\mu(t)}[x],\alpha_1\rangle=\langle\mathbb{E}_{\mu_0}[x],\alpha_1\rangle\mathrm{exp}\left(t-\int_0^t\int \langle x',\alpha_1\rangle^2 \mu(s, \diff x') \diff s\right).
\end{equation*} 
So $\langle\mathbb{E}_{\mu(t)}[x],\alpha_1\rangle$ does not change sign along $t\mapsto\mu(t)$, and $\frac{\diff}{\diff t} \langle\mathbb{E}_{\mu(t)}[x],\alpha_1\rangle=0$ whenever $\mathbb{E}_{\mu_0}[x]$ is orthogonal to $\alpha_1$ or if $\mu(t)=\updelta_{\pm \alpha_1}$. Hence for any $x(t)\in\supp\,\mu(t)$,
\begin{equation*}
\frac{\diff}{\diff t} \langle x(t),\alpha_1\rangle =\langle \mathbb{E}_{\mu(t)}[x],\alpha_1 \rangle\left(1-\langle \alpha_1,x(t)\rangle^2\right)
\end{equation*}
which implies 
\begin{equation*} 
    \lim_{t\to+\infty}x(t)=\pm \alpha_1
\end{equation*}
whenever $\langle \mathbb{E}_{\mu_0}[x], \alpha_1\rangle \neq 0$. 
Therefore, for every $\varepsilon_1>0$ we can take $T_1>0$ large enough so that
\begin{equation*}
    \supp\,\mu(T_1)\subset B(\alpha_1,\varepsilon_1)\cup B(-\alpha_1,\varepsilon_1)
\end{equation*}
whenever $\langle \mathbb{E}_{\mu_0}[x],\alpha_1\rangle \neq 0$. We can repeat the argument for every $k$ to deduce
\begin{equation}\label{eq: supports.in.balls}
    \supp\,\mu(T_{d-1})\subset \bigcup_{k\in\llbracket1,d-1\rrbracket} B(\alpha_k,C_{k}\varepsilon_k)\cup B(-\alpha_k,C_{k}\varepsilon_k)
\end{equation}
where $C_k>0$  does not depend on $\varepsilon_k$, but does depend on  $\varepsilon_\ell$ for $\ell>k$. 
We can choose all radii $\varepsilon_k$ small enough so that
    \begin{equation}\label{eq: union.out}
        \bigcup_{k\in\llbracket1,d-1\rrbracket} B(\alpha_k,C_{k}\varepsilon_k)\cup B(-\alpha_k,C_{k}\varepsilon_k)\subset\S\setminus\Q_1^{d-1}.
    \end{equation}
We have thus constructed a map 
$\Psi_{1}:\mathscr{P}(\S)\to\mathscr{P}(\S),$
with $\Psi_1(\mu_0)=\mu(T_d)$, where $\mu$ denotes the solution to the Cauchy problem \eqref{eq: cauchy.pb.reloaded} with the choice of parameters specified at the very beginning.  Since $\mathbb{E}_{\mu_0^i}[x]$ is not colinear with $\mathbb{E}_{\mu_0^N}[x]$, and thanks to \eqref{eq: supports.in.balls} and \eqref{eq: union.out},
$\supp\,\Psi_1(\mu^j_0)\subset\S\setminus\Q_1^{d-1}$ for $j\in\llbracket 1,N-1\rrbracket$, as well as 
$\Psi_1(\mu^N_0)=\mu^N_0,$ and $ \Psi_1(\nu_0)=\nu_0.$

\subsubsection*{Step 2. Clustering the supports of $\mu_0^N$ and $\nu_0$}

Let $a\in\S$ and $\underline{b}\in\R$ be such that 
    \begin{align*}
    \langle a,x\rangle +\underline{b}&>0\hspace{1cm}\text{ for }x\in \Q_1^{d-1}\\
\langle a,x\rangle +\underline{b}&<0\hspace{1cm}\text{ for }x\in \bigcup_{k\in\llbracket1,d-1\rrbracket} B(\alpha_k,C_k\varepsilon_k)\cup B(-\alpha_k, C_k\varepsilon_k).
\end{align*}
For instance, this can be ensured by taking $\{\varepsilon_k\}_{k\in\llbracket1,d-1\rrbracket}$ small enough and setting 
\begin{equation*}
    a=\mathbb{E}_{\mu_0^N}[x]/{\|\mathbb{E}_{\mu_0^N}[x]\|} \hspace{1cm}\text{ and } \hspace{1cm} b=- \max_{k\in\llbracket1,d-1\rrbracket} C_k\varepsilon_k.
\end{equation*} 
Let $\delta>0$ be arbitrary; in the interval $(T_{d},T_\delta)$, for $T_\delta>0$ to be determined later on, consider 
\begin{align*}
(\bW(t),\bU(t),b(t))\equiv(\bW_2,\bU,\underline{b})1_{[T_d,T_\delta]}(t),
\end{align*}
where $\bU={\bf1}a^\top$, and $\bW_2$ is any $d\times d$ matrix such that
$\bW_2 {\bf1}= \mathbb{E}_{\mu_0^N}[x].$
For this choice of parameters, the measures $\mu^i(T_d)$, are invariant by the action of the corresponding flow map of \eqref{eq: cauchy.pb} for $i\in\llbracket1,N-1\rrbracket$. We can pick $T_\delta>0$ large enough so that
\begin{equation} \label{eq: inclusion}
\supp\,\nu(T_\delta)\cup\supp\,\mu^N(T_\delta)\subset B\left(  \frac{\mathbb{E}_{\mu_0^N}[x]}{\|\mathbb{E}_{\mu_0^N}[x]\|},\delta\right).
\end{equation}
This follows by observing that 
\begin{equation*}
\lim_{t\to+\infty}\left\langle x(t),\frac{\mathbb{E}_{\mu_0^N}[x]}{\|\mathbb{E}_{\mu_0^N}[x]\|}\right\rangle =1    
\end{equation*}
for every $x_0\in\supp\,\mu_0^N$, where $x(t)$ follows the characteristics of \eqref{eq: cauchy.pb}, by adapting the same arguments as for \eqref{eq: dissipation.sphere} in the proof of \Cref{lem: two.balls}, or \eqref{eq: Hartman.Grobman} in the proof of \Cref{lem: induction.neural.ode}.
This construction yields a flow map 
$\Psi_2:\mathscr{P}(\S)\to\mathscr{P}(\S),$ 
with 
$\Psi_2(\mu_0)=\mu(T_\delta)$
where $\mu$ denotes the solution to \eqref{eq: cauchy.pb} on $[T_d, T_\delta]$ with the parameters specified in this step, which satisfies 
$\Psi_2(\mu^j(T_d))=\mu^j(T_d)$ for $j\in\llbracket1,N-1\rrbracket$,
and $\Psi_2(\mu^N(T_d)),\Psi_2(\nu(T_d))$
satisfy \eqref{eq: inclusion}.

\subsubsection*{Step 3. Flow reversal}

We finally employ $\Psi_{1}^{-1}$ and choose $\delta>0$ small enough to obtain the result----namely, setting $\Phi_{\text{fin}}\coloneqq\Psi_1^{-1}\circ\Psi_2\circ\Psi_1$, 
we have $\Phi_{\text{fin}}(\mu_0^i)=\mu_0^i$
for $i\in\llbracket1,N-1\rrbracket$, and
\begin{equation*}
    \supp\,\Phi_{\text{fin}}(\nu_0)\cup\supp\,\Phi_{\text{fin}}(\mu_0^N)\subset B\left( \frac{\mathbb{E}_{\mu_0^N}[x]}{\|\mathbb{E}_{\mu_0^N}[x]\|} , C_T\delta\right),
\end{equation*}
for $C_T>0$ depending on $\Psi_1$ but not on $\Psi_2$. Pick $\delta>0$ small enough to conclude.
\end{proof}

\subsection{\texorpdfstring{Proof of \Cref{lem: perturbation}}{Proof of Lemma}}
\label{proof: perturbation}

\begin{proof}[Proof of \Cref{lem: perturbation}]
We begin with the first part of the statement. 

\subsubsection*{Part 1.}
There exists an open ball $\mathscr{B}\subset \supp\,\mu_0\cup\supp\,\nu_0$ such that $
\mu_0(\mathscr{B})\neq\nu_0(\mathscr{B}).$
We now claim that there exists some $x^*\in\mathscr{B}$ such that 
\begin{equation*}
\mu_0(\mathscr{B})x^*+\int_{\S\setminus \mathscr{B}}x \mu_0(\diff x)\neq \nu_0(\mathscr{B})x^*+\int_{\S\setminus\mathscr{B}}x \nu_0(\diff x).
\end{equation*}
Indeed if this were to be false, then we'd have 
\begin{equation*}
x^*=\frac{1}{\mu_0(\mathscr{B})-\nu_0(\mathscr{B})}\int_{\S\setminus\mathscr{B}} x(\nu_0(\diff x)-\mu_0(\diff x))
\end{equation*}
for all $x^*\in\mathscr{B}$, which cannot hold. Take $x^*\in \mathscr{B}$ as above. 
Let $a$ be the center of $\mathscr{B}$ and $R$ its radius. 
Consider
\begin{align} \label{eq: controls.1}
\bU = -{\bf1}a^\top,\quad b= R {\bf1}, 
\end{align}
and any $\bW\in\mathscr{M}_{d\times d}(\R)$ satisfying
\begin{equation} \label{eq: controls.2}
\bW {\bf1}=x^*.
\end{equation}
By \Cref{lem: two.balls}, for any $\varepsilon>0$ we can take a large enough $T>0$ such that the solution to \eqref{eq: cauchy.pb}--\eqref{eq: vf} (with $\bV\equiv0$) satisfies $\mathsf{W}_2(\mu(T),\alpha)\leq\varepsilon$
with $\alpha(A)=\mu_0(\mathscr{B})\updelta_{x^*}(A\setminus\mathscr{B})+\mu_0(A\setminus\mathscr{B})$
for any Borel $A\subset\S$.
Since the expectation of a measure is continuous with respect to the measure in the sense of the Wasserstein distance, it follows that there is a Lipschitz invertible flow map $\upphi:\S\to\S$ of \eqref{eq: neural.ode.sphere} such that $\mathbb{E}_{\upphi_\#\mu_0}[x]\neq\mathbb{E}_{\upphi_\#\nu_0}[x]$. Furthermore, $\upphi(x)=x$ for $x\notin\mathscr{B}$ by construction.

\subsubsection*{Part 2.}
The parameters take the form
\begin{align*}
\bV(t)= I_d 1_{[0,T_*]}(t),\hspace{0.1cm} (\bW(t),\bU(t),b(t))=(\bW, \bU, b) 1_{[T_*,T]}(t),
\end{align*}
for $T_*>0$ and $T>T_*$ determined later on. (Recall that $\bB\equiv0$.) We first prove that if
\begin{equation}\label{eq: diff.support}
\supp\,\mu_0\neq\supp\,\nu_0
\end{equation}
is not satisfied, it ought to hold after some time. 
Indeed, suppose that \eqref{eq: diff.support} does not hold. 
Let $\tau>0$ be arbitrary.
For any $x_0\in \partial\conv_g\, \supp\,\mu_0\cap \supp\,\mu_0$ consider
\begin{equation*}
\begin{cases}
\dot{x}(t)=\mathbb{E}_{\mu(t)}[x]-\langle \mathbb{E}_{\mu(t)}[x],x(t)\rangle x(t)&\text{ in }  [0,\tau]\\
x(0)=x_0
\end{cases}
\end{equation*}
and
\begin{equation*}
\begin{cases}
\dot{y}(t)=\mathbb{E}_{\nu(t)}[x]-\langle \mathbb{E}_{\nu(t)}[x],y(t)\rangle y(t) &\text{ in }  [0,\tau]\\
y(0)=x_0.
\end{cases}
\end{equation*}
Taylor-expanding within the Duhamel formula, for $\tau$ small enough, we find 
\begin{equation*}
x(\tau)=x_0+\tau\left(\mathbb{E}_{\mu_0
}[x]-\left\langle\mathbb{E}_{\mu_0}[x],x_0\right\rangle x_0\right)+O(\tau^2)
\end{equation*}
and
\begin{equation*}
y(\tau)=x_0+\frac{\tau}{\gamma_1}\left(\mathbb{E}_{\mu_0
}[x]-\left\langle \mathbb{E}_{\mu_0}[x],x_0\right\rangle x_0\right)+O(\tau^2)
\end{equation*}
Then
\begin{equation*}
\left\langle y(\tau)-x(\tau),\frac{\mathbb{E}_{\mu_0}
[x]}{\|\mathbb{E}_{\mu_0
}[x]\|}\right\rangle=\tau\left(\frac{1}{\gamma_1}-1\right)\left(\|\mathbb{E}_{\mu_0}[x]\|-\frac{\langle \mathbb{E}_{\mu_0}[x],x_0\rangle^2}{\|\mathbb{E}_{\mu_0}[x]\|}\right)+O(\tau^2).
\end{equation*}
Suppose\footnote{If $\upsigma_d(\conv_g\,\supp\,\mu_0)=0$, we can argue as in the proof of \Cref{prop: targets.atoms}, reducing the dynamics to $\mathbb{S}^{d-2}$ (or a lower-dimensional sphere), where the same proof can be repeated.}  $\upsigma_d(\conv_g \,\supp\,\mu_0)>0$.
As $x_0\in\partial\conv_g\,\supp\,\mu_0$ and $\frac{\mathbb{E}_{\mu_0}[x]}{\|\mathbb{E}_{\mu_0}[x]\|}\in \mathrm{int }\,\mathrm{ conv}_g\,\supp\,\mu_0$,
\begin{equation*}
    \|\mathbb{E}_{\mu_0}[x]\|-\frac{\langle \mathbb{E}_{\mu_0}[x],x_0\rangle^2}{\|\mathbb{E}_{\mu_0}[x]\|}\geq c
\end{equation*}
for some $c>0$. Since $\gamma_1\in(0,1)$ we gather that
\begin{equation*}
\left\langle y(\tau)-x(\tau),\frac{\mathbb{E}_{\mu_0}
[x]}{\|\mathbb{E}_{\mu_0
}[x]\|}\right\rangle>c_1\tau+O(\tau^2)>0
\end{equation*}
for some $c_1>0$ and for $\tau$ small enough.
Consequently for $T_*$ small enough, we {have
$\supp\,\nu(T_*)\subset\supp\,\mu(T_*)$} as well as $\supp\,\mu(T_*)\neq \supp\,\nu(T_*)$. Therefore, there exist $\varepsilon>0$ and an open ball $\mathscr{B}$ such that
    \begin{equation}\label{eq: discriminatory.ball}
        \mathscr{B}\cap\supp\,\nu(T_*)\neq \varnothing,\quad \mathscr{B}\cap \supp\,\mu(T_*)=\varnothing
    \end{equation}
and
\begin{equation*}
    \mathscr{B}\subset \left\{ x\in \S: \quad \inf_{y\in \conv_g\, \supp\,\mu(T_*)}d_g(x,y)\leq \varepsilon\right\}.
\end{equation*}
Let $a$ be the center of $\mathscr{B}$ and $R$ its radius. In $(T_*,T)$, take $\bV\equiv0$, $\bW,\bU\in \mathscr{M}_{d\times d}(\mathbb{R})$ and $b\in \mathbb{R}^d$ as in \eqref{eq: controls.1}--\eqref{eq: controls.2}
    for some $x^*\in \mathscr{B}$ to be determined later on.  Because of \eqref{eq: discriminatory.ball}, $\nu$ is invariant by the action of the the flow map generated by the parameters defined in \eqref{eq: controls.1} and \eqref{eq: controls.2}. 
     We change the coordinate system so that
    \begin{equation*}
        \left(\int_{\S} x\nu(T_*)\right)_1=\alpha,\quad \left(\int_{\S} x\nu(T_*)\right)_k=0\quad \text{ for }k\geq 2.
    \end{equation*}
    Using the fact that $\mathscr{B}$ is open {and \eqref{eq: discriminatory.ball}}, it is impossible that for every $x^*\in \mathscr{B}$,
    \begin{equation*}
        \left(\int_{\S\setminus\mathscr{B}} x\,\mu(T_*)\right)_2+\mu(T_*, \mathscr{B})(x^*)_2=0.
    \end{equation*}
    Consequently there exist $x^*\in \mathscr{B}$ for which 
\begin{equation*}
    \int_{\S\setminus\mathscr{B}} x\,\mu(T_*)+\mu(T_*, \mathscr{B})x^* \hspace{1cm}\text{ and } \hspace{1cm}\int_{\S} x\nu(T_*)
\end{equation*}
are not colinear. Therefore, letting $T$ large enough, by the same arguments as in \Cref{lem: perturbation} {and since $\mathscr{B}\subset \conv_g\,\supp\,\mu_0\cup \conv_g\,\supp\,\nu_0$}, we can conclude. 

\end{proof}

\subsection{\texorpdfstring{Proof of \Cref{lem: hyp.propagation}}{Proof of Lemma}}
\label{proof: hyp.propagation}

\begin{proof}[Proof of \Cref{lem: hyp.propagation}]  
Since the vector field in \eqref{eq: cauchy.pb} is Lipschitz, for all $i\in\llbracket1,N\rrbracket$ there exist Lipschitz-continuous, invertible  $\mathsf{T}^i_{\Phi_1}:\S\to\S$ and $\mathsf{T}^i_{\Phi_3}:\S\to\S$ such that
\begin{equation*}
\Phi_1(\mu_0^i)=(\mathsf{T}^i_{\Phi_1})_\#\mu_0^i,\,\quad \text{ and }\,\Phi_3(\mu_1^i)=(\mathsf{T}^i_{\Phi_3})_\#\mu_1^i.
\end{equation*}
Then
\begin{equation}\label{eq: separated}
\supp\,(\mathsf{T}^i_{\Phi_1})_\#\mu_0^i\cap \supp\,(\mathsf{T}^j_{\Phi_1})_\#\mu_0^j=\varnothing,
\end{equation}
and
\begin{equation*}
\supp\,(\mathsf{T}^i_{\Phi_3})_\#\mu_1^i\cap \supp\,(\mathsf{T}^j_{\Phi_3})_\#\mu_1^j=\varnothing
\end{equation*}
for $i\neq j$. We wish to find an integrable $\uppsi^i:\S\to\S$ that satisfies
\begin{equation*}
(\uppsi^i\circ\mathsf{T}^i_{\Phi_1})_\#\mu_0^i= (\mathsf{T}^i_{\Phi_3})_\#\mu_1^i.
\end{equation*}
Since $\mathsf{T}^i_\#\mu_0^i=\mu_1^i$, and $\mathsf{T}^i_{\Phi_1},\mathsf{T}^i_{\Phi_1}$ are bijective, this is equivalent to $(\mathsf{T}^i_{\Phi_3})^{-1}\circ\uppsi^i\circ\mathsf{T}^i_{\Phi_1} =\mathsf{T}^i,$
so $\uppsi^i=\mathsf{T}^i_{\Phi_3}\circ\mathsf{T}^i\circ(\mathsf{T}^i_{\Phi_1})^{-1}.$
Thanks to \eqref{eq: separated}, there also exists a Lipschitz-continuous $\uppsi:\S\to \S$ satisfying \eqref{eq: diffeo.supports}.
\end{proof}

\subsection{\texorpdfstring{Proof of \Cref{lem: univ.approx}}{Proof of Lemma}}
\label{proof: univ.approx}

\begin{proof}[Proof of \Cref{lem: univ.approx}]
Consider
\begin{equation} \label{eq: simple.function}
\uppsi^\dagger_\varepsilon(x)\coloneqq\sum_{m=1}^{M(\varepsilon)} y_m^\varepsilon 1_{\Omega_m(\varepsilon)}(x),
\end{equation}
where $\Omega_m(\varepsilon)\subset\S$ are connected and pairwise disjoint with 
\begin{equation}\label{eq: partition}
    \bigcup_{m} \Omega_m(\varepsilon) = \S,
\end{equation}
whereas $y_m^\varepsilon\neq y_{m'}^\varepsilon$ when $m\neq m'$, and
\begin{equation}\label{eq: approx.by.simple.function}
    \left\|\uppsi^\dagger_\varepsilon-\uppsi\right\|_{L^2(\mu)}\leq \frac{\varepsilon}{2}.
\end{equation}
Our goal is to approximate $\uppsi^\dagger_\varepsilon$ by means of some flow map $\uppsi_\varepsilon:\S\to\S$ of \eqref{eq: neural.ode.sphere}. To this end, we also approximate $\mu$ as $|\mu(\S)-\mu^\eta(\S)|\leq\eta,$
with $\mu^\eta$ curated so we can apply \Cref{prop: interpolation.neural.ode} and \Cref{prop: compression} ``more easily''. Then,
\begin{align}
    \int \left\|\uppsi_\varepsilon(x)-\uppsi_\varepsilon^\dagger(x)\right\|^2\mu(\diff x)&=\int \left\|\uppsi_\varepsilon(x)-\sum_{m}y_m^\varepsilon 1_{\Omega_m}\right\|^2\mu^\eta(\diff x)\nonumber\\
    &+\int\left\|\uppsi_\varepsilon(x)-\uppsi^\dagger_\varepsilon(x)\right\|^2(\mu(\diff x)-\mu^\eta(\diff x))\nonumber\\
    &\hspace{-0.5cm}\leq \int \left\|\uppsi_\varepsilon(x)-\sum_{m}y_m^\varepsilon 1_{\Omega_m}\right\|^2\mu^\eta(\diff x)+2\pi\eta. 
    \label{eq: TV.implication}
\end{align}

\subsubsection*{Step 1: Constructing $\mu^\eta$}

Fix $\eta>0$. By the Lebesgue decomposition theorem, we split $\mu$ into purely atomic and diffusive parts:
$\mu=\mu_{\text{pp}}+\mu_{\text{diff}},$
with $\mu_{\text{diff}}$ having no atoms, and $
\mu_{\text{pp}}=\sum_{n=1}^{+\infty}\mu(\{x_n\})\updelta_{x_n}.$
Let $N(\eta)\geq1$ be such that $\mu_{\text{pp}}^\eta \coloneqq\sum_{n=1}^{N(\eta)}\mu(\{x_n\})\updelta_{x_n}$ satisfies $\mu_{\text{pp}}(A)-\mu_{\text{pp}}^\eta(A)\leq\eta/2$ for any Borel $A\subset\S$. 
Fix $\eta_1>0$ to be determined later on but such that for all $n\in\llbracket1,N(\eta)\rrbracket$,
\begin{equation} \label{eq: balls.small}
    B(x_n,\eta_1)\cap B(x_m,\eta_1)=\varnothing \hspace{1cm} \text{ for } m\neq n\in\llbracket1,N(\eta)\rrbracket.
\end{equation}
Consider
\begin{equation}\label{eq: Lebesgue.decomp}
    \mu^\eta\coloneqq\mu_{\text{pp}}^\eta+\mu_{\text{diff}}^\eta,
\end{equation}
where\footnote{If $\mu_{\text{pp}}=0$, consider an arbitrary  $x_1\in\S$ and then define $
    \mu_{\text{diff}}^{\eta}(A)\coloneqq\mu_{\text{diff}}\left(A\setminus B(x_1,\eta_1)\right).$}
\begin{equation} \label{eq: removing.mass.from.mu}
    \mu_{\text{diff}}^{\eta}(A)\coloneqq\mu_{\text{diff}}\left(A\setminus \bigcup_{n} B(x_n,\eta_1)\right)
\end{equation}
for any Borel $A\subset \S$. Furthermore, take $\eta_1>0$ small enough so that, in addition to \eqref{eq: balls.small},
$|\mu(\S)-\mu^\eta(\S)|\leq\eta.$

\subsubsection*{Step 2: Toward a sufficient matching problem}

We further decompose $\mu^\eta$ in several parts. For $m\in\llbracket1,M(\varepsilon)\rrbracket$, consider 
\begin{equation}\label{eq: mu.m}
    \mu_m(A)\coloneqq\mu^\eta_{\text{diff}}(A\cap\Omega_m)
\end{equation}
for any Borel $A\subset \S$. 
Because of \eqref{eq: removing.mass.from.mu}, \eqref{eq: mu.m} and \eqref{eq: partition}, we have
\begin{equation} \label{eq: mu.eta.mu.m}
    \mu^\eta_{\mathrm{diff}}(A) = \sum_{m} \mu_m(A)
\end{equation}
for any Borel $A\subset \S$. Therefore, thanks to \eqref{eq: Lebesgue.decomp} and \eqref{eq: mu.eta.mu.m}, bounding \eqref{eq: TV.implication} boils down to bounding
\begin{align}
    &\int \left\|\uppsi_\varepsilon(x)-\sum_{m}y_m^\varepsilon 1_{\Omega_m}\right\|^2\mu^\eta(\diff x)\nonumber\\
    &=\sum_{m}\int\left\|\uppsi_\varepsilon(x)-y_m^\varepsilon\right\|^2\mu_m+\sum_{n=1}^{N(\eta)}\mu(\{x_n\})\left\|\uppsi_\varepsilon(x_n)-\uppsi^\dagger_\varepsilon(x_n)\right\|^2.\label{eq: new.bound}
\end{align}
For the second term in \eqref{eq: new.bound} we will employ exact matching via \Cref{prop: interpolation.neural.ode}, whereas for the first, we first note that for any $\eta_3>0$, one has the  trivial identity
\begin{align}
    &\int\left\|\uppsi_\varepsilon(x)-y_m^\varepsilon\right\|^2\mu_m(\diff x)=\mu_m(\S)\Bigg(\int_{(\uppsi_\varepsilon)^{-1}(B(x_m,\eta_3))} \|\uppsi_\varepsilon(x)-y_m^\varepsilon\|^2\frac{\mu_m(\diff x)}{\mu_m(\S)}\nonumber\\&\hspace{4.2cm}+\int_{(\uppsi_\varepsilon)^{-1}(B(x_m,\eta_3))^c} \|\uppsi_\varepsilon(x)-y_m^\varepsilon\|^2\frac{\mu_m(\diff x)}{\mu_m(\S)}\Bigg)\label{eq: trivial.split}.
\end{align}
We use the following.

\begin{claim}\label{cl: W.to.ball}
Suppose $\mu\in\mathscr{P}(\S)$ and $x_0\in \S$  satisfy $
    \mathsf{W}_2(\mu,\updelta_{x_0})\leq\eta_2.$
There exists a universal constant $C>0$ such that $1-\mu(B(x_0,\eta_3))\leq C\eta_2/\eta_3$
for all $\eta_3>0$.
\end{claim}

\begin{proof}[Proof of \Cref{cl: W.to.ball}]
By compactness of $\S$ and Kantorovich-Rubinstein duality,
\begin{equation*}
   \mathsf{W}_1(\mu,\updelta_{x_0})=\sup_{\mathrm{Lip}(g)\leq 1}\int g(\mu -\updelta_{x_0} )\leq C\cdot\eta_2
\end{equation*}
for some numerical constant $C>0$.
Hence, for $g:\S\to \S$ defined as 
\begin{equation*}
    g(x)=\begin{cases}
        1\quad &x\in B(x,\eta_3)\\
        1-\frac{1-\eta_3}{\eta_3}\quad &x\in B(x,\eta_3)\cap B(x,2\eta_3)\\
        0\quad &x\notin B(x,2\eta_3),
    \end{cases}
\end{equation*}
we obtain $1-\mu(B(x,\eta_3))\leq C\eta_2/\eta_3.$
\end{proof}

From \eqref{eq: trivial.split}, if $\mathsf{W}_2(\mu_m,\updelta_{
y^\varepsilon_m})\leq \eta_2 $
were to hold, by applying \Cref{cl: W.to.ball} one would find
\begin{equation}\label{eq: estimate.in.m}
    \int\left\|\uppsi_\varepsilon(x)-y_m\right\|^2\mu_m(\diff x)\leq\mu_m(\S)\left(\eta_3^2+2\pi\cdot  C\cdot \frac{\eta_2}{\eta_3}\right).
\end{equation}
\eqref{eq: new.bound} and \eqref{eq: estimate.in.m} naturally raise the following problem: find a flow map that matches
\begin{align*}
    \left(\mu_m,\mu_m(\S)\updelta_{y_m}\right) &\quad\text{ for } m\in\llbracket1,M(\varepsilon)\rrbracket,\\ 
    \left(\mu(\{x_n\})\updelta_{x_n},\mu(\{x_n\})\updelta_{\uppsi^\dagger_\varepsilon(x_n)}\right) &\quad\text{ for } n\in\llbracket1,N(\eta)\rrbracket.
\end{align*}
We aim for the matching to be exact for the discrete input measures (second line) and approximate in $\mathsf{W}_2$ for the diffuse ones (first line).

\subsubsection*{Step 3: Constructing $\uppsi_\varepsilon$ through matching}
We look to use \Cref{prop: compression} to cluster the diffuse input measures to a single atom, which paired with \Cref{prop: interpolation.neural.ode} for matching all atoms approximately, would lead to the conclusion. 
Specifically, we construct the candidate $\uppsi_\varepsilon:\S\to\S$ as
\begin{equation}\label{eq: final.flow}
    \uppsi_\varepsilon\coloneqq\upphi_3\circ \upphi_2\circ \upphi_1,
\end{equation}
where
\begin{itemize}
    \item $\upphi_1:\S\to \S$ is the flow map induced by \Cref{prop: interpolation.neural.ode}\footnote{Should the assumption in \Cref{prop: interpolation.neural.ode} not hold, one can always choose slightly different \(y_m^\varepsilon\) in \eqref{eq: simple.function} so that the approximation error is not altered and the assumption does hold.}, which exactly matches $\mu(\{x_n\})\updelta_{x_n}$ to $\mu(\{x_n\})\updelta_{\uppsi^\dagger_\varepsilon(x_n)}$, for all $n\in\llbracket1,N(\eta)\rrbracket$;
    \item $\upphi_2:\S\to \S$ is the flow map induced by \Cref{prop: compression} that concentrates $\mu_m$ near some atom inside $\supp\,\upphi_{1\#}\mu_m$, for all $m\in\llbracket1,M(\varepsilon)\rrbracket$;
    \item $\upphi_3:\S\to \S$ is the flow map induced by \Cref{prop: interpolation.neural.ode} that matches the atoms from the previous step to $\mu_m(\S)\updelta_{y_m^\varepsilon}$, for all $m\in\llbracket1,M(\varepsilon)\rrbracket$.
\end{itemize}
We now make the construction of \eqref{eq: final.flow} precise, and with the help of \eqref{eq: estimate.in.m}, \Cref{prop: compression}, and \Cref{prop: interpolation.neural.ode}, we bound the right hand side in \eqref{eq: new.bound}.
\begin{enumerate}
    \item Thanks to \Cref{prop: interpolation.neural.ode} we have $\upphi_{1\#}\mu^\eta_{\text{pp}}=\uppsi^\dagger_{\varepsilon\#} \mu^\eta_{\text{pp}}.$ Exact matching ensures
\begin{equation}\label{eq: first.flow.discrete}
    \sum_{n=1}^{N(\eta)} \mu(\{x_n\})\left\|\upphi_1(x_n)-\uppsi^\dagger_\varepsilon(x_n)\right\|^2=0.
\end{equation}
\item We apply \Cref{prop: compression} to $\upphi_{1\#}\mu_m$ to deduce that, for all $m\in\llbracket1,M(\varepsilon)\rrbracket$,
\begin{equation*}
\mathsf{W}_2\left((\upphi_2\circ\upphi_1)_{\#}\frac{\mu_m}{ \mu_m(\S)}, \updelta_{x_{m}}\right)\leq \eta_2
\end{equation*}
for some $x_{m}\in \supp\, \upphi_{1\#}\mu_m$ and for small enough $\eta_2>0$ to be determined later on. Note that when we apply \Cref{prop: compression} for each $m$ in view of clustering $\upphi_{1\#}\mu_m$ to a discrete measure supported inside $\supp\,\upphi_{1\#}\mu_m$, the flow map stemming from \Cref{prop: compression} also satisfies---because of how \Cref{lem: tubular.mass.movement} is applied in the proof of \Cref{prop: compression}
\begin{equation}\label{eq: identity.flow.2}
\upphi_2\bigg|_{\S\setminus\bigcup_{m} \supp\,\upphi_{1\#}\mu_m}\equiv\mathrm{Id}.
\end{equation}
Then, by the continuity of the flow map $\upphi_1$, and \eqref{eq: removing.mass.from.mu},
we have 
$$
\supp\,\upphi_{1\#}\mu_{\text{pp}}^\eta\subset \S\setminus\bigcup_{m}\supp\,\upphi_{1\#}\mu_m,
$$
and from \eqref{eq: identity.flow.2}
\begin{equation*}
    (\upphi_2\circ\upphi_1)_\#\mu^\eta_{\text{pp}} =\uppsi^\dagger_{\varepsilon\#}\mu^\eta_{\text{pp}}.
\end{equation*}
This means that, paired with \eqref{eq: first.flow.discrete}, we also have
\begin{equation*} 
    \sum_{n=1}^{N(\eta)} \mu(\{x_n\})\left\|(\upphi_2\circ\upphi_1)(x_n)-\uppsi^\dagger_\varepsilon(x_n)\right\|^2=0.
\end{equation*}

\item  We then apply \Cref{prop: interpolation.neural.ode} to find a flow map $\upphi_3$ which matches the pairs $(x_m,y_m)_{m\in\llbracket1,M(\varepsilon)\rrbracket}$, and leads us to deduce, by virtue of continuity with respect to the data of \eqref{eq: neural.pde.sphere}, that
\begin{equation}\label{eq: close.in.W2}
\mathsf{W}_2\left((\upphi_3\circ\upphi_2\circ\upphi_1)_{\#}\frac{\mu_m}{\mu_m(\S)}  ,\updelta_{y_m^\varepsilon}\right)\leq C_{M(\varepsilon)}\cdot \eta_2
\end{equation}
holds for some $C_{M(\varepsilon)}>0$ independent of $\eta$. Moreover, after applying $\upphi_3$, thanks to \Cref{prop: interpolation.neural.ode} (or \Cref{lem: induction.neural.ode}), we have that the pure point part remains unaltered: $
    (\upphi_3\circ\upphi_2\circ\upphi_1)_\# \mu^\eta_{\text{pp}} =\uppsi^\dagger_{\varepsilon\#}\mu^\eta_{\text{pp}}.$
Hence,
\begin{equation*}
    \sum_{n=1}^{N(\eta)} \mu(\{x_n\})\left\|\uppsi_\varepsilon(x_n)-\uppsi^\dagger_\varepsilon(x_n)\right\|^2=0.
\end{equation*}

\end{enumerate}

 \subsubsection*{Step 4: Putting everything together}
 Thanks to \eqref{eq: estimate.in.m} and \eqref{eq: close.in.W2}, for any $\varepsilon_1>0$ we can choose $\eta_2$ and $\eta_3$ small enough as to ensure
\begin{equation*}
    \int\|\uppsi_\varepsilon(x)-y_m^\varepsilon\|^2\mu_m(\diff x)\leq\mu_m(\S)\varepsilon_1.
\end{equation*}
Since $\sum_{m} \mu_m(\S)\leq 1$ by construction,
\begin{equation*}
    \sum_{m} \int\|\uppsi_\varepsilon(x)-y_m^\varepsilon\|^2\mu_m(\diff x)\leq \varepsilon_1.
\end{equation*}
Combining all the estimates, and choosing $\varepsilon_1$ and $\eta$ small enough, we can deduce that
\begin{equation*}
    \left\|\uppsi_\varepsilon-\uppsi^\dagger_\varepsilon\right\|^2_{L^2(\mu)}=\int\left\|\uppsi_\varepsilon(x)-\sum_{m}y_m^\varepsilon1_{\Omega_m}\right\|^2\mu(\diff x)\leq \varepsilon_1+2\pi\cdot\eta\leq\frac{\varepsilon^2}{4},
\end{equation*}
which paired with \eqref{eq: approx.by.simple.function} leads us to the conclusion.
\end{proof}

\bibliographystyle{alpha}
\bibliography{refs}

@article{geshkovski2023mathematical,
  title={A mathematical perspective on transformers},
  author={Geshkovski, Borjan and Letrouit, Cyril and Polyanskiy, Yury and Rigollet, Philippe},
  journal={Bulletin of the American Mathematical Society},
  volume={62},
  number={3},
  pages={427--479},
  year={2025}
}

@article{geshkovski2026constructiveconditionalnormalizingflows,
      title={Constructive conditional normalizing flows}, 
      author={Borjan Geshkovski and Domènec Ruiz-Balet},
      year={2026},
      journal = {arXiv preprint arXiv:2602.08606}
}

@article{alvarez2026perceptrons,
  title={Perceptrons and localization of attention's mean-field landscape},
  author={{\'A}lvarez-L{\'o}pez, Antonio and Geshkovski, Borjan and Ruiz-Balet, Dom{\`e}nec},
  journal={arXiv preprint arXiv:2601.21366},
  year={2026}
}

@article{geshkovski2024emergence,
  title={The emergence of clusters in self-attention dynamics},
  author={Geshkovski, Borjan and Letrouit, Cyril and Polyanskiy, Yury and Rigollet, Philippe},
  journal={Advances in Neural Information Processing Systems},
  volume={36},
  year={2024}
}

@article{brenier1991polar,
  title={Polar factorization and monotone rearrangement of vector-valued functions},
  author={Brenier, Yann},
  journal={Communications on Pure and Applied Mathematics},
  volume={44},
  number={4},
  pages={375--417},
  year={1991},
  publisher={Wiley Online Library}
}

@article{ruiz2023neural,
  title={Neural ode control for classification, approximation, and transport},
  author={Ruiz-Balet, Domenec and Zuazua, Enrique},
  journal={SIAM Review},
  volume={65},
  number={3},
  pages={735--773},
  year={2023},
  publisher={SIAM}
}

@book{coron2007control,
  title={Control and nonlinearity},
  author={Coron, Jean-Michel},
  number={136},
  series = {Mathematical Surveys and Monographs},
  year={2007},
  publisher={American Mathematical Soc.}
}

@inproceedings{lu2019understanding,
  title={Understanding and Improving Transformer From a Multi-Particle Dynamic System Point of View.},
  author={Lu, Yiping and Li, Zhuohan and He, Di and Sun, Zhiqing and Dong, Bin and Qin, Tao and Wang, Liwei and Liu, Tie-yan},
  booktitle={ICLR 2020 Workshop on Integration of Deep Neural Models and Differential Equations},
  year={2020}
}

@inproceedings{sander2022sinkformers,
  title={{Sinkformers: Transformers with doubly stochastic attention}},
  author={Sander, Michael E and Ablin, Pierre and Blondel, Mathieu and Peyr{\'e}, Gabriel},
  booktitle={International Conference on Artificial Intelligence and Statistics},
  pages={3515--3530},
  year={2022},
  organization={PMLR}
}

@article{delalande2023quantitative,
  title={Quantitative stability of optimal transport maps under variations of the target measure},
  author={Delalande, Alex and Merigot, Quentin},
  journal={Duke Mathematical Journal},
  volume={172},
  number={17},
  pages={3321--3357},
  year={2023},
  publisher={Duke University Press}
}

@article{jiang2023algorithms,
  title={Algorithms for mean-field variational inference via polyhedral optimization in the Wasserstein space},
  author={Jiang, Yiheng and Chewi, Sinho and Pooladian, Aram-Alexandre},
  journal={Foundations of Computational Mathematics},
  pages={1--52},
  year={2025},
  publisher={Springer}
}

@article{agrachev2024generic,
  title={Generic controllability of equivariant systems and applications to particle systems and neural networks},
  author={Agrachev, Andrei and Letrouit, Cyril},
  journal={Annales de l'Institut Henri Poincar{\'e} C},
  year={2025}
}

@inproceedings{chen2024provably,
  title={Provably learning a multi-head attention layer},
  author={Chen, Sitan and Li, Yuanzhi},
  booktitle={Proceedings of the 57th Annual ACM Symposium on Theory of Computing},
  pages={1744--1754},
  year={2025}
}

@article{wu2024role,
  title={{On the Role of Attention Masks and LayerNorm in Transformers}},
  author={Wu, Xinyi and Ajorlou, Amir and Wang, Yifei and Jegelka, Stefanie and Jadbabaie, Ali},
  journal={arXiv preprint arXiv:2405.18781},
  year={2024}
}

@article{adu2024approximate,
  title={Approximate controllability of continuity equation of transformers},
  author={Adu, Daniel Owusu and Gharesifard, Bahman},
  journal={IEEE Control Systems Letters},
  volume={8},
  pages={964--969},
  year={2024},
  publisher={IEEE}
}

@article{paul2022microscopic,
  title={From microscopic to macroscopic scale equations: mean field, hydrodynamic and graph limits},
  author={Paul, Thierry and Tr{\'e}lat, Emmanuel},
  journal={arXiv preprint arXiv:2209.08832},
  year={2022}
}

@book{sontag2013mathematical,
  title={Mathematical control theory: deterministic finite dimensional systems},
  author={Sontag, Eduardo D},
  volume={6},
  year={2013},
  publisher={Springer Science \& Business Media}
}

@inproceedings{
yun2019transformers,
title={Are Transformers universal approximators of sequence-to-sequence functions?},
author={Chulhee Yun and Srinadh Bhojanapalli and Ankit Singh Rawat and Sashank Reddi and Sanjiv Kumar},
booktitle={International Conference on Learning Representations},
year={2020},
url={https://openreview.net/forum?id=ByxRM0Ntvr}
}

@inproceedings{
kratsios2021universal,
title={Universal Approximation Under Constraints is Possible with Transformers},
author={Anastasis Kratsios and Behnoosh Zamanlooy and Tianlin Liu and Ivan Dokmani{\'c}},
booktitle={International Conference on Learning Representations},
year={2022},
url={https://openreview.net/forum?id=JGO8CvG5S9}
}

@inproceedings{alberti2023sumformer,
  title={Sumformer: Universal approximation for efficient transformers},
  author={Alberti, Silas and Dern, Niclas and Thesing, Laura and Kutyniok, Gitta},
  booktitle={Topological, Algebraic and Geometric Learning Workshops 2023},
  pages={72--86},
  year={2023},
  organization={PMLR}
}

@article{elamvazhuthi2022neural,
  title={Neural ode control for trajectory approximation of continuity equation},
  author={Elamvazhuthi, Karthik and Gharesifard, Bahman and Bertozzi, Andrea L and Osher, Stanley},
  journal={IEEE Control Systems Letters},
  volume={6},
  pages={3152--3157},
  year={2022},
  publisher={IEEE}
}

@article{li2022deep,
  title={Deep learning via dynamical systems: An approximation perspective},
  author={Li, Qianxiao and Lin, Ting and Shen, Zuowei},
  journal={Journal of the European Mathematical Society},
  volume={25},
  number={5},
  pages={1671--1709},
  year={2022}
}

@article{cheng2023interpolation,
  title={Interpolation, approximation, and controllability of deep neural networks},
  author={Cheng, Jingpu and Li, Qianxiao and Lin, Ting and Shen, Zuowei},
  journal={SIAM Journal on Control and Optimization},
  volume={63},
  number={1},
  pages={625--649},
  year={2025},
  publisher={SIAM}
}

@article{chen2025quantitative,
  title={Quantitative Clustering in Mean-Field Transformer Models},
  author={Chen, Shi and Lin, Zhengjiang and Polyanskiy, Yury and Rigollet, Philippe},
  journal={arXiv preprint arXiv:2504.14697},
  year={2025}
}

@article{alvarez2025constructive,
  title={Constructive approximate transport maps with normalizing flows},
  author={{\'A}lvarez-L{\'o}pez, Antonio and Geshkovski, Borjan and Ruiz-Balet, Dom{\`e}nec},
  journal={Applied Mathematics \& Optimization},
  volume={92},
  number={2},
  pages={33},
  year={2025},
  publisher={Springer}
}

@inproceedings{edelman2022inductive,
  title={Inductive biases and variable creation in self-attention mechanisms},
  author={Edelman, Benjamin L and Goel, Surbhi and Kakade, Sham and Zhang, Cyril},
  booktitle={International Conference on Machine Learning},
  pages={5793--5831},
  year={2022},
  organization={PMLR}
}

@article{geshkovski2024dynamic,
  title={Dynamic metastability in the self-attention model},
  author={Geshkovski, Borjan and Koubbi, Hugo and Polyanskiy, Yury and Rigollet, Philippe},
  journal={arXiv preprint arXiv:2410.06833},
  year={2024}
}

@article{tabuada2022universal,
  title={Universal approximation power of deep residual neural networks through the lens of control},
  author={Tabuada, Paulo and Gharesifard, Bahman},
  journal={IEEE Transactions on Automatic Control},
  volume={68},
  number={5},
  pages={2715--2728},
  year={2022},
  publisher={IEEE}
}

@article{mccann2001polar,
  title={{Polar factorization of maps on Riemannian manifolds}},
  author={McCann, Robert J},
  journal={Geometric \& Functional Analysis GAFA},
  volume={11},
  number={3},
  pages={589--608},
  year={2001},
  publisher={Springer}
}

@article{scagliotti2021deep,
  title={Deep Learning approximation of diffeomorphisms via linear-control systems},
  author={Scagliotti, Alessandro},
  journal={Mathematical Control and Related Fields},
  volume={13},
  number={3},
  pages={1226--1257},
  year={2023},
  publisher={Mathematical Control and Related Fields}
}

@article{agrachev2020control,
  title={Control in the spaces of ensembles of points},
  author={Agrachev, Andrei and Sarychev, Andrey},
  journal={SIAM Journal on Control and Optimization},
  volume={58},
  number={3},
  pages={1579--1596},
  year={2020},
  publisher={SIAM}
}

@article{agrachev2022control,
  title={Control on the manifolds of mappings with a view to the deep learning},
  author={Agrachev, Andrei and Sarychev, Andrey},
  journal={Journal of Dynamical and Control Systems},
  volume={28},
  number={4},
  pages={989--1008},
  year={2022},
  publisher={Springer}
}

@article{benamou2000computational,
  title={{A computational fluid mechanics solution to the Monge-Kantorovich mass transfer problem}},
  author={Benamou, Jean-David and Brenier, Yann},
  journal={Numerische Mathematik},
  volume={84},
  number={3},
  pages={375--393},
  year={2000},
  publisher={Springer-Verlag Berlin/Heidelberg}
}

@article{vaswani2017attention,
  title={Attention is all you need},
  author={Vaswani, Ashish and Shazeer, Noam and Parmar, Niki and Uszkoreit, Jakob and Jones, Llion and Gomez, Aidan N and Kaiser, {\L}ukasz and Polosukhin, Illia},
  journal={Advances in Neural Information Processing Systems},
  volume={30},
  year={2017}
}

@article{vuckovic2020mathematical,
  title={A mathematical theory of attention},
  author={Vuckovic, James and Baratin, Aristide and Combes, Remi Tachet des},
  journal={arXiv preprint arXiv:2007.02876},
  year={2020}
}

@article{noci2022signal,
  title={Signal propagation in transformers: Theoretical perspectives and the role of rank collapse},
  author={Noci, Lorenzo and Anagnostidis, Sotiris and Biggio, Luca and Orvieto, Antonio and Singh, Sidak Pal and Lucchi, Aurelien},
  journal={Advances in Neural Information Processing Systems},
  volume={35},
  pages={27198--27211},
  year={2022}
}

@inproceedings{dong2021attention,
  title={Attention is not all you need: Pure attention loses rank doubly exponentially with depth},
  author={Dong, Yihe and Cordonnier, Jean-Baptiste and Loukas, Andreas},
  booktitle={International Conference on Machine Learning},
  pages={2793--2803},
  year={2021},
  organization={PMLR}
}

@article{dovonon2024setting,
  title={{Setting the Record Straight on Transformer Oversmoothing}},
  author={Dovonon, Gb{\`e}tondji JS and Bronstein, Michael M and Kusner, Matt J},
  journal={arXiv preprint arXiv:2401.04301},
  year={2024}
}

@inproceedings{chen2022principle,
  title={The principle of diversity: Training stronger vision transformers calls for reducing all levels of redundancy},
  author={Chen, Tianlong and Zhang, Zhenyu and Cheng, Yu and Awadallah, Ahmed and Wang, Zhangyang},
  booktitle={Proceedings of the IEEE/CVF Conference on Computer Vision and Pattern Recognition},
  pages={12020--12030},
  year={2022}
}

@inproceedings{ru2023token,
  title={Token contrast for weakly-supervised semantic segmentation},
  author={Ru, Lixiang and Zheng, Heliang and Zhan, Yibing and Du, Bo},
  booktitle={Proceedings of the IEEE/CVF Conference on Computer Vision and Pattern Recognition},
  pages={3093--3102},
  year={2023}
}

@inproceedings{
guo2023contranorm,
title={ContraNorm: A Contrastive Learning Perspective on Oversmoothing and Beyond},
author={Xiaojun Guo and Yifei Wang and Tianqi Du and Yisen Wang},
booktitle={The Eleventh International Conference on Learning Representations },
year={2023},
url={https://openreview.net/forum?id=SM7XkJouWHm}
}

@article{scholkemper2024residual,
  title={Residual Connections and Normalization Can Provably Prevent Oversmoothing in GNNs},
  author={Scholkemper, Michael and Wu, Xinyi and Jadbabaie, Ali and Schaub, Michael},
  journal={arXiv preprint arXiv:2406.02997},
  year={2024}
}

@article{castin2025unified,
  title={A unified perspective on the dynamics of deep transformers},
  author={Castin, Val{\'e}rie and Ablin, Pierre and Carrillo, Jos{\'e} Antonio and Peyr{\'e}, Gabriel},
  journal={arXiv preprint arXiv:2501.18322},
  year={2025}
}

@article{burger2025analysis,
  title={Analysis of mean-field models arising from self-attention dynamics in transformer architectures with layer normalization},
  author={Burger, Martin and Kabri, Samira and Korolev, Yury and Roith, Tim and Weigand, Lukas},
  journal={Philosophical Transactions A},
  volume={383},
  number={2298},
  pages={20240233},
  year={2025},
  publisher={The Royal Society}
}

@article{polyanskiy2025synchronization,
  title={Synchronization of mean-field models on the circle},
  author={Polyanskiy, Yury and Rigollet, Philippe and Yao, Andrew},
  journal={arXiv preprint arXiv:2507.22857},
  year={2025}
}

@inproceedings{abellaconsensus,
  title={Consensus Is All You Get: The Role of Attention in Transformers},
  author={Abella, {\'A}lvaro Rodr{\'\i}guez and Silvestre, Jo{\~a}o Pedro and Tabuada, Paulo},
  year = {2025},
  booktitle={Forty-second International Conference on Machine Learning}
}

@article{criscitiello2024synchronization,
  title={Synchronization on circles and spheres with nonlinear interactions},
  author={Criscitiello, Christopher and Rebjock, Quentin and McRae, Andrew D and Boumal, Nicolas},
  journal={arXiv preprint arXiv:2405.18273},
  year={2024}
}

@article{alcalde2025clustering,
  title={Clustering in pure-attention hardmax transformers and its role in sentiment analysis},
  author={Alcalde, Albert and Fantuzzi, Giovanni and Zuazua, Enrique},
  journal={SIAM Journal on Mathematics of Data Science},
  volume={7},
  number={3},
  pages={1367--1393},
  year={2025},
  publisher={SIAM}
}

@article{bao2024self,
  title={Self-attention Networks Localize When QK-eigenspectrum Concentrates},
  author={Bao, Han and Hataya, Ryuichiro and Karakida, Ryo},
  journal={arXiv preprint arXiv:2402.02098},
  year={2024}
}

@article{feng2022rank,
  title={Rank diminishing in deep neural networks},
  author={Feng, Ruili and Zheng, Kecheng and Huang, Yukun and Zhao, Deli and Jordan, Michael and Zha, Zheng-Jun},
  journal={Advances in Neural Information Processing Systems},
  volume={35},
  pages={33054--33065},
  year={2022}
}

@inproceedings{zhao2023are,
title={Are More Layers Beneficial to Graph Transformers?},
author={Haiteng Zhao and Shuming Ma and Dongdong Zhang and Zhi-Hong Deng and Furu Wei},
booktitle={The Eleventh International Conference on Learning Representations },
year={2023},
url={https://openreview.net/forum?id=uagC-X9XMi8}
}

@article{noci2024shaped,
  title={The shaped transformer: Attention models in the infinite depth-and-width limit},
  author={Noci, Lorenzo and Li, Chuning and Li, Mufan and He, Bobby and Hofmann, Thomas and Maddison, Chris J and Roy, Dan},
  journal={Advances in Neural Information Processing Systems},
  volume={36},
  year={2024}
}

@inproceedings{zhai2023stabilizing,
  title={Stabilizing transformer training by preventing attention entropy collapse},
  author={Zhai, Shuangfei and Likhomanenko, Tatiana and Littwin, Etai and Busbridge, Dan and Ramapuram, Jason and Zhang, Yizhe and Gu, Jiatao and Susskind, Joshua M},
  booktitle={International Conference on Machine Learning},
  pages={40770--40803},
  year={2023},
  organization={PMLR}
}

@article{joudaki2023impact,
  title={On the impact of activation and normalization in obtaining isometric embeddings at initialization},
  author={Joudaki, Amir and Daneshmand, Hadi and Bach, Francis},
  journal={Advances in Neural Information Processing Systems},
  volume={36},
  pages={39855--39875},
  year={2023}
}

@article{wu2024demystifying,
  title={Demystifying oversmoothing in attention-based graph neural networks},
  author={Wu, Xinyi and Ajorlou, Amir and Wu, Zihui and Jadbabaie, Ali},
  journal={Advances in Neural Information Processing Systems},
  volume={36},
  year={2024}
}

@article{cowsik2024geometric,
  title={Geometric Dynamics of Signal Propagation Predict Trainability of Transformers},
  author={Cowsik, Aditya and Nebabu, Tamra and Qi, Xiao-Liang and Ganguli, Surya},
  journal={arXiv preprint arXiv:2403.02579},
  year={2024}
}

@article{raginsky2024some,
  title={{Some Remarks on Controllability of the Liouville Equation}},
  author={Raginsky, Maxim},
  journal={arXiv preprint arXiv:2404.14683},
  year={2024}
}

@article{geshkovski2024number,
  title={{On the number of modes of Gaussian kernel density estimators}},
  author={Geshkovski, Borjan and Rigollet, Philippe and Sun, Yihang},
  journal={arXiv preprint arXiv:2412.09080},
  year={2024}
}

@article{alcalde2025attention,
  title={{Attention's forward pass and Frank-Wolfe}},
  author={Alcalde, Albert and Geshkovski, Borjan and Ruiz-Balet, Dom{\`e}nec},
  journal={arXiv preprint arXiv:2508.09628},
  year={2025}
}

@inproceedings{brockett2008control,
  title={{On the control of Liouville equations}},
  author={Brockett, R},
  booktitle={Differential Equation and Topology: Abstracts of International Conference Dedicated to the Centennial Anniversary of Lev Semenovich Pontryagin. Lomonosov Moscow State University. Moscow},
  pages={7},
  year={2008}
}

@article{chen2016optimal,
  title={Optimal transport over a linear dynamical system},
  author={Chen, Yongxin and Georgiou, Tryphon T and Pavon, Michele},
  journal={IEEE Transactions on Automatic Control},
  volume={62},
  number={5},
  pages={2137--2152},
  year={2016},
  publisher={IEEE}
}

@article{agrachev2009controllability,
  title={Controllability on the group of diffeomorphisms},
  author={Agrachev, Andrei and Caponigro, Marco},
  journal={Annales de l'Institut Henri Poincar{\'e} C, Analyse non lin{\'e}aire},
  volume={26},
  number={6},
  pages={2503--2509},
  year={2009}
}

@article{agrachev2009optimal,
  title={Optimal transportation under nonholonomic constraints},
  author={Agrachev, Andrei and Lee, Paul},
  journal={Transactions of the American Mathematical Society},
  volume={361},
  number={11},
  pages={6019--6047},
  year={2009}
}

@article{khesin2009,
author = {Boris Khesin and Paul Lee},
title = {{A nonholonomic Moser theorem and optimal transport}},
volume = {7},
journal = {Journal of Symplectic Geometry},
number = {4},
publisher = {International Press of Boston},
pages = {381 -- 414},
year = {2009},
}

@inproceedings{
furuya2024transformers,
title={Transformers are Universal In-context Learners},
author={Takashi Furuya and Maarten V. de Hoop and Gabriel Peyr{\'e}},
booktitle={The Thirteenth International Conference on Learning Representations},
year={2025},
url={https://openreview.net/forum?id=6S4WQD1LZR}
}

@book{shub2013global,
  title={Global stability of dynamical systems},
  author={Shub, Michael},
  year={2013},
  publisher={Springer Science \& Business Media}
}

@article{duprez2019approximate,
  title={Approximate and exact controllability of the continuity equation with a localized vector field},
  author={Duprez, Michel and Morancey, Morgan and Rossi, Francesco},
  journal={SIAM Journal on Control and Optimization},
  volume={57},
  number={2},
  pages={1284--1311},
  year={2019},
  publisher={SIAM}
}

@article{garg2022can,
  title={What can transformers learn in-context? a case study of simple function classes},
  author={Garg, Shivam and Tsipras, Dimitris and Liang, Percy S and Valiant, Gregory},
  journal={Advances in Neural Information Processing Systems},
  volume={35},
  pages={30583--30598},
  year={2022}
}

@inproceedings{chiang2023tighter,
  title={Tighter bounds on the expressivity of transformer encoders},
  author={Chiang, David and Cholak, Peter and Pillay, Anand},
  booktitle={International Conference on Machine Learning},
  pages={5544--5562},
  year={2023},
  organization={PMLR}
}

@article{jiang2023approximation,
  title={Approximation theory of transformer networks for sequence modeling},
  author={Jiang, Haotian and Li, Qianxiao},
  journal={arXiv preprint arXiv:2305.18475},
  year={2023}
}

@article{bolley2005weighted,
  title={{Weighted Csisz{\'a}r-Kullback-Pinsker inequalities and applications to transportation inequalities}},
  author={Bolley, Fran{\c{c}}ois and Villani, C{\'e}dric},
  journal={Annales de la Facult{\'e} des sciences de Toulouse: Math{\'e}matiques},
  volume={14},
  number={3},
  pages={331--352},
  year={2005}
}

@inproceedings{
bruno2024emergence,
title={Emergence of meta-stable clustering in mean-field transformer models},
author={Giuseppe Bruno and Federico Pasqualotto and Andrea Agazzi},
booktitle={The Thirteenth International Conference on Learning Representations},
year={2025},
url={https://openreview.net/forum?id=eBS3dQQ8GV}
}

@article{wang2024understanding,
  title={Understanding the Expressive Power and Mechanisms of Transformer for Sequence Modeling},
  author={Wang, Mingze and Weinan, E},
  journal={arXiv preprint arXiv:2402.00522},
  year={2024}
}

@article{jiang2023brief,
  title={A brief survey on the approximation theory for sequence modelling},
  author={Jiang, Haotian and Li, Qianxiao and Li, Zhong and Wang, Shida},
  journal={arXiv preprint arXiv:2302.13752},
  year={2023}
}

@article{petrov2024prompting,
  title={Prompting a pretrained transformer can be a universal approximator},
  author={Petrov, Aleksandar and Torr, Philip HS and Bibi, Adel},
  journal={arXiv preprint arXiv:2402.14753},
  year={2024}
}

@article{sander2024towards,
  title={Towards Understanding the Universality of Transformers for Next-Token Prediction},
  author={Sander, Michael E and Peyr{\'e}, Gabriel},
  journal={arXiv preprint arXiv:2410.03011},
  year={2024}
}

\begin{minipage}[t]{.5\textwidth}
{\footnotesize{\bf Borjan Geshkovski}\par
  Inria \&
  Laboratoire Jacques-Louis Lions\par
  Sorbonne Université\par
  4 Place Jussieu\par
  75005 Paris, France\par
 \par
  e-mail: \href{mailto:borjan.geshkovski@inria.fr}{\textcolor{blue}{\scriptsize borjan.geshkovski@inria.fr}}
  }
\end{minipage}
\begin{minipage}[t]{.5\textwidth}
  {\footnotesize{\bf Philippe Rigollet}\par
  Department of Mathematics\par
  Massachusetts Institute of Technology\par
  77 Massachusetts Ave\par
  Cambridge 02139 MA, United States\par
 \par
  e-mail: \href{mailto:blank}{\textcolor{blue}{\scriptsize rigollet@math.mit.edu}}
  }
\end{minipage}%

\begin{center}
\begin{minipage}[t]{.5\textwidth}
  {\footnotesize{\bf Domènec Ruiz-Balet}\par
  Departament de Matemàtiques\par
  Facultat de Matemàtiques i Informàtica\par
  Universitat de Barcelona\par
  585 Gran Via de les Corts Catalanes\par
  08007 Barcelona, Spain\par
 \par
  e-mail: \href{mailto:blank}{\textcolor{blue}{\scriptsize domenec.ruizibalet@ub.edu}}
  }
\end{minipage}%
\end{center}

\end{document}